\documentclass[11pt,a4paper,reqno]{amsart}
\usepackage[latin1]{inputenc}
\usepackage{amsmath}
\usepackage{amsfonts}
\usepackage{mathrsfs}
\usepackage{bm}
\usepackage{amssymb}
\usepackage[all]{xy}
\usepackage{tikz}
\usetikzlibrary{calc}
\usepackage{amsthm}
\usepackage{anysize}
\usepackage{enumitem}
\usepackage{amscd}
\usepackage{hyperref}
\usepackage[square, comma, numbers, sort&compress]{natbib}
\usepackage{indentfirst}
\marginsize{3.0cm}{3.0cm}{3.0cm}{3.0cm}
\setenumerate{label={\normalfont(\arabic*)}}
\newcommand{\form}[1]{{\langle #1 \rangle }}

\newcommand{\pfister}[1]{{\langle \! \langle #1 \rangle \! \rangle}}

\newcommand{\qpfister}[2]{{\langle \!  \langle #1, #2 ]]}}
\newcommand{\mydim}[1]{{\mathrm{dim}\!\; #1}}

\newcommand{\Izhdim}[1]{{\mathrm{dim}_{\mathrm{Izh}}\!\; #1}}
\newcommand{\anispart}[1]{#1_{\mathrm{an}}}
\newcommand{\windex}[1]{{\mathfrak{i}_W(#1)}}
\newcommand{\witti}[2]{{\mathfrak{i}_{#1}(#2)}}
\newcommand{\wittj}[2]{{\mathfrak{j}_{#1}(#2)}}

\newcommand{\stb}[0]{\stackrel{\mathrm{stb}}{\sim}}

\newcommand{\ratchow}[1]{\overline{\mathrm{Ch}}(#1)}
\newcommand{\chowbar}[1]{\mathrm{Ch}(\overline{#1})}

\newcommand{\pyrdot}[2]{
    \begin{tikzpicture}[scale=.4] 
        \pgfmathsetmacro{\n}{#1}                            
        \edef\m{0}                                          
        \foreach \i in {1, ..., \n} {                       
            \foreach \k in {1, ..., \i} {                   
                \draw ($(\k,-\i)+(-.5*\i,0)$) circle (1.5mm); 
                \pgfmathparse{int(\m+1)}                    
                \xdef\m{\pgfmathresult}
            }
        }
        \node[anchor=north] at (current bounding box.south) {\small #2};
    \end{tikzpicture}
}

\newtheorem{theorem}{Theorem}[section]

\newtheorem{lemma}[theorem]{Lemma}

\newtheorem{proposition}[theorem]{Proposition}
\newtheorem{corollary}[theorem]{Corollary}
\newtheorem{conjecture}[theorem]{Conjecture}
\newtheorem{conjectures}[theorem]{Conjectures}

\theoremstyle{definition}
\newtheorem{definition}[theorem]{Definition}
\newtheorem{example}[theorem]{Example}
\newtheorem{examples}[theorem]{Examples}
\newtheorem{question}[theorem]{Question}

\theoremstyle{remark}
\newtheorem{remark}[theorem]{Remark}
\newtheorem{remarks}[theorem]{Remarks}

\numberwithin{equation}{section}
\makeatletter
\@namedef{subjclassname@2020}{\textup{2020} Mathematics Subject Classification}
\makeatother

\begin{document}

\title[Rationality of cycles on products of generically smooth quadrics ]{Rationality of cycles modulo 2 on products of generically smooth quadrics in characteristic 2}
\author{Stephen Scully}
\address{Department of Mathematics and Statistics, University of Victoria}
\email{scully@uvic.ca}
\author{Guangzhao Zhu}
\address{Department of Mathematical and Statistical Sciences, University of Alberta}
\email{guangzha@ualberta.ca}


\maketitle

\begin{abstract} A 2022 result of Karpenko establishes a conjecture of Hoffmann-Totaro on the possible values of the first higher isotropy index of an arbitrary anisotropic quadratic form of given dimension over an arbitrary field. For nondegenerate forms, this essentially goes back to a 2003 article of the same author on quadratic forms over fields of characteristic not 2. To handle the more involved case of degenerate forms in characteristic $2$, Karpenko showed that certain aspects of the algebraic-geometric approach to nondegenerate quadratic forms developed by Karpenko, Merkurjev, Rost, Vishik and others can be adapted to a study of rational cycles modulo $2$ on powers of a given generically smooth quadric. In this paper, we extend this to a broader study of rational cycles modulo $2$ on arbitrary products of generically smooth quadrics in characteristic $2$.  A basic objective is to have tools available to study correspondences between general quadrics, in particular, between smooth and non-smooth quadrics.  Applications of the theory to the study of degenerate quadratic forms in characteristic $2$ are provided, and a number of open problems on forms of this type are also formulated and discussed. \end{abstract}
 
\section{Introduction} \label{SECintroduction} \label{SECintroduction}

Let $F$ be a field with algebraic closure $\overline{F}$, $\varphi$ an anisotropic quadratic form of dimension $d+2$ over $F$, and $X$ the $d$-dimensional projective $F$-quadric with equation $\varphi = 0$. The latter is smooth precisely when $\varphi$ is nondegenerate in the sense of \cite[\S 7.A]{EKM}. In this case, the motive of $X$ in the category of Chow motives over $F$ with $\mathbb{F}_2$-coefficients decomposes in an essentially unique way as a finite direct sum of indecomposable objects. When $F = \overline{F}$, this of course depends only on $d$. More specifically, $M(X)$ decomposes here as a direct sum of prescribed Tate motives indexed by integers in the interval $[0,d]$. In general, the complete decomposition of $M(X)$ yields a partition of the same set of Tate motives via scalar extension to $\overline{F}$. We call this the \emph{motivic decomposition type of $\varphi$}, and denote it $\mathrm{MDT}(\varphi)$.\footnote{The term \emph{motivic decomposition type of $X$} (resp. the notation $\mathrm{MDT}(X)$) is more appropriate, but we shall not consider here motives of other varieties that one may naturally attach to quadratic forms.} This invariant has played a key role in some of the major advances on nondegenerate quadratic forms achieved since the late 90s, notably in work of Karpenko, Merkurjev and Vishik. As expounded in \cite{EKM}, its study forms part of a well-developed algebraic-geometric approach to nondegenerate quadratic forms based on the investigation of algebraic cycles on products of smooth quadrics and quadratic Grassmannians. 

An important early achievement of this algebraic-geometric approach was Karpenko's theorem on the possible values of the first higher isotropy index 
$$\witti{1}{\varphi} := \mathrm{min}\lbrace \witti{0}{\varphi_L}\;|\; L \text{ an extension of } F \text{ with $\varphi_L$ isotropic}\rbrace$$
for nondegenerate $\varphi$ (\cite{Karpenko1}). Like many articles on the topic appearing at that time, \cite{Karpenko1} limited its considerations to the case where the characteristic of the base field is not $2$, owing to an essential use of the cohomological-type Steenrod operations for mod-2 Chow groups of smooth varieties constructed by Brosnan in \cite{Brosnan}. With the recent construction of the analogous operations over fields of characteristic $2$ by Primozic (\cite{Primozic}), many of the characteristic restrictions in the existing literature on the algebraic-geometric approach to nondegenerate quadratic forms can now be relaxed. 

As far as the general scope of these ideas is concerned, however, interesting questions remain regarding fields of characteristic $2$. Indeed, while anisotropic quadratic forms over fields of characteristic not 2 are necessarily nondegenerate, the study of nondegenerate forms in characteristic $2$ is only part of a broader theory of quadratic forms within which standard algebraic-geometric tools are less directly applicable. For instance, one encounters here the extreme class of \emph{quasilinear} quadratic forms, whose associated quadrics have no smooth points at all. The first real attempts to explore this broader picture came in the 00s with a series of works by Hoffmann, Laghribi, Totaro and others, where some well-known results of inherently algebraic-geometric nature on nondegenerate forms were shown to admit extensions to the degenerate case. This suggested that certain aspects of the algebraic-geometric perspective may be adaptable to the study of degenerate forms, despite the lack of a well-developed intersection theory for non-smooth varieties.

The first steps in this direction were recently taken in \cite{Karpenko2}, where Karpenko extended his result on the $\mathfrak{i}_1$ invariant to the case of degenerate but nonquasilinear quadratic forms in characteristic $2$.\footnote{The statement is also known to be valid for quasilinear forms (\cite{Scully1}), but this requires different methods.} The basic point is the following: Suppose that $\mathrm{char}(F) = 2$, and let $U$ be the smooth locus of $X$. If $\varphi$ is not quasilinear, then $U$ is nonempty, and the scalar extension homomorphism $\mathrm{CH}(X^r)/2 \rightarrow \mathrm{CH}(X^r_{\overline{F}})/2$ factors canonically through $\mathrm{CH}(U^r)/2$ for any positive integer $r$. By passing through the smooth variety $U^r$, one can then apply some of the standard tools of intersection theory to the study of the image of the scalar extension map. As is well known in the nondegenerate case, getting a handle on this image is already sufficient for interesting applications to discrete invariants of quadratic forms.

While \cite{Karpenko2} did not go beyond its intended application to the study of $\mathfrak{i}_1$, it was clear that the arguments found there could be extended to develop a theory of ``rational cycles modulo $2$'' for arbitrary products of generically smooth anisotropic quadrics in characteristic $2$ along the lines of that developed for smooth quadrics in \cite{EKM} and \cite{Vishik1}. The purpose of the present article is to make some of this explicit in order to have tools available for handling certain algebraic-geometric problems for degenerate quadratic forms in characteristic $2$. Of particular interest here are conjectures of the first author on the possible splitting behaviour of forms under scalar extension to function fields of quadrics (\cite[Conj. 1.1]{Scully2} and its refinement discussed in \cite[\S 1]{Scully3}).  These are expected to be valid for all forms, nondegenerate or otherwise, and the tools discussed here are directly applicable to the study of the nonquasilinear case.\footnote{The quasilinear case has been fully resolved in \cite{Scully3} using different methods.} This will be considered in a separate text.  \vspace{.5 \baselineskip}

\noindent {\bf Overview.} After a preliminary section on quadratic forms and their associated quadrics, sections \ref{SECnumericaltriviality}, \ref{SECrationalcycles}, \ref{SECcorrespondences} and \ref{SECisotropicreduction} of the present text extend the discussion of \cite{Karpenko2} to a wider study of arbitrary products of generically smooth anisotropic quadrics in characteristic $2$. In short, if $Y$ is such a product, then minor modifications of the arguments in \cite{Karpenko2} show that
$$ \overline{\mathrm{Ch}}(Y): = \mathrm{Im}(\mathrm{CH}(Y)/2 \rightarrow \mathrm{CH}(Y_{\overline{F}})/2) $$
inherits from the smooth locus of $Y$ the structure of an $\mathbb{F}_2$-algebra and an action of cohomological-type Steenrod operations. Given three such products $Y_1,Y_2,Y_3$, one may then define a composition law $\overline{\mathrm{Ch}}(Y_1 \times Y_2) \otimes_{\mathbb{F}_2}  \overline{\mathrm{Ch}}(Y_2 \times Y_3) \rightarrow \overline{\mathrm{Ch}}(Y_1 \times Y_3)$ that serves as a substitute for the standard composition of Chow correspondences for smooth varieties. 

After presenting the basic tools, we introduce in \S \ref{SECMDT} the obvious extension of the motivic decomposition type invariant to degenerate but nonquasilinear anisotropic forms in characteristic $2$. In the interest of consistency, we continue to use the notation $\mathrm{MDT}$, though we do not introduce any formal categorical framework for the study of this invariant. With the intersection-theoretic tools in place, some of the basic results on the $\mathrm{MDT}$ for nondegenerate forms can be immediately extended to the degenerate case.

In \S \ref{SECVishik}, we establish one of the main objectives of the work, namely a variant of a theorem of Vishik concerning stable birational equivalences of quadratic Grassmannians. Vishik's result, which in the literature is limited to the characteristic-not-$2$ setting, establishes motivic decompositions of smooth quadrics arising from such equivalences. Our Theorem \ref{THMVishik} is a discrete variant of this valid for all nonquasilinear anisotropic forms in characteristic $2$ (for nondegenerate forms, it yields the stronger motivic statement of Vishik's result by the discussion of \cite[Ch.  XVII]{EKM}). A key aspect of this result is that it permits, in certain situations, to relate the $\mathrm{MDT}$ invariants of nondegenerate and degenerate forms.  For example, we can derive degenerate variants of the well-known results of Rost on the motivic structure of nondegenerate Pfister neighbours and excellent forms (see \S \ref{SUBSECPNSandexcellent}). More significantly, the results of \S 8 can be used to recast certain problems on degenerate forms as problems lying within the nondegenerate framework.  For instance, in \S \ref{SECPfisterneighbourconjecture}, we reduce the essential part of a conjecture of Hoffmann-Laghribi on the classification of degenerate Pfister neighbours to a well-known conjecture of Vishik on binary direct summands in the motives of smooth quadrics (see Proposition \ref{PROPreductionofPNproblem}). As part of this discussion, we also provide some more direct evidence for the Hoffmann-Laghribi conjecture (Theorem \ref{THMPfisterneighbourproblem}), and in fact reduce it to another important open problem in the degenerate setting, namely the classification of forms of \emph{nondefective height $1$} (Corollary \ref{CORequivalenceofconjectures}).  In \S \ref{SECexcellentconnections}, we also consider the problem of determining the extent  to which Karpenko's theorem on the possible values of the $\mathfrak{i}_1$ invariant remains valid if the dimension of the quasilinear part of the form is taken into account. We raise here a general question (Question \ref{Qi1}) for which we conjecture a positive answer when the dimension of the quasilinear part is sufficiently small (Conjecture \ref{CONJi1smalls}). Using the results of \S 8 and known results on nondegenerate quadratic forms, we provide some evidence for this conjecture (see Theorem \ref{THMi1restrictionssingular} in particular). 

Finally, while the proofs of some important results on nondegenerate forms (e.g., Karpenko's theorem on $\mathfrak{i}_1$) can be directly adapted to the degenerate setting, there are others lying at a deeper level for which things are more involved. The primary issue here is the fact that the action of the cohomological-type Steenrod operations on the groups $\overline{\mathrm{Ch}}(Y)$ is not intrinsically defined over $F$ outside of the nondegenerate setting. There, the descent of the action has been used to establish a major result on the $\mathrm{MDT}$ invariant, namely Vishik's theorem on the existence of so-called \emph{excellent connections} (\cite{Vishik2}).  In \S \ref{SECexcellentconnections}, we conjecture that the obvious variant of this for degenerate but nonquasilinear anisotropic forms in characteristic $2$ is valid (Conjecture \ref{CONJexcellentconnections}), and discuss some implications of this claim.  While proving the statement requires further development of the algebraic-geometric machinery, we at least verify its validity in dimension $\leq 9$ (Lemma \ref{LEMexcellentconnectionsindimensionatmost9}).  \vspace{.5 \baselineskip}

\noindent {\bf Terminology and Notation.} Throughout this text, a \emph{scheme} is a separated scheme of finite type over a field, and a \emph{variety} is an integral scheme. If $X$ is a scheme, then we shall write $d_X$ for the dimension of $X$.  The letters $\mathrm{CH}$ shall be used to denote integral Chow groups, and $\mathrm{Ch}$ to denote Chow groups modulo $2$.  We shall essentially only work with the latter groups. Basic facts in Chow theory (see, e.g., \cite[Pt. 2]{EKM}) shall be used freely.

\section{Preliminaries on Quadratic Forms and Quadrics in Characteristic 2} \label{SECpreliminaries}

For the remainder of the paper, we fix a field $F$ of characteristic $2$. In this section, we present some preliminary material on quadratic forms and their associated quadrics. For all further background information, the reader is referred to \cite{EKM}.

\subsection{Symmetric Bilinear Forms} By a \emph{symmetric bilinear form over $F$}, we mean a pair $(V, \mathfrak{b})$ consisting of a finite-dimensional $F$-vector space $V$ and a nondegenerate symmetric $F$-bilinear form $\mathfrak{b} \colon V \times V \rightarrow F$. In practice, we suppress $V$ from our notation and simply talk about the form $\mathfrak{b}$. If $a_1,\hdots,a_n \in F^\times$, then we write $\form{a_1,\hdots,a_n}_b$ for form $F^n \times F^n \rightarrow F$ that sends $((x_1,\hdots,x_n),(y_1,\hdots,y_n))$ to $\sum_{i=1}^n a_ix_iy_i$. We also write $\pfister{a_1,\hdots,a_n}_b$ for the $n$-fold bilinear Pfister form $\form{1,a_1}_b \otimes \cdots \otimes \form{1,a_n}_b$. Any such form $\mathfrak{b}$ is \emph{round}, in the sense that $\mathfrak{b} \simeq \mathfrak{b}(v,v)\mathfrak{b}$ for all vectors $v$ with $\mathfrak{b}(v,v) \neq 0$ (\cite[Cor. 6.2]{EKM}).

\subsection{Quadratic Forms} \label{SUBSECquadraticforms}By a \emph{quadratic form over $F$}, we mean a pair $(V, \varphi)$ consisting of a finite-dimensional $F$-vector space $V$ and a quadratic form $\varphi \colon V \rightarrow F$. In practice, we suppress $V$ from our notation and simply talk about the form $\varphi$. In particular, we shall write $\mydim{\varphi}$ for the dimension of $V$ and refer to it as the \emph{dimension of $\varphi$}. If the need arises to consider the underlying space, we shall denote it $V_{\varphi}$. Given $a, b \in F$, we write $[a,b]$ for the form $F^2 \rightarrow F$ that sends $(x,y)$ to $ax^2 + xy + by^2$. Given $a_1,\hdots,a_n \in F$, we write $\form{a_1,\hdots,a_n}$ for the form $F^n \rightarrow F$ that sends $(x_1,\hdots,x_n)$ to $\sum_{i=1}^n a_ix_i^2$. Quadratic forms of this type are said to be \emph{quasilinear}. If $\varphi$ is a quadratic form over $F$, then there is, up to isometry, a unique quadratic form $\mathrm{ql}(\varphi)$ over $F$ with the following properties:
\begin{itemize} \item[(i)] $\mathrm{ql}(\varphi)$ is quasilinear;
\item[(ii)] $\varphi \simeq \left(\perp_{i=1}^r [a_i,b_i]\right) \perp \mathrm{ql}(\varphi)$ for some nonnegative integer $r$ and elements $a_i,b_i \in F$. \end{itemize}
In (ii), the form $\perp_{i=1}^r [a_i,b_i]$ is not uniquely determined by $\varphi$ in general, but the integer $r$ is. If $\mydim{\mathrm{ql}(\varphi)} = s$, then $\mydim{\varphi} = 2r + s$, and we say that \emph{$\varphi$ has type $(r,s)$}. If $W$ is a finite-dimensional $F$-vector space, then we write $\mathbb{H}(W)$ for the form $W \oplus W^{\vee} \rightarrow F$ that sends $(w,f)$ to $f(w)$. Quadratic forms isometric to one of this type are said to be \emph{hyperbolic}. If $\mydim{W} = d$, then $\mathbb{H}(W) \simeq d \cdot \mathbb{H}$, where $\mathbb{H}: = \mathbb{H}(F) \simeq [0,0]$ is the \emph{hyperbolic plane}. Witt decomposition says that if $\varphi$ is a quadratic form over $F$, then there exists an anisotropic quadratic form $\anispart{\varphi}$ over $F$ and nonnegative integers $\windex{\varphi}$, $\mathfrak{i}_{\mathrm{d}}(\varphi)$ such that $\varphi \simeq \anispart{\varphi} \perp \windex{\varphi} \cdot \mathbb{H} \perp \mathfrak{i}_{\mathrm{d}}(\varphi) \cdot \form{0}$. The form $\anispart{\varphi}$ is unique up to isometry and is called the \emph{anisotropic part of $\varphi$}. The integers $\windex{\varphi}$ and $\witti{\mathrm{d}}{\varphi}$ are also unique, and called the \emph{Witt index of $\varphi$} and \emph{defect index of $\varphi$}, respectively. The \emph{isotropy index of $\varphi$}, denoted $\witti{0}{\varphi}$, is defined as the sum of $\windex{\varphi}$ and $\witti{\mathrm{d}}{\varphi}$. Alternatively, $\witti{0}{\varphi}$ is the maximal dimension of a totally isotropic subspace of $V_\varphi$. Note that $\witti{\mathrm{d}}{\varphi}$ coincides with the isotropy index of $\mathrm{ql}(\varphi)$. If this integer is nonzero (i.e., if $\mathrm{ql}(\varphi)$ is isotropic), then we say that $\varphi$ is \emph{defective}. If $\witti{\mathrm{d}}{\varphi} = 0$ and $\mathrm{ql}(\varphi)$ has dimension at most $1$, then we shall say that $\varphi$ is \emph{nondegenerate}. Thus, a nondegenerate quadratic form over $F$ is one isometric to $\perp_{i=1}^r [a_i,b_i]$ or $\left(\perp_{i=1}^r [a_i,b_i]\right) \perp \form{c}$ for some nonnegative integer $r$ and elements $a_i,b_i \in F$, $c \in F^\times$. In the even-dimensional case, the term \emph{nonsingular} is also sometimes used to indicate nondegeneracy. If $\psi$ is another quadratic form over $F$, then we say that $\varphi$ and $\psi$ are \emph{Witt equivalent}, and write $\varphi \sim \psi$, if $\anispart{\varphi} \simeq \anispart{\psi}$. Witt equivalence has the properties of an equivalence relation, and nondegenerate quadratic forms of even dimension are Witt equivalent if and only if they represent the same element of the quadratic Witt group $W_q(F)$.  For nondefective forms whose quasilinear parts have equal dimension, we have:

\begin{lemma} \label{LEMWittequivalenceofnondefective} Let $\varphi$ and $\psi$ be nondefective quadratic forms over $F$ of types $(r,s)$ and $(r',s)$ respectively.  Then the following are equivalent:
\begin{enumerate} \item $\varphi \sim \psi$;
\item $\varphi \perp \psi \sim \mathrm{ql}(\varphi)$;
\item $\witti{0}{\varphi \perp \psi} = r+r' + s$;
\item $\windex{\varphi \perp \psi} = r + r'$ and $\mathrm{ql}(\varphi) \simeq \mathrm{ql}(\psi)$. \end{enumerate}
\begin{proof} Note that if $\sigma$ is any quadratic form over $F$, then the following hold:
\begin{itemize} \item[(i)] $\sigma \perp \mathrm{ql}(\sigma) \sim \sigma$;
\item[(ii)] $\sigma \perp \sigma \sim \mathrm{ql}(\sigma)$. \end{itemize}
Indeed, for the first point, it suffices to show that $\mathrm{ql}(\varphi) \perp \mathrm{ql}(\varphi) \sim \mathrm{ql}(\varphi)$. But if $V$ denotes the underlying vector space of $\mathrm{ql}(\varphi)$, then the automorphism of $V \oplus V$ mapping each pair $(v,w)$ to $(v+w,w)$ gives an isometry from $\mathrm{ql}(\varphi) \perp \mathrm{ql}(\varphi)$ onto $\mathrm{ql}(\varphi) \perp \mathrm{dim}(V)\langle 0 \rangle$.  As for the second point, the first point reduces us to the case where $\sigma$ is nondegenerate of even dimension. But in this case, $\sigma \perp \sigma$ is a nondegenerate form that evidently admits a totally isotropic subspace of dimension equal to half the dimension of its underlying subspace, and is hence hyperbolic. We now prove the desired equivalences:

$(1) \Rightarrow (2)$: We may assume that $\varphi$ and $\psi$ are anisotropic, and hence that $\varphi \simeq \psi$.  The claim then follows from statement (ii) above.

$(2) \Rightarrow (3)$: Since $\mathrm{ql}(\varphi)$ is anisotropic, this is clear.

$(3) \Rightarrow (4)$: Note that $\varphi \perp \psi$ has type $(r+r',2s)$ and quasilinear part $\mathrm{ql}(\varphi) \perp \mathrm{ql}(\psi)$. If $\mathfrak{i}_{\mathrm{d}}(\varphi \perp \psi)$ were greater than $s$, then $\mathrm{ql}(\varphi)$ would be isotropic (being a codimension-$s$ subform of $\mathrm{ql}(\varphi) \perp \mathrm{ql}(\psi)$). Since this is not the case, (3) then implies that $\windex{\varphi \perp \psi} = r+r'$ and $\mathfrak{i}_{\mathrm{d}}(\varphi \perp \psi) = s$. Let $V_1$ and $V_2$ be the underlying vector spaces of $\mathrm{ql}(\varphi)$ and $\mathrm{ql}(\psi)$, respectively. Since $\mathfrak{i}_{\mathrm{d}}(\varphi \perp \psi) = s$ every subform of $\mathrm{ql}(\varphi) \perp \mathrm{ql}(\psi)$ of codimension $<s$ is isotropic. In particular, if $v \in V_1$, then there exists a vector $f(v) \in V_2$ such that $\varphi(v) = \psi(f(v))$. If $f(v)$ were not unique, then $\mathrm{ql}(\psi)$ would be isotropic, contrary to our hypothesis. Thus, $f(v)$ is unique, and the quasilinearity of $\mathrm{ql}(\varphi)$ and $\mathrm{ql}(\psi)$ then implies that $v \mapsto f(v)$ defines an isometry from the former onto the latter. Thus, (4) holds. 

$(4) \Rightarrow (1)$: Since $\varphi \perp \psi$ has type $(r+r',2s)$ and Witt index $r+r'$, we have $\varphi \perp \psi \sim \mathrm{ql}(\varphi \perp \psi) \sim \mathrm{ql}(\varphi) \perp \mathrm{ql}(\psi)$. Since $\mathrm{ql}(\varphi) \simeq \mathrm{ql}(\psi)$, statements (i) and (ii) from the beginning of the proof then gives that 
$$ \varphi \sim \varphi \perp \mathrm{ql}(\varphi) \perp \mathrm{ql}(\psi) \sim \varphi \perp \varphi \perp \psi \sim \mathrm{ql}(\varphi) \perp \psi \sim \mathrm{ql}(\psi) \perp \psi \sim \psi, $$
as desired.  \end{proof} \end{lemma}

If $\mathfrak{b}$ and $\varphi$ are symmetric bilinear and quadratic forms over $F$, respectively, then we may consider the tensor product quadratic form $\mathfrak{b} \otimes \varphi$ as defined on \cite[P. 51]{EKM}. If $\varphi$ is nondegenerate, then the same is true of $\mathfrak{b} \otimes \varphi$, and the construction then equips the quadratic Witt group $W_q(F)$ with the structure of a $W(F)$-module, where $W(F)$ is the Witt ring of symmetric bilinear forms over $F$. Finally, if $K$ is a field extension of $F$, then we shall write $\varphi_K$ for the quadratic form over $K$ induced by $\varphi$. Note that we have $\mathrm{ql}(\varphi_K) \simeq \mathrm{ql}(\varphi)_K$. We shall repeatedly use the following basic facts:
\begin{itemize} \item Anisotropic quadratic forms remain anisotropic under purely transcendental extensions (\cite[Lem. 7.15]{EKM}) and finite extensions of odd degree (Springer's theorem, \cite[Cor. 18.5]{EKM});
\item Anisotropic quasilinear quadratic forms remain anisotropic under separable extensions (\cite[Prop. 5.3]{Hoffmann2}). \end{itemize}

\subsection{Subforms and Domination} \label{SUBSECdomination} Let $\psi$ be a quadratic form of type $(r,s)$ over $F$. If $\psi$ is isometric to an orthogonal summand of a quadratic form $\varphi$ over $F$, then we shall say that \emph{$\psi$ is a subform of $\varphi$}, and write $\psi \subset \varphi$. More generally, if $\psi$ is isometric to the restriction of $\varphi$ onto a subspace of $V_{\varphi}$, then we shall say that \emph{$\psi$ is dominated by $\varphi$}, and write $\psi \prec \varphi$. If $\varphi$ is nondegenerate of even dimension, then it is shown in \cite[\S 3]{HoffmannLaghribi1} that $\psi \prec \varphi$ if and only if there exist quadratic forms $\psi_r, \tau, \sigma$ over $F$, and elements $c_1,d_1,\hdots,c_s,d_s \in F$ such that
\begin{itemize} \item $\psi_r$ and $\tau$ are nondegenerate and of even dimension;
\item $\psi \simeq \psi_r \perp \form{c_1,\hdots,c_s}$ (in particular, $\mathrm{ql}(\psi) \simeq \form{c_1,\hdots,c_s}$);
\item $\varphi \simeq \psi_r \perp \tau \perp [c_1,d_1] \perp \cdots \perp [c_s,d_s]$. \end{itemize}
While the individual forms appearing here are not uniquely determined, one readily observes that the form $\psi_{\varphi}^c: = \tau \perp \form{c_1,\hdots,c_s} \simeq \tau \perp \mathrm{ql}(\psi)$ satisfies $\psi_{\varphi}^c \sim \psi \perp \varphi$, and hence only depends on the pair $(\psi, \varphi)$. We call it the \emph{complementary form of $\psi$ in $\varphi$}. Note that $\mydim{\varphi} - \mydim{\psi} = \mydim{\psi_{\varphi}^c} \geq s$, and that $(\psi_{\varphi}^c)^c \simeq \psi$. The first inequality implies in particular that $\mydim{\varphi} \geq 2(r+s)$. If equality holds here, then we say that \emph{$\varphi$ is a nonsingular completion of $\psi$}. In this case, we have $\psi_{\varphi}^c \simeq \mathrm{ql}(\psi)$. 

\subsection{Projective Quadrics} \label{SUBSECquadrics} If $\varphi$ is a quadratic form over $F$, then we shall write $X_{\varphi}$ for the quadric hypersurface in $\mathbb{P}(V_{\varphi})$ defined by the vanishing of $\varphi$ (when $\mydim{\varphi} \leq 1$, this means that $X_{\varphi} = \emptyset$). By (the proof of) \cite[ Prop. 22.1]{EKM}, the singular (i.e., nonsmooth) locus of $X_{\varphi}$ is the closed subscheme defined by the vanishing of the quasilinear part $\mathrm{ql}(\varphi)$. In particular, $X_{\varphi}$ is smooth precisely when $\varphi$ is nondegenerate, and generically smooth (i.e., has nonempty smooth locus) precisely when $\varphi$ is nonquasilinear. We shall say that $X_{\varphi}$ is isotropic (resp. anisotropic, nondefective, quasilinear) if $\varphi$ is.

\subsection{Function Fields of Quadrics, the Knebusch Splitting Tower and the Izhboldin Dimension} \label{SUBSECfunctionfieldsandsplittingpatterns} Let $\varphi$ be a quadratic form over $F$. If the quadric $X_{\varphi}$ is nonempty and integral, then we write $F(\varphi)$ for its function field. Otherwise, we set $F(\varphi): = F$. It is easy to see that $X_{\varphi}$ is nonempty and integral if and only if $\witti{0}{\varphi} \leq \mydim{\varphi} - 2$, so the second case is exceptional. As $X_{\varphi}$ has an $F(\varphi)$-point, the form $\varphi_{F(\varphi)}$ is isotropic. 

\begin{lemma} \label{LEMfunctionfields} Let $\varphi$ be a nonzero quadratic form of dimension $\geq 2$ over $F$ such that $\witti{0}{\varphi} \leq \mydim{\varphi} - 2$ $($so that $X_{\varphi}$ is nonempty and integral$)$. Then:
\begin{enumerate} \item $F(\varphi)/F$ is separable if and only if $\varphi$ is nonquasilinear;
\item $F(\varphi)/F$ is purely transcendental if and only if $\windex{\varphi}>0$. 
\end{enumerate}
\begin{proof} (1) The function field of an $F$-variety $X$ is a separable extension of $F$ precisely when $X$ is generically smooth, and the quadric $X_{\varphi}$ is generically smooth precisely when $\varphi$ is nonquasilinear (\S \ref{SUBSECquadrics}). 

(2) Suppose first that $F(\varphi)/F$ is purely transcendental. Since $\varphi_{F(\varphi)}$ is isotropic, $\varphi$ must then be isotropic. If $\windex{\varphi}=0$, we then have that $\varphi \simeq \anispart{\varphi} \perp i \cdot \form{0}$ for some $i \geq 1$, and so $X_{\varphi}$ is a cone over $X_{\anispart{\varphi}}$. In particular, $F(\anispart{\varphi})$ is $F$-isomorphic to a subfield of $F(\varphi)$. Since $F(\varphi)/F$ is purely transcendental, $\anispart{\varphi}$ must then remain anisotropic over $F(\anispart{\varphi})$. But implies that $\mydim{\anispart{\varphi}} \leq 1$, contradicting the integrality of $X_{\varphi}$. We must therefore have that $\windex{\varphi} > 0$ in this case. Conversely, if $\windex{\varphi}>0$, then $X_{\varphi}$ is isomorphic to a projective hypersurface of equation $xy = p(z_1,\hdots,z_n)$ for some homogeneous quadratic polynomial $p(z_1,\hdots,z_n) \in F[z_1,\hdots,z_n]$. The locus where $x \neq 0$ then constitutes a rational open subvariety of $X_{\varphi}$, and so $F(\varphi)/F$ is purely transcendental.  \end{proof} \end{lemma}

Following a classical construction of Knebusch, we associate to $\varphi$ a finite sequence of pairs $(F_i,\varphi_i)$ consisting of an extension $F_i$ of $F$ and an anisotropic quadratic form $\varphi_i$ over $F_i$ as follows: Given a pair $(F_{i-1},\varphi_{i-1}$) with $\mydim{\varphi_{i-1}} \geq 2$, we define $F_i$ to be the field $F_{i-1}(\varphi_{i-1})$ and $\varphi_i$ to be the form $\anispart{(\varphi_{F_i})}$. The process is initiated by setting $(F_0,\varphi_0) : = (F, \anispart{\varphi})$, and terminates at the first integer $h$ for which $\mydim{\varphi_h} \leq 1$. The tower $F = F_0 \subseteq F_1 \subseteq \cdots \subseteq F_h$ is called the \emph{Knebusch splitting tower of $\varphi$}, the form $\varphi_i$ is called the \emph{$i$th anisotropic kernel form of $\varphi$}, and the integer $h$ is called the \emph{Knebusch height of $\varphi$}. For all $0 \leq t \leq h$, we set $\wittj{t}{\varphi} : = \witti{0}{\varphi_{F_t}}$. When $t \geq 1$, we also set $\witti{t}{\varphi} := \wittj{t}{\varphi} - \wittj{t-1}{\varphi}$. The set $\lbrace \wittj{t}{\varphi}\;|\; 0 \leq t \leq h \rbrace$ is called the \emph{Knebusch splitting pattern of $\varphi$}. In the sequel, we will be interested in the case where $\varphi$ is nondefective, and here will shall only deal with a certain piece of the Knebusch splitting tower:

\begin{lemma} \label{LEMnondefectiveheight} Suppose, in the above situation, that $\varphi$ is nondefective. If $\varphi$ has type $(r,s)$, then there exists a unique integer $0 \leq h_{\mathrm{nd}} \leq h$ for which $\wittj{h_{\mathrm{nd}}}{\varphi} = r$. Moreover, we have
\begin{alignat*}{3} \lbrace \wittj{t}{\varphi}\;|\; 0 \leq t \leq h_{\mathrm{nd}} \rbrace \quad & = \quad && \lbrace \windex{\varphi_K}\;|\; K/F \text{ a separable extension} \rbrace \\ \quad & = \quad && \lbrace \witti{0}{\varphi_K}\;|\; K/F \text{ an extension with } \witti{\mathrm{d}}{\varphi_K} = 0 \rbrace.  \end{alignat*}
\begin{proof} Let $h_{\mathrm{nd}}$ be the smallest integer for which $\varphi_{h_{\mathrm{nd}}}$ is quasilinear. The Witt index of $\varphi$ over $F_{h_{\mathrm{nd}}}$ is then $r$, and so $\wittj{h_{\mathrm{nd}}}{\varphi} \geq r$. To prove that equality holds, we have to show that $\varphi$ remains nondefective over $F_{h_{\mathrm{nd}}}$. But the extension $F_{h_{\mathrm{nd}}}$ is separable by repeated application of Lemma \ref{LEMfunctionfields} (1), and so the claim follows from the fact that anisotropic quasilinear quadratic forms remain anisotropic under separable extensions. This proves the first statement. Moreover, the preceding discussion also gives the inclusions
\begin{alignat*}{3} \lbrace \wittj{t}{\varphi}\;|\; 0 \leq t \leq h_{\mathrm{nd}} \rbrace \quad & \subseteq \quad && \lbrace \windex{\varphi_K}\;|\; K/F \text{ a separable extension} \rbrace \\ \quad & \subseteq \quad && \lbrace \witti{0}{\varphi_K}\;|\; K/F \text{ an extension with } \witti{\mathrm{d}}{\varphi_K} = 0 \rbrace,  \end{alignat*}
so to prove the second statement, we just have to show that if $K/F$ is a field extension with $\witti{\mathrm{d}}{\varphi} = 0$, then $\witti{0}{\varphi_K} = \wittj{t}{\varphi}$ for some $0 \leq t \leq h_{\mathrm{nd}}$. Since $\witti{\mathrm{d}}{\varphi} = 0$, we have $\witti{0}{\varphi_K} \leq r$, and so there exists a smallest integer $0 \leq t \leq h_{\mathrm{nd}}$ such that $\windex{\varphi_K} \leq \wittj{t}{\varphi}$. We claim that equality holds here. Since anisotropic quadratic forms remain anisotropic under purely transcendental extensions, it suffices to show that the compositum $K \cdot F_t$ is a purely transcendental extension of $K$. But since $\windex{\varphi_K} = \witti{0}{\varphi_K} > \wittj{t-1}{\varphi}$, $K \subseteq K \cdot F_1 \subseteq \cdots \subseteq K \cdot F_t$ is a tower of function fields of integral quadrics defined by quadratic forms of positive Witt index, and the claim then holds by Lemma \ref{LEMfunctionfields} (2).  \end{proof}
\end{lemma}

In the situation of the lemma, the integer $h_{\mathrm{nd}}$ will be called the \emph{nondefective height of $\varphi$}. In turn, the truncated tower $F = F_0 \subseteq F_1 \subseteq \cdots \subseteq F_{h_{\mathrm{nd}}}$ will be called the \emph{nondefective splitting tower of $\varphi$}, and the set $\lbrace \wittj{t}{\varphi}\;|\; 0 \leq t \leq h_{\mathrm{nd}} \rbrace$ the \emph{nondefective splitting pattern of $\varphi$}. If we need to emphasize the dependence on $\varphi$, we will write $h_{\mathrm{nd}}(\varphi)$ instead of $h_{\mathrm{nd}}$. 

\begin{remarks} In the above situation, the \emph{full splitting pattern} of $\varphi$ is defined as the set consisting of the isotropy indices attained by $\varphi$ over \emph{all} possible extensions of $F$. When $\varphi$ is nondegenerate, Lemma \ref{LEMnondefectiveheight} shows that this coincides with the Knebusch splitting pattern of $\varphi$ (which is the same thing as nondefective splitting pattern of $\varphi$ in this case). In general, however, the Knebusch splitting pattern will only constitute a subset of the full splitting pattern. For example, one readily checks that if $X$, $Y$ and $Z$ are indeterminates, then the quasilinear quadratic form $\form{1,X,Y,XY,Z}$ over $F(X,Y,Z)$ has Knebusch splitting pattern $\lbrace 0,1,3,4 \rbrace$, but full splitting pattern $\lbrace 0,1,2,3,4 \rbrace$. The point is that if we only consider extensions that do not alter the defect index of $\varphi$ (e.g., separable extensions), then Knebusch's construction gives a ``universal'' partial splitting tower for $\varphi$. If we wish to allow $\mathrm{ql}(\varphi)$ to split, however, then this universality is lost. 
\end{remarks}

Although the Knebusch splitting tower does not in general recover all possible isotropy indices attained by $\varphi$ over extensions of $F$, one can nevertheless show that $\wittj{1}{\varphi}$ is the second smallest integer in the full splitting pattern of $\varphi$ when defined. In other words, if $\varphi$ is anisotropic of dimension $\geq 2$, then $\witti{1}{\varphi}$ is the smallest positive element of the set $\lbrace \witti{0}{\varphi_K}\;|\; K/F \text{ an extension} \rbrace$ (see \cite[\S 0]{Karpenko2}). The following theorem gives a nontrivial restriction on the possible values of $\witti{1}{\varphi}$ in this case. The analogous statement over fields of characteristic not 2 was first proved by Karpenko in \cite{Karpenko1}. In our setting, the case where $\varphi$ is quasilinear was treated in \cite{Scully1}, and the nonquasilinear cases were more recently treated by Primozic in \cite{Primozic} (nondegenerate case) and Karpenko in \cite{Karpenko2} (general case). 

\begin{theorem} \label{THMi1} Let $\varphi$ be an anisotropic quadratic form of dimension $\geq 2$ over $F$, and let $u$ be the largest integer for which $\Izhdim{\varphi}$ is divisible by $2^u$. Then $\witti{1}{\varphi} \leq 2^u$. 
\end{theorem}

When $\varphi$ is nondegenerate or quasilinear, all values of the integer $\witti{1}{\varphi}$ permitted by Theorem \ref{THMi1} can be realized. In the case where $\varphi$ is neither nondegenerate nor quasilinear, however, we will see in \S \ref{SECexcellentconnections} below that the type of $\varphi$ imposes some further restrictions. When $\varphi$ is anisotropic of dimension $\geq 2$, we shall define the \emph{Izhboldin dimension of $\varphi$} to be the integer $\Izhdim{\varphi}: = \mydim{\varphi} - \witti{1}{\varphi}$. Theorem \ref{THMi1} then says that $\Izhdim{\varphi}$ is divisible by the smallest $2$-power bounding $\witti{1}{\varphi}$ from above. In particular, we have the following result due to Hoffmann and Laghribi (\cite[Lem. 4.1]{HoffmannLaghribi2}):

\begin{corollary} \label{CORmaxsplitting} Let $\varphi$ be an anisotropic quadratic form of dimension $2^n + m$ for some non-negative integer $n$ and integer $1 \leq m \leq 2^n$. Then $\Izhdim{\varphi} \geq 2^n$.
\end{corollary}

If equality holds in the conclusion of the corollary (equivalently, if $\witti{1}{\varphi} = m$), then we shall say that \emph{$\varphi$ has maximal splitting}. Now, another key result on the Izhboldin dimension is the following theorem which is due to Karpenko and Merkurjev in the nondegenerate case (\cite[Thm. 76.5]{EKM}) and Totaro in the degenerate case (\cite[Thm. 5.2]{Totaro}):

\begin{theorem} \label{THMKM} Let $\varphi$ and $\psi$ be anisotropic quadratic forms of dimension $\geq 2$ over $F$. If $\varphi_{F(\psi)}$ is isotropic, then $\Izhdim{\psi} < \mydim{\varphi}$.
\end{theorem}

When combined with Corollary \ref{CORmaxsplitting}, this gives the following ``separation theorem'', again due to Hoffmann and Laghribi (\cite[Thm. 1.1]{HoffmannLaghribi2}):

\begin{corollary} \label{CORseparation} Let $\varphi$ and $\psi$ be anisotropic quadratic forms of dimension $\geq 2$ over $F$. If $\mydim{\varphi} \leq 2^n < \mydim{\psi}$ for some non-negative integer $n$, then $\varphi_{F(\psi)}$ is anisotropic. 
\end{corollary}

\subsection{Stable Birational Equivalence} \label{SUBSECstb} Let $\varphi$ and $\psi$ be anisotropic quadratic forms of dimension $\geq 2$ over $F$. The following are then known to be equivalent:

\begin{enumerate} \item $X_{\varphi}$ and $X_{\psi}$ are stably birational as varieties over $F$;
\item For every field extension $K/F$, $\varphi_K$ is isotropic if and only if $\psi_K$ is isotropic;
\item $\varphi_{F(\psi)}$ and $\psi_{F(\varphi)}$ are isotropic. \end{enumerate}
Indeed, since anisotropic quadratic forms remain anisotropic under purely transcendental extensions, the implications $(1)\Rightarrow (2) \Rightarrow (3)$ are clear. For $(3) \Rightarrow (1)$, the case where neither $\varphi$ nor $\psi$ are quasilinear is a straightforward consequence of Lemma \ref{LEMfunctionfields} (2), and the case where at least one of $\varphi$ and $\psi$ is quasilinear is a result due to Totaro (\cite{Totaro}). In fact, if $\varphi$ is quasilinear and (3) holds, then $\psi$ must also be quasilinear by Lemma \ref{LEMfunctionfields} (2) and the fact that anisotropic quasilinear quadratic forms remain anisotropic under separable extensions. Under the assumption that both $\varphi$ and $\psi$ are quasilinear, however, the implication $(3) \Rightarrow (1)$ follows from \cite[Thms. 5.2 and 6.5]{Totaro}, and so the equivalence holds in all cases. If (1), (2) and (3) hold, we shall say that $\varphi$ and $\psi$ are \emph{stably birationally equivalent}, and write $\varphi \stb \psi$. The following important result, which extends Theorem \ref{THMKM}, is again due to Karpenko and Merkurjev in the nondegenerate case (\cite[Thm. 76.5]{EKM}) and Totaro in the degenerate case (\cite[Thm. 5.2]{Totaro}):

\begin{theorem} \label{THMstb} Let $\varphi$ and $\psi$ be anisotropic quadratic forms of dimension $\geq 2$ over $F$. Then $\varphi \stb \psi$ if and only if $\varphi_{F(\psi)}$ is isotropic and $\Izhdim{\varphi} = \Izhdim{\psi}$. 
\end{theorem}

We then have the following basic fact:

\begin{lemma} \label{LEMstbsubforms} Let $\varphi$ and $\psi$ be anisotropic quadratic forms of dimension $\geq 2$ over $F$ with $\psi \prec \varphi$. If $\mydim{\psi} > \Izhdim{\varphi}$, then $\varphi \stb \psi$. In particular, $\Izhdim{\varphi} = \Izhdim{\psi}$. 
\begin{proof} We can assume that $\psi$ is the restriction of $\varphi$ to a subspace $U$ of $V_{\varphi}$. It is clear that $\varphi_{F(\psi)}$ is clearly isotropic. Now, by definition, the $F(\varphi)$-vector space $V_{\varphi} \otimes_F F(\varphi)$ admits a totally isotropic subspace of dimension $\witti{1}{\varphi}$. Since $\mydim{\psi} > \Izhdim{\varphi}$, this subspace must intersect $U \otimes_F F(\varphi)$ nontrivially, and so $\psi_{F(\varphi)}$ is also isotropic. This proves the first statement, and the second then follows by Theorem \ref{THMstb}.
\end{proof}
\end{lemma}

\subsection{Pfister Neighbours and (Strongly) Excellent Forms} \label{SUBSECexcellentforms} Let $a_1,\hdots,a_n\in F^\times$. If $a_{n+1} \in F$, then we write $\qpfister{a_1,\hdots}{a_{n+1}}$ for the $(n+1)$-fold Pfister form $\pfister{a_1,\hdots,a_{n}}_b \otimes [1,a_{n+1}]$.  Any such form $\pi$ is \emph{round}, in the sense that $a\pi \simeq \pi$ for all $a \in F^\times$ represented by $\pi$ (\cite[Cor. 9.9]{EKM}). By a \emph{general} \emph{$(n+1)$-fold} \emph{Pfister form}, we shall mean any form similar to (i.e., isometric to a scalar multiple of) an $(n+1)$-fold Pfister form. Any such form is nondegenerate of dimension $2^{n+1}$. An isotropic general Pfister form is hyperbolic, i.e., has Witt index equal to half its dimension. In particular, if $\pi$ is an anisotropic general $(n+1)$-fold Pfister form, then $\witti{1}{\pi} = 2^{n}$, i.e., $\pi$ has maximal splitting. Now, let $\varphi$ be a nonzero quadratic form over $F$, and let $n$ be the unique integer for which $2^{n} < \mydim{\varphi} \leq 2^{n+1}$. If $\varphi$ is is dominated by a general $(n+1)$-fold Pfister form $\pi$, then we shall say that $\varphi$ is a \emph{Pfister neighbour}. In this case, the form $\pi$ is uniquely determined up to isometry, and we refer to it as the \emph{ambient general Pfister form of $\varphi$}. The complementary form $\varphi_{\pi}^c$ (see \S \ref{SUBSECdomination}), which has dimension $2^{n+1} - \mydim{\varphi}$, will simply be denoted $\varphi^c$. We thus have that $\varphi \perp \pi \sim \varphi^c$. In the case where $\varphi$ is anisotropic, it follows from the Cassels-Pfister subform theorem (\cite[Thm. 22.5]{EKM}) and Lemma \ref{LEMstbsubforms} that $\varphi$ is a Pfister neighbour if and only if it is stably birationally equivalent to some anisotropic Pfister form (which is then similar to the ambient general Pfister form of $\varphi$). We note the following:

\begin{lemma} \label{LEMPneighbours} Let $\varphi$ be an anisotropic quadratic form of type $(r,s)$ over $F$, and $n$ the unique integer for which $2^n < \mydim{\varphi} \leq 2^{n+1}$. If $\varphi$ is a Pfister neighbour, then:
\begin{enumerate} \item $r+s \leq 2^n$;
\item $\varphi_1 \simeq (\varphi^c)_{F(\varphi)}$.
\item $\varphi$ has maximal splitting, i.e., $\Izhdim{\varphi} = 2^n$;
\end{enumerate}
\begin{proof} Let $\pi$ be the ambient general Pfister form of $\varphi$. 

(1) Per the discussion in \S \ref{SUBSECdomination}, any nondegenerate form of even dimension that dominates $\varphi$ has dimension at least $2(r+s)$. Since $\mydim{\pi} = 2^{n+1}$, (1) then follows.

(2) Since $\varphi \stb \pi$, $\pi_{F(\varphi)}$ is isotropic, and hence hyperbolic. Since $\varphi \perp \pi \sim \varphi^c$, it follows that $\varphi_1 \sim (\varphi^c)_{F(\varphi)}$. But $\mydim{\varphi^c} = 2^{n+1} - \mydim{\varphi} < 2^n$, and so $(\varphi^c)_{F(\varphi)}$ is anisotropic by the separation theorem (Corollary \ref{CORseparation}). As $\varphi_1$ is also anisotropic, we then have that $\varphi_1 \simeq (\varphi^c)_{F(\varphi)}$.

(3) Implicit in (2). 
\end{proof}
\end{lemma}

A quadratic form $\varphi$ over $F$ is said to be \emph{excellent} if for every field extension $K/F$ with $\witti{0}{\varphi_K}> \witti{0}{\varphi}$, there exists a quadratic form $\tau$ over $F$ such that $\anispart{(\varphi_K)} \simeq \tau_K$. By \cite[Lem. 5.1]{Hoffmann2}, all quasilinear quadratic forms are excellent. For nonquasilinear forms, the situation is more complicated. A source of examples is provided by the following:

\begin{lemma} \label{LEMstrongexcellence} Let $\varphi$ be a nonquasilinear quadratic form over $F$ with Knebusch splitting tower $F = F_0 \subseteq F_1 \subseteq \cdots $. Suppose that there exist a positive integer $h$ and Pfister neighbours $\psi_0,\psi_1,\hdots,\psi_{h-1}$ over $F$ such that:
\begin{enumerate} \item $\psi_0 = \anispart{\varphi}$;
\item $\psi_i \simeq \psi_{i-1}^c$ for all $1 \leq i < h$;
\item $\psi_{h-1}^c \simeq \anispart{(\mathrm{ql}(\varphi))}$. \end{enumerate}
Then $\varphi$ is excellent of nondefective height $h$, and $\varphi_i \simeq (\psi_i)_{F_i}$ for all $0 \leq i < h$. 
\begin{proof} By part (2) of Lemma \ref{LEMPneighbours}, we only have to prove the excellence of $\varphi$. To this end, let $K$ be an extension of $F$ with $\witti{0}{\varphi_K}> \witti{0}{\varphi}$. There then exists a largest integer $i <h$ with $(\psi_i)_K$ is isotropic. For each $0 \leq j < h$, let $\pi_j$ be the ambient general Pfister form of $\psi_j$. Then $(\pi_j)_K$ is hyperbolic for all $j \leq i$, and so 
$$\varphi_K \sim (\psi_0)_K \sim (\psi_0^c)_K \simeq (\psi_1)_K \sim \cdots \sim (\psi_{i-1}^c)_K \simeq (\psi_i)_K \sim (\psi_i^c)_K = (\psi_{i+1})_K$$
by (1) and (2). If $i \neq h-1$, we then have that $\anispart{(\varphi_K)} \simeq (\psi_{i+1})_K$. If $i = h-1$, then (3) gives that $\anispart{(\varphi_K)} \simeq \anispart{(\mathrm{ql}(\varphi)_K)}$. By the remarks preceding the statement, however, we then have that $\anispart{(\varphi_K)} \simeq \tau_K$ for some quasilinear form $\tau$ over $K$, and so we're done.
\end{proof}
\end{lemma}

If $\varphi$ is as in Lemma \ref{LEMstrongexcellence}, then we shall say that $\varphi$ is \emph{strongly excellent}. By a result essentially due to Knebusch, all nondegenerate excellent forms are strongly excellent (see \cite[Thm. 28.3]{EKM}). In general, however, excellence is weaker. For instance, anisotropic forms of type $(1,s)$ are readily seen to be excellent, but such a form cannot be strongly excellent unless $s = 2^n - 1$ for some integer $n$ by Lemma \ref{LEMPneighbours} (1). Nevertheless, it is expected that an excellent quadratic form of type $(r,s)$ is strongly excellent when $s$ is ``sufficiently small''. This can be made precise via a conjecture of Hoffmann and Laghribi on degenerate Pfister neighbours that will be discussed in \S \ref{SECPfisterneighbourconjecture} below. Here, we note that anisotropic strongly excellent forms can be described more explicitly as follows:

\begin{proposition} Let $\varphi$ be an anisotropic nonquasilinear quadratic form over $F$, and let $h$ be a positive integer. Then the following are equivalent:
\begin{enumerate} \item $\varphi$ is strongly excellent of nondefective height $h$;
\item There exist anisotropic Pfister forms $\pi_0,\pi_1,\hdots,\pi_{h-1}$ over $F$, an anisotropic quasilinear quadratic form $\tau$ over $F$, and a scalar $a \in F^\times$ such that:
\begin{itemize} \item[$\mathrm{(i)}$] $\varphi \sim a(\pi_0 \perp \pi_1 \perp \cdots \perp \pi_{h-1} \perp \tau)$;
\item[$\mathrm{(ii)}$] For all $1 \leq i < h$, $\pi_i$ is a proper subform of $\pi_{i-1}$;
\item[$\mathrm{(iii)}$] $\tau \prec \pi_{h-1}$ and $\mydim{\tau} < \frac{\mydim{\pi_{h-1}}}{2}$;
\item[$\mathrm{(iv)}$] If $\tau = 0$, then $\mydim{\pi_{h-2}} > 2 \mydim{\pi_{h-1}}$. 
\end{itemize} \end{enumerate}
\begin{proof} $(1) \Rightarrow (2)$: Suppose that $\varphi$ is strongly excellent of nondefective height $h$, and let $\psi_0,\hdots,\psi_{h-1}$ be as in the statement of Lemma \ref{LEMstrongexcellence}. For each $0 \leq i \leq h-1$, let $\rho_{i}$ be the ambient general Pfister form of $\psi_i$, and write $\rho_i = a_i\pi_i$ for some Pfister form $\pi_i$ over $F$ and scalar $a_i \in F^\times$. Since $\psi_{h-1}^c = \mathrm{ql}(\varphi)$, the anisotropic quasilinear quadratic form $\tau: = a_{h-1}\mathrm{ql}(\varphi)$ satisfies condition (iii) in (2). Moreover, since $\psi_{i-1} \perp \rho_{i-1} \sim \psi_{i-1}^c \simeq \psi_{i}$ for all $1 \leq i < h$, we have that
\begin{alignat*}{3} \varphi \perp (a_0\pi_0 \perp \cdots \perp a_{h-1}\pi_{h-1}) \quad & \simeq \quad && (\psi_0 \perp \rho_0) \perp (\rho_1 \perp \cdots \perp \rho_{h-1}) \\
\quad & \sim \quad && (\psi_1 \perp \rho_1) \perp (\rho_2 \perp \cdots \perp \rho_{h-1}) \\
\quad & \hspace{.2cm} \vdots \quad && \hspace{2cm} \vdots \\
\quad & \sim \quad && \psi_{h-1} \perp \rho_{h-1} \\
\quad & \sim \quad && \psi_{h-1}^c \\
\quad & \simeq \quad && \mathrm{ql}(\varphi) \\
\quad &  \simeq \quad && a_{h-1}\tau. \end{alignat*}
Since $a_0\pi_0 \perp \cdots \perp a_{h-1}\pi_{h-1}$ is nondegenerate, we then have that $\varphi \sim a_0\pi_0 \perp \cdots \perp a_{h-1}\pi_{h-1} \perp a_{h-1}\tau$ (Lemma \ref{LEMWittequivalenceofnondefective}). To show that conditions (i) and (ii) in (2) are satisfied, it now remains to show that for each $1 \leq i < h$, we have $\pi_i \subset \pi_{i-1}$ and $a_{i-1}\pi_{i-1} \simeq a_i \pi_i$ (then (i) will hold with $a = a_{h-1}$). For any such $i$, however, we have $\pi_i \stb \psi_i = \psi_{i-1}^c \prec \rho_{i-1} = a_{i-1}\pi_{i-1}$, and so $\pi_{i-1}$ becomes isotropic over $F(\pi_i)$. Since $\pi_{i-1}$ is Pfister, it follows that $(\pi_{i-1})_{F(\pi_i)}$ is hyperbolic, and so $\pi_{i} \subset \pi_{i-1}$ by the Cassels-Pfister subform theorem (\cite[Thm. 22.5]{EKM}). Since $a_{i-1}\pi_{i-1} $ and $a_i \pi_i =\rho_i$ both dominate $\psi_i$, it then follows that $\pi_{i-1}$ represents elements $x,y \in F^\times$ such that $a_{i-1}x = a_i y$. By the roundness of Pfister forms, we then have that $a_i\pi_{i-1} \simeq a_ixy^{-1}\pi_{i-1} \simeq a_i\pi_{i-1}$, as desired. Finally, if $\tau = 0$, then $\psi_{h-1} = a_{h-1}\pi_{h-1}$. Since $\mydim{\psi_{h-1}} = \mydim{\psi_{h-2}^c} =  \mydim{\pi_{h-2}} - \mydim{\psi_{h-2}} < \frac{\mydim{\pi_{h-2}}}{2}$, we have that $\mydim{\pi_{h-2}} > 2 \mydim{\pi_{h-1}}$ in this case, i.e. (iv) is satisfied.

$(2) \Rightarrow (1)$: Let $\pi_0,\hdots,\pi_{h-1},\tau$ and $a$ be as in (2). For each $0 \leq i <h$, set $\psi_i := \anispart{a(\pi_i \perp \pi_{i+1} \perp \cdots \perp \pi_{h-1} \perp \tau)}$. By hypothesis, $\psi_0 \simeq a(\pi_0 \perp \pi_1 \perp \cdots \perp \pi_{h-1} \perp \tau)_{\mathrm{an}} \simeq \varphi$. Since the form $a(\pi_0 \perp \pi_1 \perp \cdots \perp \pi_{h-1})$ is nondegenerate, and since $\tau$ is anisotropic, we then also have that $\mathrm{ql}(\varphi) \simeq a\tau$. To prove the claim, it therefore suffices to show the following:
\begin{itemize} \item For each $1 \leq i < h$, $\psi_{i-1}$ is a Pfister neighbour with ambient general Pfister form $a\pi_{i-1}$ and complementary form $\psi_i$;
\item $\psi_{h-1}$ is a Pfister neighbour with $\psi_{h-1}^c \simeq a\tau$. 
\end{itemize}
We start with the second claim. By hypothesis, $\tau$ is dominated by $\pi_{h-1}$. If we let $\tau^c$ be the complementary form, then $\psi_{h-1} = \anispart{a(\pi_{h-1} \perp \tau)} \sim a\tau^c$. Since $\pi_{h-1}$ is anisotropic, the same is true of $\tau^c$, and so $\psi_{h-1} \simeq a\tau^c$. But since $\mydim{\tau} < \frac{\mydim{\pi_{h-1}}}{2}$, we have $\mydim{a\tau^c} > \frac{\mydim{\pi_{h-1}}}{2}$, and so $\psi_{h-1}$ is a Pfister neighbour with complementary form $a\tau$. For the first claim, let $1 \leq i < h$. By definition, we have $\psi_{i-1} \simeq \anispart{(a\pi_{i-1} \perp \psi_{i})}$. Arguing by induction on $i$ (and using the preceding discussion for $i= h-1$), we can assume that $\psi_i$ is dominated by $a\pi_i$, and hence by $a\pi_{i+1}$. Since $\pi_i$ is anisotropic, it follows that $\psi_{i-1} \simeq \psi_i^c$. To prove the claim, it now only remains to check that $\mydim{\psi_{i-1}} > \mydim{\psi_i}$. Suppose, for the sake of contradiction, that this is not the case. Since $\psi_i = \psi_{i-1}^c$, we have $\mydim{\psi_{i-1}} = \mydim{\pi_{i-1}} - \mydim{\psi_i}$. Since $\psi_i \prec a \pi_i$, and since $\mydim{\pi_i} \leq \frac{\mydim{\pi_{i-1}}}{2}$, the only possibility here is that $\psi_i \simeq a\pi_i$ and $\mydim{\pi_i} = \frac{\mydim{\pi_{i-1}}}{2}$. But the first point implies that $i = h-1$ and $\tau = 0$, and the second then contradicts assumption (iv). Thus, the first claim also holds, and so the proposition is proved.
\end{proof} \end{proposition}

\begin{remarks} \label{REMSexcellent} \begin{enumerate}[leftmargin=*] \item Again, in the nondegenerate case, this is essentially due to Knebusch.
\item In the statement of the proposition, the forms $\pi_0,\hdots,\pi_{h-1},\tau$ in (2) are uniquely determined up to isometry.  Indeed, as the proof shows, $\tau \simeq \mathrm{ql}(\varphi)$, $a\pi_0$ is the ambient general Pfister form of $\varphi$, $a\pi_1$ is the ambient general Pfister form of $\varphi^c$, and so on. \end{enumerate}
\end{remarks}

\section{Preliminaries on Cycles} \label{SECnumericaltriviality}

Recall that we use the notation $\mathrm{Ch}$ for Chow groups modulo $2$. We record here some preliminary facts on the latter.

\subsection{Degree Homomorphisms and Numerical Triviality} \label{SUBSECnumericaltriviality} Let $Y$ be a scheme over $F$. If $Y$ is projective, then we have the degree homomorphism $\mathrm{deg} \colon \mathrm{Ch}(Y) \rightarrow \mathbb{F}_2$ defined as pushforward along the structure morphism $Y \rightarrow \mathrm{Spec}(F)$. If $Y$ is smooth, then we shall say that an element $\alpha \in \mathrm{Ch}(Y)$ is \emph{numerically trivial} if $\mathrm{deg}(\alpha \beta) = 0$ for all $\beta \in \mathrm{Ch}(Y)$. The subset of $\mathrm{Ch}(Y)$ consisting of all numerically trivial elements is an ideal in $\mathrm{Ch}(Y)$ which we shall denote $N_0(Y)$. In the sequel, we shall have occasion to consider a slightly more general situation where these remarks still apply. More precisely, suppose now that $Y$ is an open subscheme of a projective $F$-scheme $Y'$ for which the (reduced scheme of the) complement $Y' \setminus Y$ has no closed point of odd degree. In this case, the localization sequence for the closed embedding $Y' \setminus Y \rightarrow Y'$ shows that the the degree homomorphism for $Y'$ factors through $\mathrm{Ch}(Y)$ by restriction. Abusing notation, we shall also denote the induced homomorphism $\mathrm{Ch}(Y) \rightarrow \mathbb{F}_2$ by $\mathrm{deg}$. By definition, it sends the class of any closed point $y \in Y$ to the degree of the residue field extension $F(y)/F$, and sends the classes of all other closed subvarieties of $Y$ to $0$. In particular, it is independent of $Y'$, and hence intrinsic to $Y$.  Let us describe this situation by saying that  \emph{$Y$ admits a degree homomorphism}.  If $Y$ is also smooth, we can then define the ideal $N_0(Y)$ of numerically trivial elements in $\mathrm{Ch}(Y)$ exactly as before. We will need the following observations: 

\begin{lemma} \label{LEMtrivialpushforward} Let $f \colon X \rightarrow Y$ be a proper morphism of schemes over $F$. Assume that $Y$ is smooth and admits a degree homomorphism $($in the sense above$)$. 
\begin{enumerate} \item If $X$ has no closed point of odd degree, then the image of the pushforward $f_* \colon \mathrm{Ch}(X) \rightarrow \mathrm{Ch}(Y)$ lies in $N_0(Y)$. In particular, if $N_0(Y)=0$, then $f_*$ is the zero map.
\item Suppose that $X$ is also smooth and admits a degree homomorphism $($in the sense above$)$. If $f$ is a closed embedding, then $f^*(N_0(Y)) \subseteq N_0(X)$. \end{enumerate}
\begin{proof} (1) Since $\mathrm{Ch}(X)$ is generated by classes of closed subvarieties of $X$, it suffices to show that $\alpha: = f_*([X])$ lies in $N_0(Y)$. But if $\beta \in \mathrm{Ch}(Y)$, then $\alpha\beta = f_*(f^*(\beta))$ by \cite[Prop. 56.11]{EKM}, and so $\mathrm{deg}(\alpha\beta) = 0$ on account of $X$ having no closed point of odd degree.

(2) Let $\alpha \in N_0(Y)$ and $\beta \in \mathrm{Ch}(X)$. We have to show that $\mathrm{deg}(f^*(\alpha)\beta) =0$. Since $f$ is a closed embedding, we have $\mathrm{deg}(\gamma) = \mathrm{deg}(f_*(\gamma))$ for all $\gamma \in \mathrm{Ch}(X)$. By the projection formula (\cite[Prop. 56.9]{EKM}),  however, we have $f_*(f^*(\alpha)\beta) = \alpha f_*(\beta)$, and so $\mathrm{deg}(f^*(\alpha)\beta) = \mathrm{deg}(\alpha f_*(\beta)) = 0$ by the numerical triviality of $\alpha$. 
\end{proof} \end{lemma}

\begin{remark} \label{REMN=0forproductsofprojectivespaces} If $Y$ is a product of projective spaces, then $N_0(Y) = 0$ by the projective bundle formula.
\end{remark}

\subsection{Scalar Extension to an Algebraic Closure and Rationality} \label{SUBSECrationality} Let $\overline{F}$ be an algebraic closure of $F$. Given a scheme $X$ over $F$, we set $\overline{X}: = X_{\overline{F}}$, and write $\ratchow{X}$ for the image of the scalar extension homomorphism $\mathrm{Ch}(X) \rightarrow \chowbar{X}$ (equipped with its induced gradings by dimension and codimension). The elements of $\chowbar{X}$ that lie in $\ratchow{X}$ will be said to be \emph{$F$-rational}. If $\alpha$ is an element of $\mathrm{Ch}(X)$, then its image in $\ratchow{X}$ shall be denoted $\overline{\alpha}$. We need the following:

\begin{lemma} \label{LEMvanishingofrationalcycles} Let $f \colon X \rightarrow Y$ be a proper morphism of schemes over $F$. Suppose that $Y$ is smooth and projective, and that $N_0(Y) = 0$.  If $X$ has no closed point of odd degree, and the pushforward homomorphism $f_* \colon \chowbar{X} \rightarrow \chowbar{Y}$ is injective, then $\ratchow{X} = 0$. 
\begin{proof} Let $\alpha \in \mathrm{Ch}(X)$. We have to show that $\overline{\alpha} = 0$. By hypothesis, it suffices to show that $f_*(\overline{\alpha}) = 0$. But $f_*(\overline{\alpha}) = \overline{f_*(\alpha)}$, so the desired assertion holds by Lemma \ref{LEMtrivialpushforward} (1).
\end{proof} \end{lemma}

\subsection{An Observation on Products of Quasilinear Quadrics} Let $\varphi$ be an anisotropic quasilinear quadratic form of dimension $\geq 2$ over $F$. Set $V: = V_{\varphi}$, and let $i \colon X_{\varphi} \rightarrow \mathbb{P}(V)$ be the canonical embedding.

\begin{lemma} \label{LEMsplittinglemmaforquasilinearquadrics} Let $Z$ be a scheme over $F$. In the above situation, there is then a natural $\mathbb{F}_2$-vector space homomorphism $\pi \colon \chowbar{\mathbb{P}(V) \times Z} \rightarrow \chowbar{X_{\varphi} \times Z}$ satisfying $(i \times \mathrm{id})_* \circ \pi = \mathrm{id}$. In particular, $(i \times \mathrm{id})_*$ is injective.
\begin{proof} Let $W$ be the set of isotropic vectors in $V \otimes_F \overline{F}$. Since $\varphi$ is quasilinear, $W$ is an $\overline{F}$-linear subspace of $V\otimes_F \overline{F}$. Since $\overline{F}$ is algebraically closed, $W$ has codimension $1$ in $V \otimes_F \overline{F}$ and $\mathbb{P}(W)$ is the reduced scheme of $\overline{X_{\varphi}}$. In particular, if we let $j \colon \mathbb{P}(W) \rightarrow \overline{X_{\varphi}}$ be the canonical embedding, then the pushforward $j_* \colon \mathrm{Ch}(\mathbb{P}(W) \times \overline{Z}) \rightarrow \chowbar{X_{\varphi} \times Z}$ is an isomorphism. By the projective bundle formula, there is a natural $\mathbb{F}_2$-vector space homomorphism $\pi' \colon \chowbar{\mathbb{P}(V) \times Z} \rightarrow \mathrm{Ch}(\mathbb{P}(W) \times \overline{Z})$ satisfying $(i \circ j \times \mathrm{id})_* \circ \pi' = \mathrm{id}$. The composite map $\pi: = j \circ \pi'$ then has the desired property. \end{proof}
\end{lemma}

This leads to the following, which will be a basic point in the sequel:

\begin{proposition} \label{PROPvanishingofrationalcyclesforquasilinear} Let $X$ be a variety over $F$ which is a product of projective spaces and positively many anisotropic quasilinear quadrics. Then $\chowbar{X} = 0$.  
\begin{proof} By the projective bundle formula, we can assume that $X$ is a product of positively many anisotropic quasilinear quadrics. By Springer's theorem, none of these quadrics have a closed point of odd degree, and so the same is true of $X$. Let $\varphi_1,\hdots,\varphi_m$ be anisotropic quasilinear quadratic forms of dimension $\geq 2$ over $F$ such that $X = X_{\varphi_1} \times \cdots \times X_{\varphi_m}$. For each $1 \leq t \leq m$, set $V_t: = V_{\varphi_t}$, and let $i_t \colon X_{\varphi_t} \rightarrow \mathbb{P}(V_i)$ be the canonical embedding. Set $Y: = \mathbb{P}(V_1) \times \cdots \times \mathbb{P}(V_m)$ and $i: = i_1 \times \cdots \times i_m$. Then $N_0(Y) = 0$ (Remark \ref{REMN=0forproductsofprojectivespaces}), and repeated application of Lemma \ref{LEMsplittinglemmaforquasilinearquadrics} shows that $i_* \colon \chowbar{X} \rightarrow \chowbar{Y}$ is injective. The claim then follows from Lemma \ref{LEMvanishingofrationalcycles}. \end{proof}
\end{proposition}

\section{Rational Cycles on Products of Generically Smooth Nondefective Quadrics} \label{SECrationalcycles}

Let $\overline{F}$ be an algebraic closure of $F$. In this section, we consider the group $\ratchow{X}$ (as defined in \S \ref{SUBSECrationality} above) in the case where $X$ is a product of generically smooth projective quadrics over $F$, each of which is nondefective (i.e., has anisotropic singular locus). For products of several copies of a single quadric of this type, the results stated here were previously established by Karpenko in \cite{Karpenko2}, and the extension to the more general situation considered here only requires minor adaptation of the arguments found in the latter. For the remainder of this section, we fix nonquasilinear quadratic forms $\varphi_1,\hdots,\varphi_m$ of dimension $\geq 2$ over $F$, each of which is assumed to be nondefective. For each $1 \leq t \leq m$, we let $(r_t,s_t)$ be the type of $\varphi_t$, and we set $V_i: = V_{\varphi_i}$ and $X_i: = X_{\varphi_t}$. Finally, we set $X: = X_1 \times \cdots \times X_m$. Since the $\varphi_i$ are nonquasilinear, the integers $r_t$ are positive and $X$ is generically smooth. We start by formulating the statements we wish to establish.

\subsection{Statements} \label{SUBSECstatementsonrationalcycles} Let $1 \leq t \leq m$. Since $\overline{F}$ is algebraically closed, we have $\windex{(\varphi_t)_{\overline{F}}} = r_t$. Let $U_t$ be a $(2r_t)$-dimensional subspace of $V_t \otimes_F \overline{F}$ such that $\varphi|_{U_t}$ is hyperbolic, and let $W_t$ be an $r_t$-dimensional totally isotropic subspace of $U_t$. For each $0 \leq i< r_t$, let us define $h^i \in \chowbar{X_i}$ to be the class of a codimension $i$-subquadric of $X_t$, and $l_i \in \chowbar{X_i}$ to be the class of any $i$-dimensional projective linear subspace of $\mathbb{P}(W_t)$ (viewed as a closed subvariety of $\overline{X_t}$). Note that we are suppressing the dependency on $t$ in order to avoid overcomplicating our notation. The element $h^i$ clearly lies in $\ratchow{X_t}$ and depends only on $\varphi_t$ (being the pullback of the unique nontrivial element of $\mathrm{Ch}^j(\mathbb{P}(V_t))$ along the canonical embedding of $X_t$ into $\mathbb{P}(V_t)$). The element $l_i$ need not be $F$-rational. While it a priori depends on the choice of $W_t$, it again does not depend on the choice of $i$-dimensional projective linear subspace of $\mathbb{P}(W_t)$. Below, we shall see that:

\begin{lemma} \label{LEMlinearindependence} $\lbrace l_i,h^i \rbrace_{0 \leq i < r_t}$ is an $\mathbb{F}_2$-linearly independent subset of $\chowbar{X_t}$. \end{lemma}

Let $R_t$ be the $2r_t$-dimensional $\mathbb{F}_2$-linear subspace of $\chowbar{X_t}$ generated by $\lbrace l_i,h^i\rbrace_{0 \leq i < r_t}$. We can equip $R_t$ with the structure of a commutative ring by defining the multiplication as follows:
$$ h^i h^j: = \begin{cases} h^{i+j} & \text{if } i+j < r_t \\ 0 & \text{otherwise} \end{cases}; \hspace{.5cm} h^i  l_j : = \begin{cases} l_{j-i} & \text{if } i \leq j \\ 0 & \text{otherwise}; \end{cases}; $$
$$ l_i l_j: = \begin{cases} l_0 & \text{if } \mydim{\varphi_t} \equiv 2 \pmod{4} \text{ and } i= j = \frac{\mydim{\varphi_t}-2}{2} \\ 0 & \text{otherwise}. \end{cases} $$
Note that the identity element is $h^0$. For each integer $j \geq 0$, we can also define an $\mathbb{F}_2$-vector space homomorphism $S^j \colon R_t \rightarrow R_t$ by setting
$$ S^j(h^i) : = \begin{cases} \binom{i}{j}h^{i+j} & \text{if } j < r_t-i \\ 0 & \text{otherwise} \end{cases} \;\;\; \text{and} \;\;\; S^j(l_i): = \begin{cases} \binom{\mydim{\varphi} - i- 1}{j} l_{i-j} & \text{if } j \leq i \\ 0 & \text{otherwise}. \end{cases} $$
Using the binomial identity $\binom{a+b}{c}=\sum_{d}\binom{a}{d}\binom{b}{c-d}$, one readily checks that $S^j(\alpha \beta) = \sum_{j_1 + j_2 = j}S^{j_1}(\alpha)S^{j_2}(\beta)$ for all $\alpha,\beta \in R_t$. Now the proof of Lemma \ref{LEMlinearindependence} (below) also yields:

\begin{lemma} \label{LEMinjectivityofexternalproduct} The external product homomorphism $R_1 \otimes_{\mathbb{F}_2} \cdots \otimes_{\mathbb{F}_2} R_m \rightarrow \chowbar{X}$ is injective. \end{lemma}

Let $R_X$ be the image of $R_1 \otimes_{\mathbb{F}_2} \cdots \otimes_{\mathbb{F}_2} R_m$ in $\chowbar{X}$, equipped with the $\mathbb{F}_2$-algebra structure it inherits from the $\mathbb{F}_2$-algebra structures on $R_1,\hdots,R_m$ described above, and the gradings by dimension and codimension it inherits from $\chowbar{X}$. For each integer $j \geq 0$, we then have a unique $\mathbb{F}_2$-vector space homomorphism $S^j \colon R_X \rightarrow R_X$ with the property that $S^j(\alpha_1 \times \cdots \times \alpha_m) = \sum_{j_1+\cdots + j_m=j}S^{j_1}(\alpha_1) \times \cdots \times S^{j_m}(\alpha_m)$ for all $\alpha_1 \in R_1,\hdots,\alpha_m \in R_m$. We then have that $S^j(\alpha \beta) = \sum_{j_1 + j_2 = j}S^{j_1}(\alpha)S^{j_2}(\beta)$ for all $\alpha,\beta \in R_X$. Now, if $U$ is the smooth locus of $X$, then the group $\mathrm{Ch}(U)$ has a canonical ring structure (with multiplication given by the intersection product), as well as the action of the cohomological-type Steenrod operations $S_U^j$ of Primozic (\cite{Primozic}). The statements we need are the following:

\begin{proposition} \label{PROPbasicresultonrationalcycles} Let $U$ be the smooth locus of $X$, and $\iota \colon \colon U \rightarrow X$ the canonical open embedding. 
\begin{enumerate} \item The $\mathbb{F}_2$-algebra $R_X$ depends only on $\varphi_1,\hdots,\varphi_m$ $($and not on $W_1,\hdots,W_m)$;
\item $\ratchow{X}$ is a subring of $R_X$ and there is a unique surjective ring homomorphism $\theta \colon \mathrm{Ch}(U) \rightarrow \ratchow{X}$ with the property that $\theta(\iota^*(\alpha)) = \overline{\alpha}$ for all $\alpha \in \mathrm{Ch}(X)$. 
\item For each integer $j \geq 0$, $\ratchow{X}$ is stable under the map $S^j$, and we have $S^j(\overline{\alpha}) = \theta(S_U^j(\iota^*(\alpha)))$ for all $\alpha \in \mathrm{Ch}(X)$. 
\end{enumerate} \end{proposition} 

In light of (3), we refer to $S^j \colon R_X \rightarrow R_X$ as the \emph{$j$th cohomological-type Steenrod operation on $R_X$}.

\begin{remark}[Orientations and standard bases] \label{REMorientations} In the preceding discussion, we have described an explicit basis of $R_X$ (as an $\mathbb{F}_2$-vector space) consisting of the external products $\alpha_1 \times \cdots \times \alpha_m$ with $\alpha_t \in \lbrace l_i,h^i \rbrace_{0 \leq i <r_t}$ for each $1 \leq t \leq m$. It is clear from the first part of the proposition that, one exception aside, the elements of the set $\lbrace l_i,h^i \rbrace_{0 \leq i <r_t} \subset \chowbar{X_t}$ are independent of the choice of the linear subspace $W_t$. The exception concerns the element $l_{r_t - 1}$ in the case where $\mathrm{ql}(\varphi_t)$ is trivial (i.e., $\varphi_t$ is nondegenerate of even dimension). Here, it is well known that there are two non-equivalent classes of $(r_t-1)$-dimensional projective linear subspaces of $\overline{X_t}$ in $\chowbar{X_t}$.  The two classes are exchanged by any reflection automorphism of $\overline{X_t}$ and their sum is the element $h^{r_t-1}$ (see \cite[P.  308]{EKM}). Following \cite{EKM}, we refer to the choice of one of these classes as an \emph{orientation of $X_t$}.  An \emph{orientation of $X$} is then the choice of an orientation for each of the factors $X_1,\hdots,X_m$ for which one is required. 
By fixing an orientation of $X$ (when required), we obtain as above a basis for $R_X$ that we shall refer to as \emph{the standard basis}. We shall then say that an element of $R_X$ \emph{involves} a given element of the standard basis if that element appears in its decomposition as a sum of the standard basis elements.\end{remark}

We now prove the above statements, following the arguments of \cite{Karpenko2}. We first note:

\begin{lemma} \label{LEMreductiontomaximalWittindex} To prove the statements above, we can assume that each of the forms $\varphi_1,\hdots,\varphi_m$ has maximal Witt index, i.e., that $\windex{\varphi_t} = r_t$ for all $1 \leq t \leq m$. 
\begin{proof} Let $F_{\mathrm{sep}}$ be the separable closure of $F$ in $\overline{F}$. Since anisotropic quasilinear quadratic forms remain anisotropic under separable extensions, each of the forms $\varphi_t$ remains nondefective over $F_{\mathrm{sep}}$. On the other hand, we also have that $\windex{(\varphi_t)_{F_{\mathrm{sep}}}} = r_t$ for all $1 \leq t \leq m$ by Lemma \ref{LEMnondefectiveheight}. What we are then claiming is that in order to prove the statements of interest, we can replace $F$ with $F_{\mathrm{sep}}$. But this is clear in the case of Lemma \ref{LEMlinearindependence}, Lemma \ref{LEMinjectivityofexternalproduct} and Proposition \ref{PROPbasicresultonrationalcycles} (1), while for parts (2) and (3) Proposition \ref{PROPbasicresultonrationalcycles} (2) and (3) we only need to appeal to the commutativity of the diagram
$$ \xymatrix{\mathrm{Ch}(X)  \ar[d]_{i^*}  \ar[r] & \mathrm{Ch}(X_{F_{\mathrm{sep}}}) \ar[d]_{i^*} \ar[r] & \chowbar{X} \ar[d]_{i^*}\\
\mathrm{Ch}(U) \ar[r] & \mathrm{Ch}(U_{F_{\mathrm{sep}}}) \ar[r] & \chowbar{U}} $$
(in which the horizontal maps are given by scalar extension) and the fact that the bottom horizontal maps are ring homomorphisms that commute with the action of the cohomological-type Steenrod operations.
\end{proof} \end{lemma}

For the remainder of this section, we therefore assume that $\windex{\varphi_t} = r_t$ for all $1 \leq t \leq m$. With this assumption, we can make use of standard partial cell decompositions of the quadrics $X_1,\hdots,X_m$, which we now discuss.

\subsection{Partial Cell Decomposition of a Generically Smooth Quadric with Maximal Witt Index} \label{SUBSECpartialcelldec} Let $\varphi$ be a nonquasilinear quadratic form of dimension $\geq 2$ over $F$ with underlying vector space $V$. Let $(r,s)$ be the type of $\varphi$, and assume that $\windex{\varphi} = r$. As in \S \ref{SUBSECstatementsonrationalcycles} above, we can then find a $r$-dimensional totally isotropic subspace $W$ of $V$ which is a maximal totally isotropic subspace for a hyperbolic subform of $V$. For each $0 \leq i <r$, let us again write $l_i$ for the class of an $i$-dimensional projective linear subspace of $\mathbb{P}(W)$ in $\mathrm{Ch}_i(X_{\varphi})$, and $h^i$ for the class of a codimension $i$ subquadric of $X_{\varphi}$ in $\mathrm{Ch}^i(X_{\varphi})$ (again, $h^i$ does not depend on the choice of subquadric). Set $U:= W + V_{\mathrm{ql}(\varphi)}$, and let $Y$ be the reduced closed subscheme of $X_{\varphi}$ given by the intersection of $X_{\varphi}$ and $\mathbb{P}(U)$. Note that we can (and do) view the classes $l_i$ as elements of $\mathrm{Ch}(Y)$. Let us now fix a scheme $Z$ over $F$. By the projective bundle theorem, the localization sequence for the canonical embedding $\mathbb{P}(W) \times Z \rightarrow Y \times Z$ is split exact, and yields a split exact sequence
$$ 0 \rightarrow \bigoplus_{i=0}^{r-1} \mathrm{Ch}_{*-i}(Z) \rightarrow \mathrm{Ch}_*(Y \times Z) \rightarrow \mathrm{Ch}_*((X_{\varphi} \setminus Y) \times Z) \rightarrow 0 $$
where the first map sends $(\alpha_0,\hdots,\alpha_{r-1})$ to $\sum_{i=0}^{r-1} l_i \times \alpha_i$. At the same time, since $W \cap V_{\mathrm{ql}(\varphi)} = 0$, we have a canonical projection $\mathbb{P}(U)\setminus \mathbb{P}(W) \rightarrow \mathbb{P}(V_{\mathrm{ql}(\varphi)})$ that induces a rank-$r$ affine bundle $f \colon Y \setminus \mathbb{P}(W) \rightarrow X_{\mathrm{ql}(\varphi)}$. By homotopy invariance, the above split exact sequence then yields an $\mathbb{F}_2$-vector space isomorphism 
$$ \left(\bigoplus_{i=0}^{r-1} \mathrm{Ch}_{* - i}(Z)\right) \bigoplus \mathrm{Ch}_{*-r}(X_{\mathrm{ql}(\varphi)} \times Z) \rightarrow \mathrm{Ch}(Y \times Z), $$
where again the map on the component in parentheses is the external product map $(\alpha_0,\hdots,\alpha_{r-1}) \mapsto \sum_{i=0}^{r-1} l_i \times \alpha_i$. While we don't have a similarly concrete description of the map on the other component, we can at least say that its image lies in the image of the canonical pushforward $\mathrm{Ch}(X_{\mathrm{ql}(\varphi)} \times Y \times Z) \rightarrow \mathrm{Ch}(Y \times Z)$ by \cite[Lem. B.1]{Karpenko2}. Now, the canonical projection $\mathbb{P}(V)\setminus \mathbb{P}(U) \rightarrow \mathbb{P}(V/U)$ also induces a rank-$(r+s)$ affine bundle $g \colon X_{\varphi} \setminus Y \rightarrow \mathbb{P}(V/U)$. Note that the preimage of a codimension $i$ projective linear subspace of $\mathbb{P}(V/U)$ under $g$ is the intersection of a codimension-$i$ subquadric of $X_{\varphi}$ and $X_{\varphi} \setminus Y$. By \cite[Thm. 66.2]{EKM} and the projective bundle formula, we then have an $\mathbb{F}_2$-vector space isomorphism
$$ \left(\bigoplus_{i=0}^{r-1} \mathrm{Ch}_{*-\mydim{\varphi} +i +2}(Z)\right) \bigoplus \mathrm{Ch}_*(Y \times Z) \rightarrow \mathrm{Ch}_*(X_{\varphi} \times Z),$$
where the map on the component in parentheses is the external product map $(\alpha_0,\hdots,\alpha_{r-1}) \mapsto \sum_{i=0}^{r-1} h^i \times \alpha_i$, and the map on the other component is pushforward along the canonical embedding $Y \times Z \rightarrow X_{\varphi} \times Z$.
Putting everything together, we get:

\begin{proposition} \label{PROPcelldecomposition} If $Z$ is a scheme over $F$, then we have an $\mathbb{F}_2$-vector space decomposition $\mathrm{Ch}_*(X_{\varphi} \times Z) = A \bigoplus B$, where
\begin{itemize} \item $A$ is isomorphic to $\bigoplus_{i=0}^{r-1}\left(\mathrm{Ch}_{* - i}(Z) \bigoplus \mathrm{Ch}_{* - \mydim{\varphi} + i + 2}(Z)\right)$ via the map
$$ \bigoplus_{i=0}^{r-1}\left(\mathrm{Ch}_{* - i}(Z) \bigoplus \mathrm{Ch}_{* - \mydim{\varphi} + i + 2}(Z)\right) \rightarrow \mathrm{Ch}_*(X_{\varphi} \times Z) $$
that sends $\left((\alpha_0,\beta_0),\hdots,(\alpha_{r-1},\beta_{r-1})\right)$ to $\sum_{i=0}^{r-1} (l_i \times \alpha_i + h^i \times \beta_i)$;
\item $B$ is isomorphic to $\mathrm{Ch}_{*- r}(X_{\mathrm{ql}(\varphi)} \times Z)$, and lies in the image of the canonical pushforward $\mathrm{Ch}_*(X_{\mathrm{ql}(\varphi)} \times X_{\varphi} \times Z) \rightarrow \mathrm{Ch}_*(X_{\varphi} \times Z)$. 
\end{itemize} \end{proposition}

Note that the decomposition obtained here is clearly compatible with scalar extension. We can now justify the statements in \S \ref{SUBSECstatementsonrationalcycles}. We first consider the case of a single quadric. Note that this case is covered by \cite{Karpenko2}, but we include it here for completeness (and to clarify some technicalities in our presentation of the needed results).

\subsection{The Case Of a Single Quadric} Suppose here that $m=1$. For ease of notation, set $\varphi: = \varphi_1$, $V := V_1$ and $(r,s) : = (r_1,s_1)$. In order to prove the desired assertions, Witt's extension theorem (\cite[Thm. 8.3]{EKM}) allows us to assume that the totally isotropic subspace $W_1$ is equal to $W \otimes_F \overline{F}$ for some $F$-linear subspace $W$ of $V$. As in the previous subsection, we can then introduce elements $l_i,h^i \in \mathrm{Ch}(X)$ $(0 \leq i < r)$ restricting to the corresponding elements of $\chowbar{X}$ introduced in \S \ref{SUBSECstatementsonrationalcycles} (we make no notational distinction here). Following the statement of Proposition \ref{PROPbasicresultonrationalcycles}, let $U$ be the smooth locus of $X$, and let $\iota \colon U \rightarrow X$ be the canonical open embedding. By Proposition \ref{PROPvanishingofrationalcyclesforquasilinear}, we have $\ratchow{X_{\mathrm{ql}(\varphi)}} = 0$. By the localization sequence for the canonical embedding $X_{\mathrm{ql}(\varphi)} \rightarrow X$, it follows that there exists a surjective $\mathbb{F}_2$-vector space homomorphism $\theta \colon \mathrm{Ch}(U) \rightarrow \chowbar{X}$ such that $\theta(\iota^*(\alpha)) = \overline{\alpha}$ for all $\alpha \in \mathrm{Ch}(X)$. Now, by Proposition \ref{PROPcelldecomposition} (applied with $Z = \mathrm{Spec}(F)$), the set $\lbrace l_i,h^i \rbrace$ is an $\mathbb{F}_2$-linearly independent subset of $\mathrm{Ch}(X)$. Moreover, if we let $A$ denote its $\mathbb{F}_2$-linear span, then we have an $\mathbb{F}_2$-vector space decomposition $\mathrm{Ch}(X) = A \bigoplus B$, where $B$ is isomorphic to $\mathrm{Ch}(X_{\mathrm{ql}(\varphi)})$ and lies in the image of the pushforward $\mathrm{Ch}(X_{\mathrm{ql}(\varphi)} \times X) \rightarrow \mathrm{Ch}(X)$.

\begin{lemma} \label{LEMBisinthekernel} In the above situation:
\begin{enumerate} \item $\overline{\alpha} = 0$ for all $\alpha \in B$. In particular, $\iota^*(B) \subseteq \mathrm{Ker}(\theta)$;
\item For all $0 \leq i < r_t$, $\overline{l_i} \neq 0$. In particular, $\theta(\iota^*(l_i)) \neq 0$.  \end{enumerate}
\begin{proof} (1) Since $B$ lies in the image of the pushforward $\mathrm{Ch}(X_{\mathrm{ql}(\varphi)} \times X) \rightarrow \mathrm{Ch}(X)$, it suffices to show that $\ratchow{X_{\mathrm{ql}(\varphi)} \times X} = 0$. By the $Z = X_{\mathrm{ql}(\varphi)}$ case of Proposition \ref{PROPcelldecomposition}, however, $\mathrm{Ch}(X_{\varphi} \times X)$ is isomorphic to a direct sum of copies of $\mathrm{Ch}(X_{\mathrm{ql}(\varphi)} \times X_{\mathrm{ql}(\varphi)})$ and $\mathrm{Ch}(X_{\mathrm{ql}(\varphi)})$, and this decomposition is compatible with scalar extension. The desired assertion therefore follows from Lemma \ref{LEMvanishingofrationalcycles}. 

(2) Clear from the injectivity of the canonical pushforward $\chowbar{\mathbb{P}(W)} \rightarrow \chowbar{\mathbb{P}(V)}$. 
\end{proof} \end{lemma}

Next, we have:

\begin{lemma} \label{LEMformulasinChU}  In the ring $\mathrm{Ch}(U)$, the following identities hold modulo $\mathrm{Ker}(\theta)$:
$$ \iota^*(h^i)\iota^*(h^j): = \begin{cases} \iota^*(h^{i+j}) & \text{if } i+j < r_t \\ 0 & \text{otherwise} \end{cases}; \hspace{.5cm} \iota^*(h^i)\iota_t^*(l_j) : = \begin{cases} \iota^*(l_{j-i}) & \text{if } i \leq j \\ 0 & \text{otherwise}; \end{cases}; $$
$$ \iota^*(l_i)\iota^*( l_j): = \begin{cases} \iota^*(l_0) & \text{if } \mydim{\varphi_t} \equiv 2 \pmod{4} \text{ and } i= j = \frac{\mydim{\varphi_t}-2}{2} \\ 0 & \text{otherwise}. \end{cases} $$
\begin{proof} For the sake of legibility, we shall drop $\iota^*$ from our notation in what follows. Consider the canonical closed embeddings $\mathbb{P}(W) \xrightarrow{v} U \xrightarrow{w} \mathbb{P}(V) \setminus X_{\mathrm{ql}(\varphi)}$. Since $X$ is a quadric, the composition $w_* \circ w^*$ is zero. Now, since $w^*$ is a ring homomorphism, $h^ih^j$ is the pullback of the class of a codimension-$(i+j)$ projective linear subspace of $\mathbb{P}(V)$. In particular, if $i+j < r_t$, then $h^ih^j = h^{i+j}$. Suppose now that $i+j \geq r_t$. To prove the first identity in the statement, we have to show that $h^ih^j \in \mathrm{Ker}(\theta)$. By Lemma \ref{LEMBisinthekernel} (1), it suffices to show that $h^ih^j \in B$. If $k: = \mydim{\varphi} +2 - (i+j)$ is greater than or equal to $r_t$, this is clear. Suppose now that $k<r_t$, and that $h^ih^j \notin B$. Then $h^ih^j \equiv l_k \pmod{B}$. Since $w_* \circ w^* = 0$, it then follows that $\overline{w_*(l_k)} = 0$. But since $\ratchow{X_{\mathrm{ql}(\varphi)}} = 0$, the scalar extension homomorphism $\mathrm{Ch}(\mathbb{P}(V)\setminus X_{\mathrm{ql}(\varphi)}) \rightarrow \chowbar{\mathbb{P}(V)\setminus X_{\mathrm{ql}(\varphi)}}$ is injective, and hence $w_*(l_k) = 0$. As an element of $\mathrm{Ch}(X)$, $l_k$ then lies in the image of the canonical pushforward $\mathrm{Ch}(X_{\mathrm{ql}(\varphi)}) \rightarrow \mathrm{Ch}(X)$. But Lemma \ref{LEMvanishingofrationalcycles} then implies that $\overline{l_k} = 0$, contradicting part (2) of Lemma \ref{LEMBisinthekernel}. Thus, the first identity in the statement holds. For the second identity, the case where $i>j$ is clear, so assume that $i \leq j$. One can then clearly choose a codimension-$j$ projective linear subspace of $\mathbb{P}(W)$ and an $i$-dimensional projective linear subspace of $\mathbb{P}(W)$ whose scheme-theoretic intersection is a ($j-i)$-dimensional projective linear subspace of $\mathbb{P}(W)$. By \cite[Prop. 57.21]{EKM}, the identity $h^il_j = l_{j-i}$ then holds. Finally, for the third identity, dimension reasons give that $l_il_j = 0$ unless $s = 0$ and $i = j = \frac{\mydim{\varphi} -2}{2}$. But in this case, $X$ is smooth and the stated identity is a standard computation (see, e.g., \cite[P. 308]{EKM}).  \end{proof}
\end{lemma}

Recall now that $\mathrm{Ch}(U)$ contains the ideal $N_0(U)$ of numerically trivial elements (see \S \ref{SUBSECnumericaltriviality}). With the above, we can now see the following:

\begin{lemma} $\mathrm{Ker}(\theta) = N_0(U) = \iota^*(B)$. 
\begin{proof} We have already noted in Lemma \ref{LEMBisinthekernel} (1) that $\iota^*(B) \subseteq \mathrm{Ker}(\theta)$. At the same time, since $\iota^*(B)$ lies in the image of the pushforward $\mathrm{Ch}(X_{\mathrm{ql}(\varphi)} \times U) \rightarrow \mathrm{Ch}(U)$, it is also lies in $N_0(U)$ by Lemma \ref{LEMtrivialpushforward} ($X_{\mathrm{ql}(\varphi)}$ has no closed point of odd degree by Springer's theorem and the anisotropy of $\mathrm{ql}(\varphi)$). To complete the proof, it now suffices to show that if $\alpha$ is an nonzero element of $A$, then $\iota^*(\alpha)$ is neither an element of $\mathrm{Ker}(\theta)$ nor $N_0(U)$. To prove this, we may assume that $\alpha \in \mathrm{Ch}_k(X)$ for some integer $k$. By the multiplicative identities given in Lemma \ref{LEMformulasinChU}, there then exists an element $\beta \in \mathrm{Ch}^k(U)$ such that $\iota^*(\alpha)\beta = \iota^*(l_0)$. By Lemma \ref{LEMBisinthekernel} (2), this shows that $\iota^*(\alpha) \notin \mathrm{Ker}(\theta)$. At the same time, we have $\mathrm{deg}(\iota^*(l_0)) = 1$, and so the same identity shows that $\alpha \notin N_0(U)$. This proves the lemma.
\end{proof} \end{lemma}

With the preceding lemmas, we have the following: If $\alpha$ is a nonzero element of $A$, then $\overline{\alpha} \neq 0$. In particular, the set $\lbrace l_i,h^i \rbrace_{0 \leq i < r}$ is $\mathbb{F}_2$-linearly independent in $\chowbar{X}$. Moreover, $R_X = \mathrm{Span}_{\mathbb{F}_2}\lbrace l_i,h^i \rbrace_{0 \leq i < r}$ is precisely the the image of the map $\theta \colon \mathrm{Ch}(U) \rightarrow \chowbar{X}$, and the map $\theta \colon \mathrm{Ch}(U) \rightarrow R_X$ is a ring homomorphism by Lemma \ref{LEMformulasinChU} (and the definition of the multiplication on $R_X$ given in \S \ref{SUBSECstatementsonrationalcycles}). This proves all statements in \S \ref{SUBSECstatementsonrationalcycles} (with $m=1$) with the exception of part (3) of Proposition \ref{PROPbasicresultonrationalcycles}. This follows from:

\begin{lemma} \label{LEMSteenrodformula} Let $j$ be a positive integer. In $\mathrm{Ch}(U)$, the following identities hold modulo $\mathrm{Ker}(\theta)$:
$$ S_U^j(\iota^*(h^i)) : = \begin{cases} \binom{i}{j}\iota^*(h^{i+j}) & \text{if } j < r-i \\ 0 & \text{otherwise} \end{cases};\;\;\; S_U^j(\iota^*(l_i)): = \begin{cases} \binom{\mydim{\varphi} - i- 1}{j} \iota^*{(l_{i-j})} & \text{if } j \leq i \\ 0 & \text{otherwise}. \end{cases} $$
\begin{proof} For the sake of legibility, we shall again drop $\iota^*$ from our notation in what follows. By Lemma \ref{LEMformulasinChU}, the first identity says that $S_U^j(h^i) = \binom{i}{j} h^ih^j$. Since the cohomological-type Steenrod operations for smooth schemes commute with pullbacks, proving this identity then amounts to showing that if $\alpha$ and $\beta$ are the classes of codimension-$i$ and codimension-$j$ projective linear subspaces in $\mathbb{P}(V)$, respectively then $S_{\mathbb{P}(V)}^j(\alpha) = \binom{i}{j}\alpha\beta$. This is well known, however (see \cite[Ex. 6.16]{EKM}; the argument given there is valid in any characteristic). For the second identity, we can assume that $j \leq i$. Let $\alpha$ be the class of an $i$-dimensional projective linear subspace in $\mathbb{P}(W)$. If $v \colon \mathbb{P}(W) \rightarrow U$ is the canonical closed embedding, then we have that $S_U^j(l_i) = v_*(c_j(N_v)(\alpha))$, where $N_v$ is the normal bundle of $v$ (Wu formula). Now an argument essentially identical to that found in \cite[Lem. 78.1, Cor. 78.2]{EKM} shows that $c_j(N_v)$ is multiplication by $\binom{\mydim{\varphi} -i-1}{j}\beta$, where $\beta$ is the unique element of $\mathrm{Ch}^j(\mathbb{P}(W))$. Since $\alpha\beta$ is the class of an $(i-j)$-dimensional subspace of $\mathbb{P}(W)$, we then get that $S_U^j(l_i) = v_*(\binom{\mydim{\varphi}-i-1}{j}\alpha\beta) = \binom{\mydim{\varphi}-i-1}{j}l_{i-j}$, as desired. \end{proof}
\end{lemma}

\subsection{The General Case} We now consider the general case. For each $1 \leq t \leq m$, we can again assume that $W_t$ is obtained from a totally isotropic subspace of $V_t$ by scalar extension, and can introduce the elements $l_i,h^i \in \mathrm{Ch}(X_t)$ ($0 \leq i < r_t$) restricting to the corresponding elements of $\chowbar{X_t}$ introduced in \S \ref{SUBSECstatementsonrationalcycles}. Let $A$ be the $\mathbb{F}_2$-linear subspace of $\mathrm{Ch}(X)$ generated by all external products $\alpha_1 \times \cdots \times \alpha_m$ with $\alpha_t \in \lbrace l_i,h^i \rbrace_{0 \leq i < r_t}$ for all $t$. By repeated application of Proposition \ref{PROPcelldecomposition}, we get that these external products are $\mathbb{F}_2$-linearly independent, and that $\mathrm{Ch}(X) = A \bigoplus B$, where $B$ is:
\begin{enumerate} \item isomorphic to a direct sum of $\mathbb{F}_2$-vector spaces of the form $\mathrm{Ch}(X_{\mathrm{ql}(\varphi_{a_1})} \times \cdots \times X_{\mathrm{ql}(\varphi_{a_l})})$ for integers $1 \leq a_1 \leq \cdots \leq a_l \leq m$;
\item contained in the image of the pushforward $\mathrm{Ch}(X_{\mathrm{ql}(\varphi_1)} \times \cdots \times X_{\mathrm{ql}(\varphi_m)} \times X) \rightarrow \mathrm{Ch}(X)$. \end{enumerate}
Following the statement of Proposition \ref{PROPbasicresultonrationalcycles}, let $U$ be the smooth locus of $X$, and let $\iota \colon U \rightarrow X$ be the canonical open embedding. Let $Y = X \setminus U$ be the singular locus of $X$. Explicitly, $Y$ is the union of the closed subschemes $X_1 \times \cdots \times X_{t-1} \times X_{\mathrm{ql}(\varphi_t)} \times X_{t+1} \times \cdots \times X_m$ with $1 \leq t \leq m$. By further repeated application of Proposition \ref{PROPcelldecomposition}, $\mathrm{Ch}(Y)$ is then also isomorphic to a direct sum of $\mathbb{F}_2$-vector spaces of the form $\mathrm{Ch}(X_{\mathrm{ql}(\varphi_{a_1})} \times \cdots \times X_{\mathrm{ql}(\varphi_{a_l})})$ for integers $1 \leq a_1 \leq \cdots \leq a_l \leq m$. By Lemma \ref{LEMvanishingofrationalcycles}, we then have that $\ratchow{Y} = 0$, and so there exists a unique surjective $\mathbb{F}_2$-vector space homomorphism $\theta \colon \mathrm{Ch}(U) \rightarrow \chowbar{X}$ such that $\theta(\iota^*(\alpha)) = \overline{\alpha}$ for all $\alpha \in \mathrm{Ch}(X)$. As in the previous subsection, we have:

\begin{lemma} \label{LEMcomputationofthekernel} $\mathrm{Ker}(\theta) = N_0(U) = \iota^*(B)$. 
\begin{proof} Recall that $B$ lies in the image of the pushforward $\mathrm{Ch}(X_{\mathrm{ql}(\varphi_1)} \times \cdots \times X_{\mathrm{ql}(\varphi_m)} \times X) \rightarrow \mathrm{Ch}(X)$. In particular $\iota^*(B)$ lies in the image of the pushforward $\mathrm{Ch}(X_{\mathrm{ql}(\varphi_1)} \times \cdots \times X_{\mathrm{ql}(\varphi_m)} \times U) \rightarrow \mathrm{Ch}(U)$, and hence lies in $N_0(U)$ by Lemma \ref{LEMtrivialpushforward} (again, the $X_{\mathrm{ql}(\varphi_t)}$ have no closed points of odd degree by Springer's theorem). Moreover, to show that $\iota^*(B) \subseteq \mathrm{Ker}(\theta)$, it suffices to show that $\ratchow{X_{\mathrm{ql}(\varphi_1)} \times \cdots \times X_{\mathrm{ql}(\varphi_m)} \times X} = 0$. But repeated application of Proposition \ref{PROPcelldecomposition} shows that this group is isomorphic to a direct sum of groups of the form $\ratchow{Z}$, where $Z$ is a direct product of the quasilinear quadrics $X_{\mathrm{ql}(\varphi_1)},\hdots,X_{\mathrm{ql}(\varphi_m)}$ (with possibly repeated factors). The desired assertion then holds by Lemma \ref{LEMvanishingofrationalcycles}. To complete the proof of the lemma, it now only remains to show that if $\alpha$ is a nonzero element of $A$, then $\iota^*(\alpha)$ is neither an element of $\mathrm{Ker}(\theta)$ nor $N_0(U)$. We can assume here that $\alpha \in \mathrm{Ch}_k(X)$ for some integer $k$. Since the multiplication in $\mathrm{Ch}(U)$ is compatible with the formation of external products, Lemma \ref{LEMformulasinChU} then shows that there exists a $\beta \in \mathrm{Ch}^k(U)$ such that $\iota^*(\alpha)\beta = \iota^*(l_0 \times \cdots \times l_0)$. Since $\mathrm{deg}(l_0 \times \cdots \times l_0) = 1$, the claims follow.  \end{proof} \end{lemma}

With this, it follows that the map $\theta$ factors through $R_X$. Since $\mathrm{Ker}(\theta) = N_0(U)$ is an ideal in $\mathrm{Ch}(U)$, and since the multiplication in $\mathrm{Ch}(U)$ is compatible with the formation of external products, it then follows from Lemma \ref{LEMformulasinChU} (and the definition of the multiplication on $R_X$ given in \S \ref{SUBSECstatementsonrationalcycles}) that the map $\theta \colon \mathrm{Ch}(U) \rightarrow R_X$ is a ring homomorphism. In particular, $\ratchow{X}$ is a subring of $R_X$, and $\theta$ is the unique surjective ring homomorphism $\mathrm{Ch}(U) \rightarrow \ratchow{X}$ with $\theta(\iota^*(\alpha)) = \overline{\alpha}$ for all $\alpha \in \mathrm{Ch}(X)$. Since the cohomological-type Steenrod operations on $\mathrm{Ch}(U)$ satisfy the Cartan formula, Lemma \ref{LEMSteenrodformula} (together with the definition of the operations $S^j \colon R_X \rightarrow R_X$ given in \S \ref{SUBSECstatementsonrationalcycles}) also gives that $S^j(\overline{\alpha}) = \theta(S^j_U(\iota^*(\alpha)))$ for all integers $j \geq 0$ and $\alpha \in \mathrm{Ch}(X)$. Since we have already seen that $R_X$ depends only on $\varphi_1,\hdots,\varphi_m$, this proves all the statements in \S \ref{SUBSECstatementsonrationalcycles}. 

\section{Composition Law for Rational Correspondences} \label{SECcorrespondences} If $X$ is a product of generically smooth nondefective quadrics over $F$, we have introduced in \S \ref{SECrationalcycles} above the $\mathbb{F}_2$-algebra $R_X \subseteq \chowbar{X}$. If $Y$ is another variety over $F$ of the same type, then the $\mathbb{F}_2$-algebra $R_{X \times Y}$ is canonically identified with $R_X \otimes_{\mathbb{F}_2} R_Y$ via the external product homomorphism $\chowbar{X} \otimes_{\mathbb{F}_2} \chowbar{Y} \rightarrow \chowbar{X \times Y}$. 

\subsection{Definitions} \label{SUBSECcompositiondefinition} Let $X$, $Y$ and $Z$ be products of generically smooth nondefectivequadrics over $F$. Let $\pi_{XY} \colon X \times Y \times Z \rightarrow X \times Y$, $\pi_{XZ} \colon X \times Y \times Z \rightarrow X \times Z$ and $\pi_{YZ} \colon X \times Y \times Z \rightarrow Y \times Z$ be the canonical (flat and proper) projections. On mod-$2$ Chow groups, the pullbacks $\pi_{XY}^*$ and $\pi_{YZ}^*$ are given by the assignments $\alpha \mapsto \alpha \times h^0$ and $\alpha \mapsto h^0 \times \alpha$, respectively. In particular, $\pi_{XY}^*(R_{X \times Y}) \subseteq R_{X \times Y \times Z}$ and $\pi_{YZ}^*(R_{Y \times Z}) \subseteq R_{X \times Y \times Z}$. Similarly, if $\alpha \in \chowbar{X}$, $\beta \in \chowbar{Y}$ and $\gamma \in \chowbar{Z}$, then $(\pi_{XZ})_*(\alpha \times \beta \times \gamma) = \mathrm{deg}(\beta)\alpha \times \gamma$, and so $(\pi_{XZ})_*(R_{X \times Y \times Z}) \subset R_{X \times Z}$. We therefore have an $\mathbb{F}_2$-vector space homomorphism 
$$ R_{X \times Y} \otimes_{\mathbb{F}_2} R_{Y \times Z} \rightarrow R_{X \times Z}, \;\;\; \alpha \otimes \beta \mapsto (\pi_{XZ})_*\big(\pi_{XY}^*(\alpha)\pi_{YZ}^*(\beta)\big).  $$
We write $\beta \circ \alpha$ for the image of $\alpha \otimes \beta$ under this map. By the preceding discussion, we have the following explicit description of this composition law:

\begin{lemma} \label{LEMcompositionformulas} If $\alpha \in R_X$, $\beta,\beta' \in R_Y$ and $\gamma \in R_Z$, then $(\beta' \times \gamma) \circ (\alpha \times \beta) = \mathrm{deg}(\beta\beta')\alpha \times \gamma$.
\end{lemma}

Note that our definition and the subsequent discussion remains valid if we take $X$ or $Z$ to be $\mathrm{Spec}(F)$. We therefore also permit this possibility in what follows. For a fixed $\alpha \in R_{X \times Y}$, we write $\alpha^*$ (resp. $\alpha_*$) for the $\mathbb{F}_2$-vector space homomorphism $R_{Y \times Z} \rightarrow R_{X \times Z}$ (resp. $R_{Z \times X} \rightarrow R_{Z \times Y}$) that sends $\beta$ to $\beta \circ \alpha$ (resp. $\alpha \circ \beta$). 

By Proposition \ref{PROPbasicresultonrationalcycles}, our composition law restricts to a composition law
$$ \ratchow{X \times Y} \otimes_{\mathbb{F}_2} \ratchow{Y \times Z} \rightarrow \ratchow{X \times Z} , \;\;\; \alpha \otimes \beta \mapsto \beta \circ \alpha. $$
In particular, if $\alpha \in \ratchow{X \times Y}$, then $\alpha^*$ (resp. $\alpha_*$) induces a homomorphism from $\ratchow{Y \times Z}$ to $\ratchow{X \times Z}$ (resp. $\ratchow{Z \times X}$ to $\ratchow{Z \times Y}$). 

\subsection{Graphs} \label{SUBSECgraphs} Let $X$ and $Y$ be products of generically smooth nondefective quadrics. Let $U_X$ and $U_Y$ be the smooth loci of $X$ and $Y$, respectively, and let $\iota_X \colon U_X \rightarrow X$ and $\iota_Y \colon U_Y \rightarrow Y$ be the canonical open embeddings. Let $f \colon X \rightarrow Y$ be a morphism with the property that $f(U_X) \subseteq U_Y$, and let $\Gamma_f$ be its graph in $X \times Y$. Note that $f$ is proper.  The following lemma extends basic facts on morphisms of smooth schemes to our situation:

\begin{proposition} \label{PROPgraphformulas} Let $Z$ be $\mathrm{Spec}(F)$ or a product of generically smooth nondefective quadrics over $F$. Then:
\begin{enumerate} \item $\overline{[\Gamma_f]}_* \colon R_{Z \times X} \rightarrow R_{Z \times Y}$ coincides with the pushforward $(\mathrm{id} \times f)_*$;
\item If $f$ is flat or a regular closed embedding of constant relative dimension, or if $Y \times Z$ is smooth, then $\overline{[\Gamma_f]}^* \colon R_{Y \times Z} \rightarrow R_{X \times Z}$ coincides with the pullback $(f \times \mathrm{id})^*$.  \end{enumerate}
\begin{proof} In view of Lemma \ref{LEMcompositionformulas} (and the remarks at the beginning of this section), it suffices to consider the case where $Z = \mathrm{Spec}(F)$. As in the proof of Lemma \ref{LEMreductiontomaximalWittindex}, we may also replace $F$ with its separable closure in $\overline{F}$, and hence assume that $R_X = \ratchow{X}$, $R_{Y} = \ratchow{Y}$ and $R_{X \times Y} = \ratchow{X \times Y}$ (see \S \ref{SUBSECpartialcelldec}). Let $\pi_X \colon X \times Y \rightarrow X$ and $\pi_Y \colon X \times Y \rightarrow Y$ be the canonical projections. Set $U: = U_X \times U_Y$ and $\iota = \iota_X \times \iota_Y \colon U \rightarrow X \times Y$. Since $f(U_X) \subseteq U_Y$, we then have the Cartesian diagram
\begin{equation} \label{diagram1} \xymatrix{U_X \ar[d]_{\iota_X}  \ar[r]^{h}  &  U \ar[d]^{\iota} \\
X \ar[r]^{g} &   X \times Y} \vspace{.5 \baselineskip}\end{equation}
where $g = \mathrm{id} \times f$ and $h$ is the restriction of $g$ to $U_X$. In particular, $\iota^* \circ g_* = h_* \circ \iota_X^*$. Under our assumption, Proposition \ref{PROPbasicresultonrationalcycles} (2) tells us that there is a unique surjective ring homomorphism $\theta \colon \mathrm{Ch}(U) \rightarrow R_{X \times Y}$ such that $\theta(\iota^*(\alpha)) = \overline{\alpha}$ for all $\alpha \in \mathrm{Ch}(X \times Y)$. 

(1) Let $\alpha \in \mathrm{Ch}(X)$.  Since $R_X = \ratchow{X}$, proving the claim amounts to showing that $\overline{[\Gamma_f]}_*(\overline{\alpha}) = \overline{f_*(\alpha)}$. By definition, $\overline{[\Gamma_f]}_*(\overline{\alpha})$ is the image of the product $\overline{\pi_X^*(\alpha)} \cdot \overline{[\Gamma_f]} \in R_{X \times Y}$ under the pushforward $(\pi_Y)_*$. Since $f_* = (\pi_Y \circ g)_* = (\pi_Y)_* \circ g_*$, proving the claim thus reduces to showing that $\overline{\pi_X^*(\alpha)} \cdot \overline{[\Gamma_f]} = \overline{g_*(\alpha)}$. Since the multiplication on $R_{X \times Y}$ is induced by that on $\mathrm{Ch}(U)$ via $\theta$, it suffices here to show that $\iota^*\left(\pi_X^*(\alpha)\right)\cdot \iota^*([\Gamma_f]) = \iota^*\left(g_*(\alpha)\right)$ in $\mathrm{Ch}(U)$. But since $[\Gamma_f] = g_*([X])$, the projection formula gives that 
\begin{alignat*}{3} \iota^*\left(\pi_X^*(\alpha)\right)\cdot \iota^*([\Gamma_f]) \quad & = \quad && \iota^*\left(\pi_X^*(\alpha)\right) \cdot h_*([U_X])\\
\quad & = \quad && h_*\left(h^*\left(\iota^*\left(\pi_X^*(\alpha)\right)\right)\right) \\ \quad &  = \quad && h_*\left((\pi_X \circ \iota \circ h)^*(\alpha)\right) \\
\quad &  = \quad &&  h_*\left(\iota_X^*(\alpha)\right) \\
\quad &  = \quad &&  \iota^*\left(g_*(\alpha)\right), \end{alignat*}
and so the claim holds.

(2) Let $\alpha \in \mathrm{Ch}(Y)$.  Since $R_Y = \ratchow{Y}$, proving the claim amounts to showing that $\overline{[\Gamma_f]}^*(\overline{\alpha}) = \overline{f^*(\alpha)}$. By definition, $\overline{[\Gamma_f]}^*(\overline{\alpha})$ is the image of the product $\overline{\pi_Y^*(\alpha)} \cdot \overline{[\Gamma_f]} \in R_{X \times Y}$ under the pushforward $(\pi_X)_*$. We first compute the product, which is the image of $\iota^*\left(\pi_Y^*(\alpha)\right) \cdot \iota^*([\Gamma_f]) \in \mathrm{Ch}(U)$ under $\theta$. By the projection formula, we have 
\begin{alignat*}{3} \iota^*\left(\pi_Y^*(\alpha)\right)\iota^*([\Gamma_f]) \quad & = \quad && \iota^*\left(\pi_Y^*(\alpha)\right)h_*([U_X])\\
\quad & = \quad && h_*\left(h^*\left(\iota^*\left(\pi_Y^*(\alpha)\right)\right)\right) \\ \quad &  = \quad && h_*\left((\pi_Y \circ \iota \circ h)^*(\alpha)\right) \\
\quad & = \quad && h_*\left((f \circ \iota_X)^*(\alpha)\right) \\
\quad & = \quad && h_*\left(\iota_X^*\left(f^*(\alpha)\right)\right) \\
\quad & = \quad && \iota^*\left(g_*\left(f^*(\alpha)\right)\right), \end{alignat*}
and so $\overline{\pi_Y^*(\alpha)} \cdot \overline{[\Gamma_f]} = \overline{g_*\left(f^*(\alpha)\right)}$ by the defining property of $\theta$. Since $(\pi_X)_* \circ g_* = (\pi_X \circ g)_* = (\mathrm{id}_X)_*$ is the identity, applying $(\pi_X)_*$ then gives the desired result. \end{proof} \end{proposition}

\subsection{The Algebra of Rational Correspondences} As above, let $X$ be a product of generically smooth nondefective quadrics over $F$. Taking $Y = Z = X$ in \S \ref{SUBSECcompositiondefinition}, we get a composition operation $\circ \colon R_{X^2} \times R_{X^2} \rightarrow R_{X^2}$. If we let $\Delta_X$ be the image of the diagonal embedding $X \rightarrow X \times X$, then we have:

\begin{proposition} \label{PROPalgebrastructureforX^2} The operation $\circ$ equips $R_{X^2}$ with the structure of an $\mathbb{F}_2$-algebra with identity $[\overline{\Delta_X}]$. Moreover, with this structure, $R_{X^2}$ admits $\ratchow{X^2}$ as an $\mathbb{F}_2$-subalgebra. 
\begin{proof} It is clear from Lemma \ref{LEMcompositionformulas} that $\circ$ is associative and $\mathbb{F}_2$-linear.  Since $\Delta_X$ is the graph of the identity morphism $X \rightarrow X$, the neutrality of $[\overline{\Delta_X}]$ for $\circ$ follows from Proposition \ref{PROPgraphformulas}. This proves the first statement. At the same time, we have already noted that $\circ$ preserves $F$-rationality of cycles, and since $[\overline{\Delta_X}]$ is $F$-rational, the second statement follows. \end{proof} \end{proposition}

\section{Isotropic Reduction} \label{SECisotropicreduction} Let $\varphi_1,\hdots,\varphi_m$ be nonquasilinear nondefective quadratic forms over $F$. For each $1 \leq t \leq m$, we set $X_t: = X_{\varphi_t}$. We are interested in the product $X: = X_1 \times \cdots \times X_m$. If necessary, we fix an orientation of $X$ (see Remark \ref{REMorientations}). We may then consider the standard basis of the $\mathbb{F}_2$-vector space $R_X$ consisting of certain external products of classes of projective linear subspaces and plane sections of the quadrics $X_t$ (again, see Remark \ref{REMorientations}).

\subsection{A Computation for a Single Quadric} Let us suppose here that $m = 1$. For ease of notation, set $\varphi := \varphi_1$ and $V: = V_{\varphi}$. Suppose that $\varphi$ has type $(r,s)$. We assume that $r-\windex{\varphi} \geq 1$, i.e., that the anisotropic part of $\varphi$ is not quasilinear.  Set $Y: = X_{\anispart{\varphi}}$. Again, we choose an orientation of $Y$ if necessary, and consider the standard basis of the $\mathbb{F}_2$-vector space $R_Y$. We then have the following:

\begin{lemma} \label{LEMclassofthediagonal} If $j = \windex{\varphi}$, then the element
$$ \alpha: = \sum_{i=0}^{r - j-1} (h^{i + j} \times l_i) + (l_{i +j} \times h^i) \in R_{X \times Y}$$
is $F$-rational. 
\begin{proof} let $W$ be a totally isotropic $j$-dimensional subspace of $V$. We can then identify $Y$ with the variety of totally isotropic $(j+1)$-dimensional subspaces of $V$ containing $W$.  Under this identification, let $Z$ be the closed subscheme of $X \times Y$ consisting of all pairs $(L,U)$ with $L \subset U$.  Note that the canonical projection $Z \rightarrow Y$ is a projective bundle of rank $j$ ($Z$ is the projectivization of the restriction of the tautological vector bundle on the Grassmannian of $(j+1)$-dimensional subspaces of $V$ to $Y$). In particular, $Z$ is integral of dimension $\mathrm{dim}(Y) + a = \mathrm{dim}(X) - a$.  We state the following: \vspace{.5 \baselineskip}

\noindent {\bf Claim.} For each $0 \leq i < r-j$, we have $\overline{[Z]}(h^{i+j} \times l_i) = \overline{[Z]}(l_{i+j} \times h^i) = l_0 \times l_0$ in $R_{X \times Y}$. \vspace{.5 \baselineskip}

Given the claim, let us show that $\alpha$ is $F$-rational. Suppose first that $s \neq 0$, i.e., that $\mathrm{ql}(\varphi)$ is non-trivial. Since $\mydim{Z} = \mydim{X} - j$, we have scalars $a_i,b_i \in \mathbb{F}_2$ such that 
$$ \overline{[Z]} = \sum_{i=0}^{r-j-1} a_i (h^{i+j} \times l_i) + b_i(l_{i+j} \times h^i)$$
in $R_{X \times Y}$.  For each $0 \leq i < r-j$, the claim then gives that
$$ l_0 \times l_0 = \left(\sum_{i=0}^{r-j-1} a_i (h^{i+j} \times l_i) + b_i (l_{i+j} \times h^i)\right)(h^{i+j} \times l_i) = a_il_0 \times l_0 $$
and
$$ l_0 \times l_0 = \left(\sum_{i=0}^{r-j-1} a_i (h^{i+j} \times l_i) + b_i (l_{i+j} \times h^i)\right)(l_{i+j} \times h^i) = b_il_0 \times l_0, $$
and so $a_i = b_i = 1$. Thus, in this case, we have $\alpha = \overline{[Z]} \in \ratchow{X \times Y}$.  Suppose now that $s = 0$, i.e., that $\varphi$ is nondegenerate of even dimension. Then 
$$ \overline{[Z]} = \left( \sum_{i=0}^{r-j-2}a_i(h^{i+j} \times l_i) + b_i(l_{i+j} \times h^i)\right) + \left(h^{r-1} \times \alpha\right) + \left(l_{r-1} \times \beta\right)$$
for some $a_i,b_i \in \mathbb{F}_2$ and $\alpha,\beta \in \mathbb{F}_2h^{r-1} + \mathbb{F}_2l_{r-1} \subset R_Y$.  Using the claim as above, we get that $a_i = b_i = 1$ for all $0 \leq i < r-j-1$, and that $\alpha h^{r-j-1} = l_0 = \beta l_{r-j-1}$ in $R_Y$.  The equality $\alpha h^{r-j-1}$ gives that $\alpha = l_{r-j-1} + ah^{r-j-1}$ for some $a \in \mathbb{F}_2$. Let $b,c \in \mathbb{F}_2$ be such that $\beta = bh^{r-j-1} + cl_{r-j-1}$.  Then $\overline{[Z]}(h^{r-1} \times h^{r-j-1}) = cl_0 \times l_0$ in $R_{X \times Y}$. Since $h^{r-1} \times h^{r-j-1}$ is $F$-rational, it follows that $cl_0 \times l_0 \in \ratchow{X \times Y}$, so $X \times Y$ admits a $0$-cycle of degree $c$. In particular, $Y$ admits a $0$-cycle of degree $c$. Since $Y$ has no $F$-rational points, Springer's theorem then tells us that $c = 0$, and so $\beta = bh^{r-j-1}$. Since $\beta l_{r-j-1} = l_0$, we must then have that $b =1$, and so
\begin{alignat*}{3} \overline{[Z]} \quad & = \quad &&  \left( \sum_{i=0}^{r-j-2}(h^{i+j} \times l_i) + (l_{i+j} \times h^i)\right) + \left(h^{r-1} \times (l_{r-j-1} + ah^{r-j-1}) \right) + \left(l_{r-1} \times h^{r-j-1}\right) \\ \quad & = \quad && \alpha + \left(ah^{r-1} \times h^{r-j-1}\right). \end{alignat*}
Since $h^{r-1} \times h^{r-j-1}$ is $F$-rational, the same is then true of $\alpha = \overline{[Z]} + \left(ah^{r-1} \times h^{r-j-1}\right)$.

It now only remains to prove the claim.  Let $0 \leq i < r-j-1$. To prove the desired identities, we are free to pass from $F$ to its separable closure in $\overline{F}$. As in the proof of Lemma \ref{LEMreductiontomaximalWittindex}, this preserves the nondefectivity of $\varphi$ while allowing us to assume that $\windex{\varphi} = r$. Let $W'$ be an $(i+j+1)$-dimensional totally isotropic subspace of $V$ containing $W$. In the case where $s = 0$ and $i = r-j-1$, we choose $W'$ so that its restriction to $\overline{F}$ represents the class $l_{r-j}$ in $R_Y$ (this can always be achieved by Remark \ref{REMorientations}). Choose also an $(i+j)$-codimensional subspace $V'$ of $V$ such that $V' \cap W'$ is a $1$-dimensional subspace of $V$ not contained in $W$.  Let $S$ be the closed subscheme of $X \times Y$ consisting of all pairs $(L,U)$ with $L \subset V'$ and $U \subset W'$.  The (scheme-theoretic) intersection $Z \cap S$ is then the reduced subscheme of $X \times Y$ given by the $F$-rational point $(V' \cap W', W + (V' \cap W'))$. Since $\varphi$ is nondefective, this point lies in the smooth locus of $X \times Y$.  By Proposition \ref{PROPbasicresultonrationalcycles} (2), it then follows that $\overline{[Z]} \cdot \overline{[S]}$ is the class of an $\overline{F}$-rational point in $\overline{X \times Y}$.  Since $\overline{[S]} = h^{i+j} \times l_i$ by construction, this means that $\overline{[Z]}(h^{i+j} \times l_i) = l_0 \times l_0$ in $R_{X \times Y}$. This proves the first of the desired identities, and the other may be obtained by analogous considerations.
\end{proof} \end{lemma}

\subsection{Isotropic Reduction} \label{SUBSECisotropicreduction} We now return to the general situation considered at the beginning of this section.  For each $1 \leq t \leq m$, let $(r_t,s_t)$ be the type of $\varphi_t$. We assume that $r_t > \windex{\varphi_t}$, i.e., that the anisotropic part of $\varphi_t$ is nonquasilinear. We set $Y: = X_{\anispart{(\varphi_1)}} \times \cdots \times X_{\anispart{(\varphi_m)}}$.  We fix an orientation of $Y$ if needed, and consider the standard basis of the $\mathbb{F}_2$ vector space $R_Y$.  

\begin{proposition} \label{PROPisotropicreduction} There are unique $\mathbb{F}_2$-vector space homomorphisms $f \colon R_X \rightarrow R_Y$ and $g \colon R_Y \rightarrow R_X$ such that:
\begin{enumerate} \item $f \circ g = \mathrm{id}$;
\item $f(\ratchow{X}) = \ratchow{Y}$ and $g(\ratchow{Y}) \subseteq f(\ratchow{X})$;
\item If $ \alpha = \alpha_1 \times \cdots \times \alpha_m$ is a standard basis element of $R_X$, then $f(\alpha) = \beta_1 \times \cdots \times \beta_m$, where for each $1 \leq t \leq m$, we have
$$ \beta_t = \begin{cases} l_{i -\windex{\varphi_t}} & \text{if }\alpha_t = l_i \text{ for some } \windex{\varphi_t} \leq i<r_t \\
h^{i - \windex{\varphi_t}} & \text{if } \alpha_t = h^i \text{ for some }\windex{\varphi_t} \leq i<r_t \\ 0 & \text{otherwise.}  \end{cases} $$
\item  If $ \beta = \beta_1 \times \cdots \times \beta_m$ is a standard basis element of $R_Y$, then $g(\beta) = \alpha_1 \times \cdots \times \alpha_m$, where for each $1 \leq t \leq m$, we have
$$ \alpha_t = \begin{cases} l_{i + \windex{\varphi_t}} & \text{if }\beta_t = l_i \text{ for some } 0 \leq i < r_t - \windex{\varphi_t} \\ h^{i +\windex{\varphi_t}} & \text{if } \beta = h^i \text{ for some } 0 \leq i < r_t - \windex{\varphi_t}. \end{cases} $$ \end{enumerate}
\begin{proof} Note that the conditions in (3) and (4) determine $f$ and $g$ uniquely, and imply that they satisfy (1). It therefore suffices to construct a pair $(f,g)$ satisfying (2), (3) and (4).  By the discussion of \S \ref{SUBSECcompositiondefinition}, it will be enough to exhibit an $F$-rational correspondence $\alpha \in \ratchow{X \times Y}$ such that the pair $(f = \alpha_*,g = \alpha^*)$ satisfies (3) and (4). Since the obvious external product map $\bigotimes_{i=1}^m R_{X_i \times Y_i} \rightarrow R_{X \times Y}$ is an isomorphism of $\mathbb{F}_2$-algebras (Lemma \ref{LEMinjectivityofexternalproduct}), we may assume that $m=1$. For ease of notation, set $\varphi : = \varphi_1$ and $j: = \windex{\varphi}$. If $(r,s)$ is the type of $\varphi$, then $r-j \geq 1$ by hypothesis, and Lemma \ref{LEMclassofthediagonal} gives that
$$ \alpha: = \sum_{i=0}^{r - j-1} (h^{i + j} \times l_i) + (l_{i +j} \times h^i) \in R_{X \times Y}$$
is $F$-rational.  Using Lemma \ref{LEMcompositionformulas}, one readily checks that $\alpha$ has the desired properties.\end{proof} \end{proposition}

\subsection{Isotropy Restrictions on the Ring of Rational Cycles} Let us now relax the assumption on the anisotropic parts of the $\varphi_t$ from the previous subsection.  Recall (Remark \ref{REMorientations}) that we say an element of $R_X$ \emph{involves} a given element of the standard basis if that element appears in its decomposition as a sum of the standard basis elements. With Proposition \ref{PROPisotropicreduction}, we can establish the following restriction on the subring $\ratchow{X} \subseteq R_X$:

\begin{proposition} \label{PROPisotropyrestrictions} Let $\alpha$ be an element of $\ratchow{X}$ involving the standard basis element $h^{i_1} \times \cdots \times h^{i_k} \times l_{i_{k+1}} \times \cdots \times l_{i_{m}}$ for some non-negative integers $k < m$ and $i_1,\hdots,i_m$.  Suppose that $K$ is an extension of $F$ with the following properties:
\begin{itemize} \item[$\mathrm{(i)}$] $\windex{(\varphi_t)_K} > i_t$ for all $1 \leq t \leq k$;
\item[$\mathrm{(ii)}$] $(\varphi_t)_K$ is nondefective for all $1 \leq t \leq m$. \end{itemize}
Then $\windex{(\varphi_t)_K} > i_t$ for all $k<t \leq m$. 
\begin{proof} In view of (ii), we can replace $F$ with $K$ and $\overline{F}$ with $\overline{K}$ to reduce to the case where $K = F$.  Let $k < t \leq m$, and set $X': = \prod_{1 \leq i \leq m, i \neq t}X_{\varphi_i}$.  By (i), the element $\beta: = l_{i_1} \times \cdots \times l_{i_k} \times h^{i_{k+1}} \times \cdots \times h^{i_{t-1}} \times h^{i_{t+1}} \times \cdots \times h^{i_m} \in R_{X'}$ lies in $\ratchow{X'}$. If we view $\alpha$ as a correspondence from $X'$ to $X_{\varphi_t}$, then $\alpha_*(\beta) = l_{i_t} \in R_{X_{\varphi_t}}$. Since $\alpha$ and $\beta$ are $F$-rational, it follows that $l_{i_t} \in \ratchow{X_{\varphi_t}}$. Suppose, for the sake of contradiction, that $\windex{\varphi_t} \leq i_t$.  If we set $Y: = X_{\anispart{(\varphi_t)}}$, Proposition \ref{PROPisotropicreduction} then implies that $l_i \in \ratchow{Y}$ for some $i \geq 0$. Since $\ratchow{Y}$ is a subring of $R_Y$, $l_0 = h^il_i$ then lies in $\ratchow{Y}$, and so $Y$ admits a $0$-cycle of degree $1$. By Springer's theorem, however, this implies that $\anispart{(\varphi_t)}$ is isotropic, which is impossible. The desired assertion therefore holds.  \end{proof} \end{proposition}

Additional (and more subtle) restrictions on $\ratchow{X}$ may be obtained by considering the action of the Steenrod operations of \S \ref{SUBSECstatementsonrationalcycles} (see Proposition \ref{PROPbasicresultonrationalcycles} (3)). We will take these into account in the next section, where we restrict to the case where $X$ is the product of two copies of a single generically smooth nondefective quadric. 

\section{The $\mathrm{MDT}$ Invariant} \label{SECMDT}

In this section, we fix a nondefective quadratic form $\varphi$ of type $(r,s)$ over $F$. We assume that $r \geq 1$, i.e., that $\varphi$ is not quasilinear.  We let $h_{\mathrm{nd}}$ be the nondefective height and $F = F_0 \subseteq F_1 \subseteq \cdots \subseteq F_{h_{\mathrm{nd}}}$ the nondefective splitting tower of $\varphi$ (see \S \ref{SUBSECfunctionfieldsandsplittingpatterns}). Recall that for each $0 \leq t < h_{\mathrm{nd}}$, we write $\varphi_t$ for the anisotropic part of $\varphi_{F_t}$, which is still nondefective and nonquasilinear. For each $0 \leq t \leq h_{\mathrm{nd}}$, we set $\mathfrak{j}_t : = \wittj{t}{\varphi}$ and $\mathfrak{i}_t : = \witti{t}{\varphi}$.  We also set $X: = X_{\varphi}$ and $X_t: = X_{\varphi_t}$ for $t \neq h_{\mathrm{nd}}$. In the case where $s = 0$, we choose orientations of $X$ and each $X_t$ (Remark \ref{REMorientations}). For each positive integer $m$, we choose the orientation of $X^m$ (resp. $X_t^m$) by assigning the given orientation of $X$ (resp. $X_t$) to each factor. This being fixed, we can then consider in all cases the standard basis of $R_{X^m}$ (resp. $R_{X_t}^m$) consisting of $m$-fold external products of elements in the set $\lbrace h^i,l_i \rbrace_{0 \leq i <r}$ (resp. $\lbrace h^i,l_i \rbrace_{0 \leq i < r - \mathfrak{j}_t}$) (Lemma \ref{LEMinjectivityofexternalproduct}). We will also have occasion to pass from $F$ to some larger field $K$ preserving the nondefectivity of $\varphi$. When doing so, we shall give $X_K$ the orientation compatible with that on $X$ (i.e., if $\pi \colon X_K \rightarrow X$ is the canonical projection, then we choose the orientation of $X_K$ for which $\pi^*(l_{r-1}) = l_{r-1}$). This also applies to powers of $X_K$.  In what follows, we are mainly interested in studying the rationality of cycles on $X^2$.  However, this also necessitates consideration of higher powers of $X$. We start with some generalities. 

\subsection{Notation, Terminology and Preliminary Facts} \label{SUBSECMDTpreliminaries} Fix a positive integer $m$. Following \cite[\S 72]{EKM}, we say that an element of $R_{X^m}$ is \emph{essential} if it doesn't involve any standard basis elements of the form $h^{i_1} \times \cdots \times h^{i_m}$ for some integers $0 \leq i_t <r$. If $\mathrm{Ess}(R_{X^m})$ is the $\mathbb{F}_2$-linear subspace of $R_{X^m}$ consisting of these elements, then there is a unique $\mathbb{F}_2$-vector space homomorphism $\mathrm{ess} \colon R_{X^m} \rightarrow \mathrm{Ess}(R_{X^m})$ that fixes each essential element of the standard basis and sends the others to $0$.  Since the elements $h^{i_1} \times \cdots \times h^{i_m}$ are $F$-rational, $\mathrm{ess}$ preserves $\ratchow{X^m}$. For all non-negative integers $k_1,\hdots,k_m$, we let $D^{k_1,\hdots,k_m}$ be the $\mathbb{F}_2$-linear endomorphism of $R_{X^m}$ given by multiplication by the element $h^{k_1}\times \cdots \times h^{k_m}$. Since the latter is $F$-rational, $D^{k_1,\hdots,k_m}$ preserves $\ratchow{X^m}$. Note that if $\alpha \in R_{X^m}$ is homogeneous of degree $j$, then $D^{k_1,\hdots,k_m}(\alpha)$ is homogeneous of degree $j-\sum_{i=1}^mk_i$. If $\alpha,\beta \in R_{X^m}$, then we write $\alpha \cap \beta$ for the sum of all the standard basis elements of $R_{X^m}$ involved in both $\alpha$ and $\beta$. If $\alpha$ and $\beta$ are both essential, then the same is true of $\alpha \cap \beta$.  If $\alpha \cap \beta = 0$, then we say that $\alpha$ and $\beta$ are \emph{disjoint}. If $\alpha \cap \beta = \beta$, then we write $\beta \subset \alpha$. An $F$-rational element $\beta \in \ratchow{X^m}$ is said to be \emph{indecomposable} (or \emph{minimal}) if for any $\beta \in \ratchow{X^m}$ with $\beta \subset \alpha$, we have $\beta = \alpha$. Any such element not of the form $h^{k_1}\times \cdots \times h^{k_m}$ is essential.

We shall need two further technical tools beyond those already presented. First, let $\mathrm{mult} \colon R_{X^{2m}} \rightarrow R_{X^m}$ be the $\mathbb{F}_2$-linear map obtained by composing the inverse of the external product isomorphism $R_{X^m} \otimes_{\mathbb{F}_2} R_{X^m} \rightarrow R_{X^{2m}}$ (Lemma \ref{LEMinjectivityofexternalproduct}) and the multiplication map $R_{X^m} \otimes_{\mathbb{F}_2} R_{X^m} \rightarrow R_{X^m}$.  For any nonnegative integer $j$, we may then get a product homomorphism $\mathrm{mult} \otimes \mathrm{id} \colon R_{X^{2m+j}} \rightarrow R_{X^{m+j}}$ by identifying $R_{X^{i+j}}$ with $R_{X^i} \otimes_{\mathbb{F}_2} R_{X^j}$ for each $i \in \lbrace 2m,m \rbrace$ (Lemma \ref{LEMinjectivityofexternalproduct}). We then have the following:

\begin{proposition} \label{PROPpullbackalongdiagonal} For any nonnegative integer $j$, the map $\mathrm{mult} \otimes \mathrm{id} \colon R_{X^{2m+j}} \rightarrow R_{X^{m+j}}$ sends $\ratchow{X^{2m+j}}$ to $\ratchow{X^{m+j}}$.  
\begin{proof} The given map factors as $R_{X^{2m+j}} \xrightarrow{\pi^*} R_{X^{2(m+j)}} \xrightarrow{\mathrm{mult}} R_{X^{m+j}}$, where $\pi \colon X^{2(m+j)} \rightarrow X^{2m + j}$ is the obvious projection. To prove the claim, we may therefore assume that $j=0$. Let $U$ be the smooth locus of $X$, and let $\Delta \colon U^{m} \rightarrow U^{2m}$ be the diagonal embedding (here we are concretely identifying $U^{2m}$ with $U^m \times U^m$). Since the powers of $U$ are smooth, we have the pullback $\Delta^* \colon \mathrm{Ch}(U^{2m}) \rightarrow \mathrm{Ch}(U^m)$.  We may then consider the diagram
$$ \xymatrix{\mathrm{Ch}(U^{2m})  \ar[d] \ar[r]^{\Delta^*} & \mathrm{Ch}(U^{m}) \ar[d] \\
R_{X^{2m}} \ar[r]^{\mathrm{mult}} & R_{X^{m}}} $$
where the vertical maps are the canonical ring homomorphisms described in Proposition \ref{PROPbasicresultonrationalcycles} (2).  Since the images of these homomorphisms are $\ratchow{X^{2m}}$ and $\ratchow{X^{m}}$ respectively, it suffices to show that this diagram is commutative. For this, we are free to replace $F$ with its separable closure in $\overline{F}$ and hence assume that $\windex{\varphi} = r$ (see the proof Lemma \ref{LEMreductiontomaximalWittindex}).  In this case, the map $\theta \colon \mathrm{Ch}(U) \rightarrow R_X$ of Proposition \ref{PROPbasicresultonrationalcycles} is surjective (see \S \ref{SUBSECpartialcelldec}), and so each element of $\mathrm{Ch}(U^m)$ can be expressed as the sum of an element in the image of the external product homomorphism $\mathrm{Ch}(U^{m}) \otimes_{\mathbb{F}_2}  \mathrm{Ch}(U^{m}) \rightarrow \mathrm{Ch}(U^{2m})$ and an element in the kernel of the left-vertical map in the diagram. If $\alpha,\beta \in \mathrm{Ch}(U^m)$, then $\alpha\beta = \Delta^*(\alpha \times \beta)$ by definition. Since the left-vertical map in our diagram is a ring homomorphism, it follows from the definition of $\mathrm{mult}$ that our diagram commutes if we replace $\mathrm{Ch}(U^{2m})$ with the image of the aforementioned external product homomorphism. To prove commutativity of the diagram in general, it then suffices to show that $\Delta^*$ sends elements in the kernel of the left-vertical map to elements in the kernel of the right-vertical map. But $U^{2m}$ and $U^m$ are smooth varieties admitting a degree function (in the sense of \S \ref{SUBSECnumericaltriviality}), and the kernels in question are the numerically trivial ideals $N_0(U^{2m})$ and $N_0(U^m)$ by Lemma \ref{LEMcomputationofthekernel}. An application of Lemma \ref{LEMtrivialpushforward} (2) with $f = \Delta$ then gives the desired assertion. \end{proof} \end{proposition}

\begin{remark} When $\varphi$ is nondegenerate (so that $X$ is smooth), the proof of the proposition shows that $\mathrm{mult}$ is given by pullback along the diagonal embedding $X^{m} \rightarrow X^{2m}$. Such pullbacks play a basic role in the treatment of nondegenerate forms found in \cite[Ch. XIII]{EKM}, and the above will allow us to extend the relevant arguments to our setting. \end{remark}

Another basic tool in what follows will be extension of scalars to a field in the nondefective splitting tower of $\varphi$. Here, Proposition \ref{PROPisotropicreduction} gives the following:

\begin{proposition} \label{PROPisotropicreductionalongsplittingtower} For each $0 \leq t < h_{\mathrm{nd}}$, there exist unique $\mathbb{F}_2$-vector space homomorphisms $f_t \colon R_{X^m} \rightarrow R_{X_t^m}$ and $g_t \colon R_{X_t^m} \rightarrow R_{X_{F_t}^m}$ such that:
\begin{enumerate} \item $f_t(\ratchow{X^m}) \subseteq \ratchow{X_t^m}$ and $g_t(\ratchow{X_t^m}) \subseteq \ratchow{X_{F_t}^m}$;
\item If $\alpha = \alpha_1 \times \cdots \times \alpha_m$ is a standard basis element of $R_{X^m}$, then $f_t(\alpha) = \beta_1 \times \cdots \times \beta_m$, where for each $1 \leq s \leq m$, we have  
$$\beta_s  = \begin{cases} l_{i - \mathfrak{j}_t} & \text{if } \alpha_s = l_i \text{ for some }\mathfrak{j}_t \leq i\\
h^{i+\mathfrak{j}_t} & \text{if }\alpha_s = h^i \text{ for some } i< r - \mathfrak{j}_t\\
0 & \text{otherwise.} \end{cases} $$
\item If $\beta = \beta_1 \times \cdots \times \beta_m$ is a standard basis element of $R_{X_t^m}$, then $g_t(\beta) = \alpha_1 \times \cdots \times \alpha_m$, where for each $1 \leq s \leq m$, we have
$$\alpha_s  = \begin{cases} l_{i + \mathfrak{j}_t} & \text{if } \beta_s = l_i \text{ for some }0 \leq i < r - \mathfrak{j}_t\\
h^{i-\mathfrak{j}_t} & \text{if }\beta_s = h^i \text{ for some } i< r - \mathfrak{j}_t. \end{cases} $$ \end{enumerate} 
\begin{proof} Properties (2) and (3) characterize $f_t$ and $g_t$ uniquely, so it suffices to show that maps satisfying (1), (2) and (3) exist.  But if $\pi \colon X_{F_t} \rightarrow X$ is the canonical projection, then we can take $f_t = f \circ \pi^*$ and $g_t = g$, where $f \colon R_{X_{F_t}^m} \rightarrow R_{X_t^m}$ and $g \colon R_{X_t^m} \rightarrow R_{X_{F_t}^m}$ are the maps from Proposition \ref{PROPisotropicreduction} (see the statement of the latter). \end{proof} \end{proposition}

\subsection{Rational Cycles on $X^2$} \label{SUBSECcyclesonX^2} Our goal now is to study the $\mathbb{F}_2$-vector space $\overline{\mathrm{Ch}}_{d_X}(X^2) \subseteq (R_{X^2})_{d_X}$. When $\varphi$ is nondegenerate, this gives important structural information about the class of the quadric $X$ in the category of Chow motives with $\mathbb{F}_2$-coefficients (which remembers a great deal about $\varphi$ itself -- see \cite[Ch. XVII]{EKM}).  In \cite[Chs. XIII, XV]{EKM}, various results are obtained in the nondegenerate case using some of the basic tools of Chow theory for smooth varieties (in particular, pullbacks, intersection products and composition of correspondences). Using the results of \S \S 4--6, as well as Proposition \ref{PROPpullbackalongdiagonal}, we can now adapt some of the arguments to our more general setting. 

As observed by Karpenko in the nondegenerate case (see \cite[Ch. XIII]{EKM}), it is convenient to extend the discussion to a study of the larger space $\overline{\mathrm{Ch}}_{\geq d_X}(X^2) \subseteq (R_{X^2})_{\geq d_X}$. By Proposition \ref{PROPalgebrastructureforX^2}, the composition law introduced in \S \ref{SECcorrespondences} extends the $\mathbb{F}_2$-vector space structure on $(R_{X^2})_{\geq d_X}$ to an $\mathbb{F}_2$-algebra structure. Note also that $(R_{X^2})_{d_X}$ is an $\mathbb{F}_2$-subalgebra of this algebra, and the same is true of $\overline{\mathrm{Ch}}_{\geq d_X}(X^2)$ and $\overline{\mathrm{Ch}}_{d_X}(X^2)$ (again, see Proposition \ref{PROPalgebrastructureforX^2}). Observe now that, one exceptional case aside, the essential standard basis elements in $(R_{X^2})_{\geq d_X}$ are those of the form $h^i \times l_{i+j}$ or $l_{i+j} \times l_i$ for some integers $j \geq 0$ and $0 \leq i < r-j$. The exceptional case is where $s =0$, in which case $l_{r-1} \times l_{r-1}$ is an essential standard basis element in $(R_{X^2})_{d_X}$.  By Proposition \ref{PROPisotropyrestrictions}, however, this can only be involved in an element of $\overline{\mathrm{Ch}}_{d}(X^2)$ when $s=0$ and $\varphi$ is hyperbolic, and will therefore be effectively irrelevant for our considerations. We can also exclude additional elements of higher degree by taking the non-defective splitting pattern of $\varphi$ into account:

\begin{lemma} \label{LEMshells} Let $i$ and $j$ be integers with $j \geq 1$ and $0 \leq i < r-j$. If an element of $\overline{\mathrm{Ch}}_{>d_X}(X^2)$ involves either of the standard basis elements $h^i \times l_{i+j}$ or $l_{i+j} \times h^i$, then there exists an integer $0 \leq t <h_{\mathrm{nd}}$ such that $\mathfrak{j}_t \leq i < \mathfrak{j}_{t+1} - j$. 
\begin{proof} Since $i<r$, there is a smallest integer $0 \leq t < h_{\mathrm{nd}}$ such that $\mathfrak{j}_{t+1} >i$. Since $\windex{\varphi_{F_{t+1}}} = \mathfrak{j}_{t+1}$, Proposition \ref{PROPisotropyrestrictions} then tells us that we in fact have $\mathfrak{j}_{t+1}>i+j$, and so the claim holds. \end{proof} \end{lemma}

Let $E_{X^2}$ be the subset of $(R_{X^2})_{\geq d_X}$ consisting of all elements of the form $h^i \times l_{i+j}$ or $l_{i+j} \times h^i$ for integers $j \geq 0$ and $0 \leq i < r-j$ satisfying the restriction of Lemma \ref{LEMshells} when $j \geq 1$.  If $0 \leq t < h_{\mathrm{nd}}$, then the \emph{$t$th shell of $E_{X^2}$} is the subset consisting of the elements $h^i \times l_{i+j}$ or $l_{i+j} \times h^i$ with $\mathfrak{j}_{t-1} \leq i < \mathfrak{j}_{t} - j$. Following the treatment of nondegenerate forms in \cite[\S 73]{EKM}, the elements of $E_{X^2}$ may be assembled in a ``shell pyramid diagram'': \vspace{.5 \baselineskip}

\begin{center} \pyrdot{2}{0} \pyrdot{6}{1} \pyrdot{4}{2} \hspace{.5cm} \pyrdot{4}{2} \pyrdot{6}{1} \pyrdot{2}{0} \end{center}

\noindent The nodes in the $j$th row of the diagram (the bottom one being the $0$th row) represent of the elements of $E_{X^2}$ of the form $h^i \times l_{i+j}$ and $l_{i+j} \times h^i$.  Reading from left to right, the elements are ordered as follows: $h^0 \times l_j, h^1 \times l_{j+1},\hdots, h^{r-1} \times l_{r-1 + j}, l_{r-1 + j} \times h^{r-1},\hdots, l_j \times h^0$. If $0 \leq t < h_{\mathrm{nd}}$, then the elements of the $t$th shell of $E_{X^2}$ correspond to the elements in the $t$th pyramid from the left as well the $t$th pyramid from the right. For instance, the diagram above depicts the situation where $r = 12$ and $\varphi$ has nondefective splitting pattern $\lbrace 2,8,12 \rbrace$ (the numbers beneath index the shell numbers). For the reader familiar with the nondegenerate case, one may interpret things as follows: Suppose $\psi$ is a nondegenerate form of dimension $2r + s$ over $F$ with the same nondefective splitting pattern as $\varphi$. Then the shell pyramid diagram for $X$ is the shell pyramid diagram for $X_{\psi}$ found in \cite[\S 73]{EKM}, but with the the shells indexed by integers $\geq h_{\mathrm{nd}}$ deleted.\footnote{In \cite[\S 73]{EKM}, the authors restrict their considerations to anisotropic forms. When $\varphi$ is anisotropic, the $0$th shell of our diagram is empty, and so the discussions agree in this case.}  For non-negative integers $j,k_1,k_2$ with $k_1 + k_2 \leq j$, the operator $D^{k_1,k_2} \colon (R_{X^2})_{d_X+j} \rightarrow (R_{X^2})_{d_X+j-k_1-k_2} $ preserves the elements of $E_{X^2}$, and admits the following simple visualization in terms of its action on any given node: Shift the given node $k_1$ times along the negative diagonal containing it, and then the resulting node $k_2$ times along the positive diagonal containing that node (cf. \cite[Ex. 73.9]{EKM}). For each $0 \leq t<h_{\mathrm{nd}}$, the maps $f_t$ and $g_t$ of Proposition \ref{PROPisotropicreductionalongsplittingtower} (with $m=2$) also admit simple visualizations in this setup: The shell pyramid diagram for $X_t$ may be viewed as the diagram for $X$ with the shells numbered $0,1,\hdots,t-1$ deleted. The restriction of $f_t$ to $E_{X^2}$ (resp. $E_{X_t^2}$) may then be visualized as the obvious projection of one diagram onto the other (c.f. \cite[Ill. 73.25]{EKM}). Similarly, $g_t$ may be visualized as the obvious inclusion of the diagram for $X_t$ into the diagram for $X_{F_t}$. If $\varphi$ is not hyperbolic, then every essential element of $\overline{\mathrm{Ch}}_{\geq d_X}(X^2)$ may be visualized by shading the nodes in the diagram corresponding to the standard basis elements it involves. The determination of $\overline{\mathrm{Ch}}_{\geq d_X}(X^2)$ then amounts to determine the permissable shadings. In fact, it suffices to determine those that correspond to the indecomposable elements of $\overline{\mathrm{Ch}}_{\geq d_X}(X^2)$.  Following \cite[Ch. XIII]{EKM}, we note some simple restrictions on the possibilities here.

First, if $\alpha$ and $\beta$ are essential elements of $(R_{X^2})_{d_X + j}$ for some nonnegative integer $j$, neither of which involves $l_{r-1} \times l_{r-1}$ in the case where $s = j = 0$, then a quick computation using Lemma \ref{LEMcompositionformulas} reveals that the cycle $\alpha \cap \beta$ coincides with $\mathrm{ess}(D^{j,0} \circ \alpha)$. This gives:

\begin{lemma} \label{LEMintersections} Suppose that $\alpha, \beta \in \overline{\mathrm{Ch}}_{d_X + j}(X^2)$ for some nonnegative integer $j$.  Then:
\begin{enumerate} \item $\alpha \cap \beta \in \overline{\mathrm{Ch}}_{d_X + j}(X^2)$;
\item If $\alpha$ is indecomposable and $\alpha \cap \beta \neq 0$, then $\alpha \subset \beta$. \end{enumerate} 
\begin{proof} (1) If $\varphi$ is hyperbolic, there is nothing to show. If not, then neither $\alpha$ nor $\beta$ involves $l_{r-1} \times l_{r-1}$ in the case where $s = j = 0$, and so the claim holds by the remarks above.
(2) If $\alpha$ is indecomposable, then (1) implies that either $\alpha \cap \beta = 0$ or $\alpha \cap \beta = \alpha$. \end{proof} \end{lemma}

The computation underlying the proof of the previous lemma also reveals that for any nonegative integer $j$, the map $D^{j,0} \colon \overline{\mathrm{Ch}}_{d_X+j}(X^2) \rightarrow \overline{\mathrm{Ch}}_{d_X}(X^2)$ sends indecomposable elements to indecomposable elements. Since $D^{k,0} = (D^{1,0})^{\circ k}$ for any positive integer $k$, the same is then true of $D^{k,0} \colon \overline{\mathrm{Ch}}_{d_X+j}(X^2) \rightarrow \overline{\mathrm{Ch}}_{d_X + j-k}(X^2)$ for any $1 \leq k \leq j$. Taking into account the factor exchange automorphism of $X^2$, we also get the same assertion for the maps $D^{0,k} \colon \overline{\mathrm{Ch}}_{d_X+j}(X^2) \rightarrow \overline{\mathrm{Ch}}_{d_X + j-k}(X^2)$. Combining, we then get:

\begin{lemma} \label{LEMderivatives} Let $j,k_1,k_2$ be nonnegative integers with $k_1 + k_2 \leq j$. Then the $\mathbb{F}_2$-linear map $D^{k_1,k_2} \colon \overline{\mathrm{Ch}}_{d+j}(X\times X) \rightarrow \overline{\mathrm{Ch}}_{d + j - k_1 - k_2}(X^2)$ is injective and sends indecomposable elements to indecomposable elements. 
\begin{proof} Injectivity is clear from the definition, and since $D^{k_1,k_2} = D^{k_1,0} \circ D^{0,k_2}$, the second statement holds by the remarks above. \end{proof} \end{lemma}

\subsection{The MDT Invariant} \label{SUBSECMDTInvariant} We now introduce a discrete invariant of $\varphi$ that reconstructs the space $\overline{\mathrm{Ch}}_{d_X}(X^2)$. For each integer $i$, we introduce formal symbols $i_{\mathrm{lo}}$ and $i^{\mathrm{up}}$. For any integer $j$, there is a shifting map on the set consisting of these symbols that sends $i_{\mathrm{lo}}$ to $i_{\mathrm{lo}}[j]: = (i+j)_{\mathrm{lo}}$ and $i^{\mathrm{up}}$ to $i^{\mathrm{up}}[j]: = (i+j)^{\mathrm{up}}$. If $\Lambda$ is a subset of this set, then we write $\Lambda[j]$ for the image of $\Lambda$ under this shifting map. 

Let $\Lambda(X)$ be the set consisting of the symbols $i_{\mathrm{lo}}$ with $0 \leq i <r$ and $i^{\mathrm{up}}$ with $d_X - r<i \leq d_X$. If $\Lambda$ is a subset of $\Lambda(X)$, then we write $\Lambda_{\mathrm{lo}}$ (resp. $\Lambda^{\mathrm{up}}$) for the subset of $\Lambda$ consisting of the elements of the form $i_{\mathrm{lo}}$ (resp. $i^{\mathrm{up}}$).  If $\Lambda_{\mathrm{lo}} \neq \emptyset$ (resp. $\Lambda^{\mathrm{up}} \neq \emptyset$), then we set $a(\Lambda): = \mathrm{min} \lbrace i\;|\; i_{\mathrm{lo}} \in \Lambda \rbrace$ (resp. $b(\Lambda): = \mathrm{max}\lbrace i\;|\; i^{\mathrm{up}} \in \Lambda \rbrace$). For each integer $0 \leq i <r$, we set $\alpha_{i_{\mathrm{lo}}}: = h^i \times l_i \in E_{X^2}$.  For each integer $d_X - r< i \leq d_X$, we set $\alpha_{i^{\mathrm{up}}} : = l_{d_X - i} \times h^{d_X - i} \in E_{X^2}$. Using Lemma \ref{LEMcompositionformulas}, one readily observes that these elements are idempotent in $(R_{X^2})_{d_X}$. Moreover, one exceptional case aside, they are pairwise mutually orthogonal. The exceptional case is that where $s = 0$ and $\mydim{\varphi} \equiv 2 \pmod{4}$, as there we have $(h^{r-1} \times l_{r-1}) \circ (l_{r-1} \circ h^{r-1}) = h^{r-1} \times h^{r-1}$.  If $\Lambda$ is a subset of $\Lambda(X)$, then we write $\alpha_{\Lambda}$ for the essential cycle $\sum_{\lambda \in \Lambda}\alpha_{\lambda} \in (R_{X^2})_{d_X}$. If $\varphi$ is not hyperbolic, then every element of $\overline{\mathrm{Ch}}_{d_X}(X^2)$ is of this form. In this case, we write $\Lambda(\alpha)$ for the subset of $\Lambda(X)$ corresponding to a given element $\alpha \in \overline{\mathrm{Ch}}_{d_X}(X^2)$.

\begin{remark} \label{REMisotropicreductionforMDT} If $0 \leq t < h_{\mathrm{nd}}$, then $d_{X_t} = d_X - 2\mathfrak{j}_t$, and one readily checks that the maps $f_t$ and $g_t$ from Proposition \ref{PROPisotropicreductionalongsplittingtower} (for $m=2$) have the following properties:
\begin{itemize} \item If $\Lambda \subseteq \Lambda(X)$, then $f(\alpha_{\Lambda}) = \alpha_{\Lambda[-\mathfrak{j}_t] \cap \Lambda(X_t)}$;
\item If $\Lambda \subseteq \Lambda(X_t)$, then $g_t(\alpha_{\Lambda}) = \alpha_{\Lambda[\mathfrak{j}_t]}$.
\end{itemize} \end{remark}

Now, Lemma \ref{LEMclassofthediagonal} yields the following:

\begin{lemma} \label{LEMrationalityoffdiagonalelement} $\alpha_{\Lambda(X)} \in \overline{\mathrm{Ch}}_{d_X}(X^2)$. 
\begin{proof} The claim is that $\sum_{i=0}^{r-1} (h^i \times l_i + l_i \times h^i)$ is $F$-rational. If $\windex{\varphi} = r$, this is clear.  If not, then Proposition \ref{PROPisotropicreduction} allows us to reduce to the case where $\varphi$ is anisotropic, and here the claim holds by Lemma \ref{LEMclassofthediagonal}. \end{proof} \end{lemma}

By Lemma \ref{LEMintersections}, it follows that $\Lambda(X)$ admits a unique partition into disjoint subsets $\Lambda_1,\hdots,\Lambda_t$ such that each of the cycles $\alpha_{\Lambda_i}$ is an indecomposable element of $\overline{\mathrm{Ch}}_{d_X}(X^2)$. The set consisting of the subsets $\Lambda_1,\hdots,\Lambda_t$ shall be denoted $\mathrm{MDT}(\varphi)$, and its elements shall be called the \emph{connected components} of $\Lambda(X)$. We shall also say that two elements of $\Lambda(X)$ belonging to the same connected component are \emph{connected}. By a \emph{summand} of $\Lambda(X)$, we shall mean a union of the connected components. Our notation and terminology is informed by the case where $\varphi$ is nondegenerate, where we have the following: 

\begin{remarks}[$\mathrm{MDT}$ in the nondegenerate case] \label{REMSMDT} Suppose that $\varphi$ is nondegenerate (so that $X$ is smooth).
\begin{enumerate}[leftmargin=*] \item The invariant $\mathrm{MDT}(\varphi)$ has the following interpretation: For a field $k$, let $\mathrm{Chow}(k,\mathbb{F}_2)$ be the category of Chow motives over $F$ with $\mathbb{F}_2$-coefficients. By a result (essentially) due to Vishik, the motive of any smooth projective quadric in this category admits an essentially unique decomposition as a direct sum of indecomposable objects. Over $\overline{F}$,  the idempotents $\alpha_{\lambda} \in \mathrm{Ch}_{d_X}(\overline{X^2}) = \mathrm{End}_{\mathrm{Chow}(\overline{F},\mathbb{F}_2)}(M(\overline{X}))$ give rise to a decomposition of $M(\overline{X})$ as a direct sum of Tate motives indexed by the elements of $\Lambda(X)$.  Via scalar extension, the complete decomposition of $M(X)$ in $\mathrm{Chow}(F,\mathbb{F}_2)$ yields a partition of these Tate motives, and the corresponding partition of $\Lambda(X)$ is exactly $\mathrm{MDT}(\varphi)$ (see \cite[Ch. XVII]{EKM} for more details). For this reason, $\mathrm{MDT}(\varphi)$ is sometimes referred to in this setting as the \emph{motivic decomposition type} of $\varphi$. Although we do not have a definitive motivic framework within which we can discuss the case of degenerate forms, we nevertheless adopt the $\mathrm{MDT}$ notation for the sake of consistency.
\item In one respect, our notation differs slightly from what can be found elsewhere in the literature in this case. In particular, in \cite{Vishik2}, $\Lambda(X)$ is (essentially) taken to be the set consisting of the symbols $i_{\mathrm{lo}}$ and $i^{\mathrm{up}}$ with $0 \leq i < r$, with $\alpha_{i^{\mathrm{up}}}$ being now the cycle $l_i \times h^i$. As indicated above, our notation is rather in line with indexing of Tate motives. This allows for cleaner use of the shift notation introduced above in later statements.  \end{enumerate} \end{remarks}

Isotropic reduction gives the following:

\begin{lemma} \label{LEMMDTforisotropicforms}If $\windex{\varphi} <r$, then $\mathrm{MDT}(\varphi)$ consists of the following sets:
\begin{itemize} \item $\lbrace i_{\mathrm{lo}} \rbrace$ with $0 \leq i < \windex{\varphi}$;
\item $\lbrace i^{\mathrm{up}} \rbrace$ with $d_X - \windex{\varphi} < i \leq d_X$;
\item $\Lambda[\windex{\varphi}]$ with $\Lambda$ an element of $\mathrm{MDT}(\anispart{\varphi})$.
\end{itemize}
\begin{proof} If $0 \leq i < \windex{\varphi}$, then $h^i \times l_i$ and $l_i \times h^i$ are clearly indecomposable elements of $\overline{\mathrm{Ch}}_{d_X}(X^2)$, and so the sets in the first two points are connected components of $\Lambda(X)$.  Now, if $\Lambda$ is an element of $\mathrm{MDT}(\anispart{\varphi})$, then $g_0(\alpha_{\Lambda}) = \alpha_{\Lambda[\windex{\varphi}]}$, where $g_0 \colon R_{X^2} \rightarrow R_{X_0^2}$ is the map from Proposition \ref{PROPisotropicreductionalongsplittingtower} (Remark \ref{REMisotropicreductionforMDT}). But it is clear from the statement of Proposition \ref{PROPisotropicreductionalongsplittingtower} that $g_0$ sends indecomposable elements of $\overline{\mathrm{Ch}}_{d_{X_0}}(X_0^2)$ to indecomposable elements of $\overline{\mathrm{Ch}}_{d_X}(X^2)$, and so the listed sets are all the connected components of $\mathrm{MDT}(\varphi)$. 
\end{proof} \end{lemma}

Using Proposition \ref{PROPpullbackalongdiagonal}, we then get:

\begin{lemma} \label{LEMevennumberofTates} The following are equivalent:
\begin{enumerate} \item $\varphi$ is isotropic;
\item $\lbrace 0_{\mathrm{lo}} \rbrace$ is an element of $\mathrm{MDT}(\varphi)$;
\item There is an odd-cardinality element of $\mathrm{MDT}(\varphi)$;
\item There is an odd-cardinality subset $\Lambda$ of $\Lambda(X)$ with $\alpha_{\Lambda} \in \overline{\mathrm{Ch}}_d(X^2)$.  \end{enumerate}
\begin{proof} $(1)\Rightarrow (2)$: Apply Lemma \ref{LEMMDTforisotropicforms}.

$(2)\Rightarrow (3) \Rightarrow (4)$: Clear.

$(4) \Rightarrow (1)$: Consider the map $\mathrm{mult} \colon R_{X^2} \rightarrow R_X$ of Proposition \ref{PROPpullbackalongdiagonal}. By definition, we have $\mathrm{mult}(\alpha_{\lambda}) = l_0$ for all $\lambda \in \Lambda(X)$. Since $\mathrm{mult}$ preserves $F$-rationality (Proposition \ref{PROPpullbackalongdiagonal}), it follows that $l_0 = |\Lambda|l_0 = \mathrm{mult}(\sum_{\lambda \in \Lambda}\alpha_{\lambda})$ is $F$-rational, and so $X$ admits a $0$-cycle of degree $1$. But Springer's theorem then tells us that $\varphi$ is isotropic, as desired.  \end{proof}
\end{lemma}

Before proceeding to the anisotropic case, we make one further basic observation:

\begin{lemma} \label{LEMMDTunderpurelytranscendental} If $K/F$ is a purely transcendental extension, then $\mathrm{MDT}(\varphi_K) = \mathrm{MDT}(\varphi)$. 
\begin{proof} We claim that the scalar extension map $\mathrm{Ch}_{d_X}(X^2) \rightarrow \mathrm{Ch}_{d_X}(X_K^2)$ is surjective. For this, it suffices to consider the case where $K = F(T)$ for a single variable $T$. In this case, the projection $X_K^2 \rightarrow X^2$ factors as the composition $X_K^2 = \mathrm{Spec}(K) \times X^2 \rightarrow \mathbb{A}^1 \times X^2 \rightarrow X^2$, where the second map is the canonical projection, and the first is given by inclusion of the generic point of $\mathbb{A}^1$ on the first factor and the identity on the second. The pullback along the second map is surjective by homotopy invariance, and the pullback along the second is surjective by \cite[Cor. 57.11]{EKM}, so the claim holds. In particular, the scalar extension map $\overline{\mathrm{Ch}}_{d_X}(X^2) \rightarrow \overline{\mathrm{Ch}}_{d_X}(X_K^2)$ is an isomorphism, and so $\mathrm{MDT}(\varphi_K) = \mathrm{MDT}(\varphi)$.  \end{proof}
\end{lemma}

\subsection{Basic Connections Arising From the Nondefective Splitting Pattern} In view of Lemma \ref{LEMMDTforisotropicforms}, we shall now assume for the remainder of this section that $\varphi$ is anisotropic.  The following extends \cite[Lem. 73.19]{EKM} to our setting:

\begin{proposition} \label{PROPsplittingpatternconnections} Let $1 \leq t \leq h_{\mathrm{nd}}$. For any integer $0 \leq i < \mathfrak{i}_t - 1$, the elements $(\mathfrak{j}_{t-1} + i)_{\mathrm{lo}}$ and $(d_X - (\mathfrak{j}_t - 1 - i))^{\mathrm{up}}$ are then connected in $\Lambda(X)$. 
\begin{proof} With what has been established thus far, the proof of \cite[Lem. 73.19]{EKM} goes through. Briefly, one readily reduces to the case where $t = 1$ by using the map $f_{t-1} \colon R_{X^2} \rightarrow R_{X_{t-1}^2}$ of Proposition \ref{PROPisotropicreductionalongsplittingtower} (note that $\mathfrak{j}_1(\varphi_{t-1}) = \mathfrak{j}_t - \mathfrak{j}_{t-1}$). For $t = 1$, let $\mu \colon X_{F_1} = \mathrm{Spec}(F_1) \times X \rightarrow X^2$ be the map given by inclusion of the generic point of $X$ on the first factor and the identity on the second. By \cite[Cor. 57.11]{EKM}, $\mu^* \colon \overline{\mathrm{Ch}}(X^2) \rightarrow \overline{\mathrm{Ch}}(X_{F_1})$ is surjective. Since $\windex{\varphi_{F_1}} = \mathfrak{j}_1$, the target contains $l_{\mathfrak{j}_1-1}$. By the definition of $\mu$, it follows that there exists an indecomposable cycle $\beta \in \overline{\mathrm{Ch}}_{d_X + \mathfrak{j}_1-1}(X^2)$ involving $h^0 \times l_{\mathrm{j}_1-1}$. Consider $\alpha: = D^{i,\mathfrak{j}_1-i-1}(\beta) \in \overline{\mathrm{Ch}}_{d_X}(X^2)$. By Lemma \ref{LEMderivatives}, $\alpha$ is indecomposable. Now by Lemma \ref{LEMshells}, the only elements in the $1$st shell of $E_{X^2}$ that have degree $d_X + \mathfrak{j}_1-1$ are $h^0 \times l_{\mathfrak{j}_1-1}$ and $l_{\mathfrak{j}_1-1} \times h^0$. As a result, $\alpha = (h^i \times l_i) + a(l_{\mathfrak{j}_1-1-i} \times h^{\mathfrak{j}_1-1-i}) + \beta$ for some $a \in \mathbb{F}_2$ and some $\beta \in (R_{X^2})_{d_X}$ that involves no elements in the $1$st shell of $E_{X^2}$.  Now since $\varphi$ is anisotropic, $|\Lambda(\alpha)|$ is even by Lemma \ref{LEMevennumberofTates}. At the same time, the same is also true of $|\Lambda(\beta)|$. Indeed, if $\beta = 0$, this is clear; if not, then $h_{\mathrm{nd}} >1$, and an application of Lemma \ref{LEMevennumberofTates} to $f_1(\alpha) \in \overline{\mathrm{Ch}}_{d_{X_1}}(X_1^2)$ gives the claim. As a result, we must have that $a=1$, and so both $i_{\mathrm{lo}}$ and $(d_X - (\mathfrak{j}_1-1-i))^{\mathrm{up}}$ lie in $\Lambda(\alpha)$, which is a connected component of $\Lambda(X)$. \end{proof} \end{proposition} 

Thus, if $\Lambda$ is an summand of $\Lambda(X)$, then each element of $\Lambda_{\mathrm{lo}}$ admits a ``dual'' element of $\Lambda^{\mathrm{up}}$ determined by the nondefective splitting pattern of $\varphi$. In particular, if the latter is known, then $\Lambda$ is determined by $\Lambda_{\mathrm{lo}}$. Moreover, the integers $a(\Lambda)$ and $b(\Lambda)$ are both defined, and if $\mathfrak{j}_{t-1} \leq a(\Lambda) < \mathfrak{j}_t$ with $1 \leq t \leq h_{\mathrm{nd}}$, then $d_X - \mathfrak{j}_t < b(\Lambda) \leq d_X - \mathfrak{j}_{t-1}$. We can make this more precise with the following, the first assertion of which extends \cite[Prop. 73.23]{EKM} to our setting:

\begin{proposition} \label{PROPsummandshifts} Let $1 \leq t < h_{\mathrm{nd}}$. If $\mathrm{MDT}(\varphi)$ admits an element $\Lambda'$ with $\mathfrak{j}_{t-1} \leq a(\Lambda') < \mathfrak{j}_t$, then it admits an element $\Lambda$ with $a(\Lambda) = \mathfrak{j}_{t-1}$. Moreover:
\begin{enumerate} \item The sets $\Lambda,\Lambda[1],\hdots,\Lambda[\mathfrak{i}_t -1]$ are all elements of $\mathrm{MDT}(\varphi)$, and $\Lambda' = \Lambda[a(\Lambda') - \mathfrak{j}_{t-1}]$;
\item Fix an integer $t \leq t' \leq h_{\mathrm{nd}}$. If $i_{\mathrm{lo}} \in \Lambda$ for some $\mathfrak{j}_{t'-1} \leq i < \mathfrak{j}_{t'}$, then $i + \mathfrak{i}_{t} \leq \mathfrak{j}_{t'}$. In particular, $\mathfrak{i}_{t'} \geq \mathfrak{i}_t$. \end{enumerate} 
\begin{proof} Let $0 \leq j  < \mathfrak{i}_t$ be such that $a(\Lambda') = \mathfrak{j}_{t-1} + j$.  By the definition of $a(\Lambda')$ and the remarks preceding the statement, $\alpha_{\Lambda}$ involves no element of $E_{X^2}$ in the shells numbered $1,\hdots,t-1$. The key step is proving: \vspace{.5 \baselineskip}

\noindent {\bf Claim.} There is an indecomposable element $\alpha \in \overline{\mathrm{Ch}}_{d_X + \mathfrak{i}_t - 1 - j}(X^2)$ that involves $h^{\mathfrak{j}_{t-1} + j} \times l_{\mathfrak{j}_t - 1}$, but no element of $E_{X^2}$ in the shells numbered $1,\hdots,t-1$.  \vspace{.5 \baselineskip}

Before remarking on this, let us show how it gives the desired conclusion. Let $\alpha$ be as in the claim. By Lemma \ref{LEMderivatives}, $D^{\mathfrak{i}_t-1-j,0}(\beta)$ is an indecomposable element of $\overline{\mathrm{Ch}}_{d_X}(X^2)$. By construction, it involves $h^{\mathfrak{j}_t-1} \times l_{\mathfrak{j}_t-1}$, but no element of $E_{X^2}$ in the shells numbered $1,\hdots,t-1$.  By Proposition \ref{PROPsplittingpatternconnections}, it also involves $l_{\mathfrak{j}_{t-1}} \times h^{\mathfrak{j}_{t-1}}$.  Letting $\gamma$ be its image under pushforward along the factor exchange automorphism of $X^2$, we then get that $\Lambda: = \Lambda(\gamma)$ is an element of $\mathrm{MDT}(\varphi)$ with $a(\Lambda) = \mathfrak{j}_{t-1}$.  The claim then holds with $j = 0$, and so we have an indecomposable element $\beta \in \overline{\mathrm{Ch}}_{d_X + \mathfrak{i}_t-1}(X^2)$ that involves $h^{\mathfrak{j}_{t-1}} \times l_{\mathfrak{j}_t-1}$, but no element of $E_{X^2}$ in the shells numbers $1,\hdots,t-1$.  For each $0 \leq k < \mathfrak{i}_t$, set $\beta_k: = D^{k,\mathfrak{i}_t-1-k}(\beta)$. By Lemma \ref{LEMderivatives}, $\beta_k$ is an indecomposable element of $\overline{\mathrm{Ch}}_{d_X}(X^2)$ involving no elements of $E_{X^2}$ in the shells numbered $1,\hdots,t-1$. Since $\beta_0$ involves $h^{\mathfrak{j}_{t-1}} \times l_{\mathfrak{j}_t-1}$, we have $\Lambda(\beta_0) = \Lambda$, and then $\Lambda(\beta_k) = \Lambda[k]$ for all $1 \leq k < \mathfrak{i}_t-1$ by construction. Thus, $\Lambda,\Lambda[1],\hdots,\Lambda[\mathfrak{i}_t-1]$ are elements of $\mathrm{MDT}(\varphi)$,  and since $a(\Lambda') = a(\Lambda[j]) = a(\Lambda[a(\Lambda') - \mathfrak{j}_{t-1}])$, we have $\Lambda' = \Lambda[a(\Lambda') - \mathfrak{j}_{t-1}]$, proving (1). At the same time, since $\beta_0 = D^{0,\mathfrak{i}_t-1}(\beta)$, it involves $h^i \times l_i$ if and only if $\beta$ involves $h^i \times l_{i+\mathfrak{i}_t-1}$. Thus,  if $i_{\mathrm{lo}} \in \Lambda$ for some integers $t \leq t' \leq h_{\mathrm{nd}}$ and $\mathfrak{j}_{t'-1} \leq i < \mathfrak{j}_{t'}$, then $\beta$ must involve $h^i \times l_{i + \mathfrak{i}_t - 1}$, and so $i + \mathfrak{i}_t-1 < \mathfrak{j}_{t'}$ by Lemma \ref{LEMshells}, proving (2). 

As for the claim, the case where $t=1$ is already implicit in the proof of Proposition \ref{PROPsplittingpatternconnections}. When $t>1$, an explicit construction of $\alpha$ from $\alpha_{\Lambda'}$ in the nondegenerate case is given in \cite[Proof of Prop. 73.23]{EKM} using scalar extension to the function field of $X$, composition of correspondences between powers of $X$ and pullback along a partial diagonal embedding $X^2 \rightarrow X^3$ (in \cite[Ill. 73.24]{EKM}, this is the construction allowing one to move from position 2 in the diagram to position 3).  With the results of \S 5, together with Propositions \ref{PROPisotropicreductionalongsplittingtower} and \ref{PROPpullbackalongdiagonal} (the latter being our substitute for pullback along partial diagonals), the very same construction also gives the desired element in our more general setting. \end{proof} \end{proposition}

In view of the remarks preceding the statement of the proposition, we get:

\begin{corollary} \label{CORdimensionofsummand} Let $\Lambda$ be an element of $\mathrm{MDT}(\varphi)$, and let $1 \leq t \leq h_{\mathrm{nd}}$ be the unique integer with $\mathfrak{j}_{t-1} \leq a(\Lambda) < \mathfrak{j}_t$. Then $b(\Lambda) = d_X - (a(\Lambda) + (\mathfrak{i}_t - 1))$. \end{corollary}

The element of $\mathrm{MDT}(\varphi)$ that contains $0_{\mathrm{lo}}$ shall be denoted $\Lambda^U(X)$. We refer to it as the \emph{upper connected component} of $\Lambda(X)$.\footnote{When $\varphi$ is nondegenerate, this is the connected component of $\Lambda(X)$ corresponding to what in the literature is called the \emph{upper motive} of $X$ -- see Remarks \ref{REMSMDT}.} This set has the following properties:

\begin{proposition} \label{PROPuppersummand} \;
\begin{enumerate} \item $a(\Lambda^U(X)) = 0$ and $b(\Lambda^U(X)) = \Izhdim{\varphi} -1$. 
\item If $|\Lambda^U(X)| \neq 2$, then $|\Lambda^U(X)_{\mathrm{lo}}| \geq 2$, and $\mathrm{min}\lbrace i>0 \;|\ i_{\mathrm{lo}} \in \Lambda^U(X) \rbrace = \mathfrak{j}_t$ for some integer $1 \leq t < h_{\mathrm{nd}}$. \end{enumerate}
\begin{proof} (1) Immediate from Corollary \ref{CORdimensionofsummand} and the definition. 

(2) If $|\Lambda^U(X)| \neq 2$, then $|\Lambda^U(X)|_{\mathrm{lo}} \geq 2$ by Proposition \ref{PROPsplittingpatternconnections}. Let $1 \leq t < h_{\mathrm{nd}}$ and $0 \leq j < \mathfrak{i}_{t+1}$ be such that $\mathfrak{j}_t + j = \mathrm{min}\lbrace i>0 \;|\ i_{\mathrm{lo}} \in \Lambda^U(X) \rbrace$. We have to show that $j = 0$. Suppose otherwise.  By Proposition \ref{PROPsummandshifts} (2), we have $\mathfrak{j}_t + \mathfrak{i}_1 -1 < \mathfrak{j}_{t+1}$. We claim that there exists an indecomposable element $\alpha \in \overline{\mathrm{Ch}}_{d_X + \mathfrak{i}_1-1}(X^2)$ that involves $h^{\mathfrak{j}_t} \times l_{\mathfrak{j}_t + \mathfrak{i}_1 - 1}$, but no element of $E_{X^2}$ in the shells numbered $1,\hdots,t-1$.  Given this, Lemma \ref{LEMderivatives} gives that $\beta: = D^{0,\mathfrak{i}_1-1}$ is an indecomposable element of $\overline{\mathrm{Ch}}_{d_X}(X^2)$ with $a(\Lambda(\beta)) = \mathfrak{j}_{t}$. By Proposition \ref{PROPsummandshifts} (1), however, $\Lambda(\beta)[j]$ is then an element of $\mathrm{MDT}(\varphi)$ with $a(\Lambda(\beta)[j]) = \mathfrak{j}_t + j$, contradicting the fact that $\mathfrak{j}_t + j \in \Lambda^U(X)$.  We must therefore have that $j=0$. For the claim, an explicit construction of $\alpha$ from $\alpha_{\Lambda^U(X)}$ in the nondegenerate case is given in \cite[Proof of Prop. 73.27]{EKM} using scalar extension to the function field of $X$, composition of correspondences between powers of $X$ and pullback along a partial diagonal embedding $X^2 \rightarrow X^3$. With the results of \S 5, together with Propositions \ref{PROPisotropicreductionalongsplittingtower} and \ref{PROPpullbackalongdiagonal}, the very same construction also yields the desired element in our setting.\end{proof} \end{proposition}

Recall (Theorem \ref{THMstb}) that $\Izhdim{\varphi}$ is a stable birational invariant of $X$ (among the class of quadrics we are considering). We shall see shortly that the same is true of $\Lambda^U(X)$.

\subsection{Example: Forms of Nondefective Height 1} \label{SUBSECnondefectiveheight1} We consider here the simplest possible case, namely that where $h_{\mathrm{nd}} = 1$. This may be described as follows:

\begin{lemma} \label{LEMnondefectiveheight1}The following are equivalent:
\begin{enumerate} \item $h_{\mathrm{nd}} = 1$;
\item $\witti{1}{\varphi} = r$;
\item $\varphi_1 \simeq \mathrm{ql}(\varphi)_{F(\varphi)}$. \end{enumerate}
\begin{proof} $(1) \Leftrightarrow (2)$: By definition.

$(2) \Leftrightarrow (3)$:  Since $\varphi$ remains nondefective over $F(\varphi)$, we have $\varphi_1 = \anispart{(\varphi_{F(\varphi)})} \cong \tau \perp \mathrm{ql}(\varphi)_{F(\varphi)}$ for some nondegenerate form $\tau$ of dimension $2r - 2\witti{1}{\varphi}$ over $F(\varphi)$. Thus, $\witti{1}{\varphi} = r$ if and only if $\varphi_1 \simeq \mathrm{ql}(\varphi)_{F(\varphi)}$.  \end{proof} \end{lemma}

When these conditions are satisfied, the results of the previous section immediately give the following ``binary decomposition'' of $\Lambda(X)$:

\begin{proposition} \label{PROPMDTfornondefectiveheight1} If $h_{\mathrm{nd}} = 1$, then $\mathrm{MDT}(\varphi)$ consists of the sets $\Lambda^U(X)[i]$ with $0 \leq i < r$, and we have $\Lambda^U(X) = \lbrace 0_{\mathrm{lo}}, (r+s-1)^{\mathrm{up}} \rbrace$.
\begin{proof} Since $h_{\mathrm{nd}} =1$, we have $\witti{1}{\varphi} = r$. In particular, $\Izhdim{\varphi} = \mydim{\varphi} - \witti{1}{\varphi} = (2r+s) - r = r + s$. By Proposition (1), it follows that $\Lambda^U(X)$ contains the elements $0_{\mathrm{lo}}$ and $(r+s-1)^{\mathrm{up}}$. On the other hand, Proposition \ref{PROPsummandshifts} (1) tells us that each of the sets $\Lambda^U(X)[i]$ with $0 \leq i < r$ is an element of $\mathrm{MDT}(\varphi)$. Since $|\Lambda(X)| = 2r$, we conclude that these are the only elements of $\mathrm{MDT}(\varphi)$, and that $\Lambda^U(X) = \lbrace 0_{\mathrm{lo}}, (r+s-1)^{\mathrm{up}} \rbrace$. 
\end{proof} \end{proposition}

If $r = 1$, then we clearly have that $h_{\mathrm{nd}} = 1$. For larger values of $r$, examples may be produced with the following:

\begin{lemma} \label{LEMclosePfisterneighbours} Write $\mydim{\varphi} = 2^n + m$ for integers $n \geq 0$ and $1 \leq m \leq 2^n$. If $\varphi$ is a Pfister neighbour, then $r+s \leq 2^n$, and the following are equivalent:
\begin{enumerate} \item $h_{\mathrm{nd}} = 1$;
\item $r+s = 2^n$;
\item $\mydim{\varphi} = 2^{n+1} - s$. \end{enumerate}
\begin{proof} The first statement holds by Lemma \ref{LEMPneighbours} (1). By part (3) of the same lemma, we have $\witti{1}{\varphi} = \mydim{\varphi} - 2^n = 2r + s - 2^n$. In particular, $h_{\mathrm{nd}} = 1$ if and only if $2r + s - 2^n = r$, i.e., $r + s = 2^n$.  This shows the equivalence of (1) and (2), and the equivalence of (2) and (3) is clear since $\mydim{\varphi} = 2r + s = 2(r+s)-s$.  \end{proof} \end{lemma}

The forms satisfying the conditions of the lemma are said to be \emph{close Pfister neighbours}.  The nondegenerate close Pfister neighbours are the general Pfister forms and their codimension-1 subforms.  It is expected that most anisotropic forms of nondefective height 1 are close Pfister neighbours - see \S \ref{SECPfisterneighbourconjecture} below for more precise statements. 

\begin{remark}  \label{REMPfister} If $\varphi$ is a general $(n+1)$-fold Pfister form (resp. a codimension-$1$ subform of such a form), then the preceding discussion shows that $\Lambda^U(X) = \lbrace 0_{\mathrm{lo}}, (2^n - 1)^{\mathrm{up}} \rbrace$, and that $\mathrm{MDT}(\varphi)$ consists of the shifts $\Lambda^U(X)[i]$ with $0 \leq i < 2^n$ (resp. $0 \leq i < 2^n - 1$). The deeper point underlying this is that the motive of $X$ in $\mathrm{Chow}(F,\mathbb{F}_2)$ decomposes are a direct sum of shifts of the (binary) Rost motive attached to $\varphi$ -- see \cite[Ex. 94.3]{EKM}. \end{remark}

\subsection{Further Restrictions Arising From the Steenrod Operations} Finer restrictions on $\mathrm{MDT}(\varphi)$ may be obtained using the action of the cohomological-type Steenrod operations on $\overline{\mathrm{Ch}}(X^2)$ (see Proposition \ref{PROPbasicresultonrationalcycles} (3)).  For a positive integer $x$, we write $v_2(x)$ for the $2$-adic valuation of $x$. The following, which refines part (2) of Proposition \ref{PROPuppersummand}, extends a result of Karpenko (\cite[Prop. 83.2, Thm. 83.3, Cor. 83.4]{EKM}) to our setting:

\begin{theorem} \label{THMKarpenkouppersummand} Let $v$ be the smallest integer for which $\mathfrak{i}_1 \leq 2^v$.  If $|\Lambda^U(X)|>2$, then:
\begin{enumerate} \item Each positive integer $i$ with $i_{\mathrm{lo}} \in \Lambda^U(X)$ is divisible by $2^v$;
\item If $i$ the smallest positive integer for which $i_{\mathrm{lo}} \in \Lambda^U(X)$, then $i = \mathfrak{j}_{t}$ for some integer $1 \leq t < h_{\mathrm{nd}}$ with $v_2(\mathfrak{i}_{t+1}) \geq v_2(\mathfrak{i}_1)$.  In particular,
$$v_2(\mathfrak{i}_1) \leq \mathrm{max}\lbrace v_2(\mathfrak{i}_2),\hdots,v_2(\mathfrak{i}_{h_{\mathrm{nd}}}) \rbrace.  $$ \end{enumerate}
\begin{proof} (1) Let $0 \leq i <r$ be an integer with $i_{\mathrm{lo}} \in \Lambda^U(X)$. We show that $i$ is divisible by $2^v$ by induction on $i$. If $i = 0$, there is nothing to show. Suppose now that $i>0$, and that each integer $j<i$ with $j_{\mathrm{lo}} \in \Lambda^U(X)$ is divisible by $2^v$.  Set $w: = v_2(i)$. We have to show that $w\geq v$. Suppose, for the sake of contradiction, that $w<v$. By the definition of $v$, we then have that $2^w < \mathfrak{i}_1$.  In particular, we may consider the map
$$ f: = D^{0,\mathfrak{i}_1 -1 - 2^w} \circ S^{2^w} \colon \overline{\mathrm{Ch}}_{d + \mathfrak{i}_1-1}(X^2) \rightarrow \overline{\mathrm{Ch}}_d(X^2). $$
Now, by the proof of Proposition \ref{PROPsplittingpatternconnections}, there is an indecomposable element $\pi \in \overline{\mathrm{Ch}}_{d + \mathfrak{i}_1 - 1}(X^2)$ that involves $h^0 \times l_{\mathfrak{i}_1-1}$.  By Lemma \ref{LEMderivatives}, $D^{0,\mathfrak{i}_1-1}(\pi)$ is then an indecomposable element of $\overline{\mathrm{Ch}}_d(X^2)$ that involves $h^0 \times l_0$. In particular, for any integer $j$, we have $j_{\mathrm{lo}} \in \Lambda^U(X)$ (resp. $j^{\mathrm{up}} \in \Lambda^U(X)$) if and only if $\pi$ involves $h^j \times l_{\mathfrak{i}_1 - 1+j}$ (resp. $\pi$ involves $l_{d_X-j+\mathfrak{i}_1} \times h^{d_X-j}$). Let $j$ be an integer with $j_{\mathrm{lo}} \in \Lambda^U(X)$. By the definition of $S^{2^w}$ (see \S \ref{SUBSECstatementsonrationalcycles}), we have
$$ f(h^j \times l_{\mathfrak{i}_1 -1 + j}) = \sum_{k \geq 0}^{2^w} \binom{j}{2^w - k}\binom{\Izhdim{\varphi}- j}{k} h^{j + 2^w - k} \times l_{j + 2^w - k}. $$
Observe now that if $j<i$, then the products $\binom{j}{2^w - k}\binom{\Izhdim{\varphi}- j}{k}$ appearing in the sum are $0$ in $\mathbb{F}_2$. Indeed, in this case, $j$ is divisible by $2^v$ by hypothesis, and the same is true of $\Izhdim{\varphi}$ by Theorem \ref{THMi1}. Both $j$ and $\Izhdim{\varphi} - j$ are therefore divisible by $2^v$, and the claim then follows from Lucas' theorem (see \cite[Lem. 78.6]{EKM}). In particular, we see that if we express $f(\pi)$ as an $\mathbb{F}_2$-linear combination of the essential standard basis elements in $\mathrm{Ch}_d(X^2)$, then the coefficient of $h^k \times l_k$ is $0$ when $k<i$ and $\binom{\Izhdim{\varphi} - i}{2^w}$ when $j = i$. Since $i_{\mathrm{lo}}$ is connected to $0_{\mathrm{lo}}$ in $\Lambda(X)$ (being an element of $\Lambda^U(X)$), it follows that $\binom{\Izhdim{\varphi} - i}{2^w} = 0$ in $\mathbb{F}_2$.  But since $\Izhdim{\varphi}$ is divisible by $2^v$, another application of Lucas' theorem then tells us that $i$ is not divisible by $2^w$, a contradiction.  The result follows. 

(2)  We already know from part (2) of Proposition \ref{PROPuppersummand} that $i = \mathfrak{j}_t$ for some integer $1 \leq t < h_{\mathrm{nd}}$ with $\mathfrak{i}_{t+1} \geq \mathfrak{i}_1$.  Set $n: = v_2(\mathfrak{i}_1)$. By Theorem \ref{THMi1}, $\mydim{\varphi} = \Izhdim{\varphi} + \mathfrak{i}_1$ is divisible by $2^n$. By part (1), the same is then true of $\mydim{\varphi_t} = \mydim{\varphi} -2\mathfrak{j}_t$. Let $m$ be the smallest integer for which $\mathfrak{i}_{t+1} \leq 2^m$.  Since $\mathfrak{i}_{t+1} \geq \mathfrak{i}_1$, we have $m \geq n$.  Now $\mydim{\varphi_t} - \mathfrak{i}_{t+1} = \Izhdim{\varphi_t}$ is divisible by $2^m$ by another application of Theorem \ref{THMi1}. Since $m \geq n$,and since $\mydim{\varphi_t}$ is divisible by $2^n$, it follows that $\mathfrak{i}_{t+1}$ is also divisible by $2^n$, i.e., $v_2(\mathfrak{i}_{t+1}) \geq n$. This proves the result. 
 \end{proof} \end{theorem}
 
In the special case where $\Lambda^U(X)$ involves the ``second shell'', we get:

\begin{corollary} \label{CORconnectiontoshell2} If $i_{\mathrm{lo}} \in \Lambda^U(X)$ for some integer $\mathfrak{j}_1 \leq i< \mathfrak{j}_2$, then there exist nonnegative integers $v$ and $x$ such that:
\begin{itemize} \item[$\mathrm{(i)}$] $\mathfrak{i}_1 = 2^v$ and $\mathfrak{i}_2 = x2^v$;
\item[$\mathrm{(ii)}$] $\lbrace 0 \leq i < \mathfrak{j}_2\;|\; i_{\mathrm{lo}} \in \Lambda^U(X) \rbrace = \lbrace 0, 2^v, 2\cdot 2^v,\hdots, x2^v \rbrace$. \end{itemize}
\begin{proof} The existence of nonnegative integers $v$ and $x$ satisfying (i) is immediate from Theorem \ref{THMKarpenkouppersummand}. Suppose now that $\Lambda$ is an element of $\mathrm{MDT}(\varphi)$ containing $i_{\mathrm{lo}}$ for some $0 \leq i < \mathfrak{j}_2$. If we had $a(\Lambda) \geq \mathfrak{j}_1$, then Proposition \ref{PROPsummandshifts} (1) would imply that the same is true of every element of $\mathrm{MDT}(\varphi)$ containing $i_{\mathrm{lo}}$ for some $\mathfrak{j}_1 \leq i < \mathfrak{j}_2$. Since $\Lambda^U(X)$ is one of these elements, however, we must then have that $a(\Lambda) < \mathfrak{j}_1$. By another application of Proposition \ref{PROPsummandshifts} (1), we then have that $\Lambda=  \Lambda^U(X)[j]$ for some $0 \leq j < \mathfrak{i}_1$. At the same time, Proposition \ref{PROPsummandshifts} (2) implies that the components $\Lambda^U(X),\Lambda^U(X)[1],\hdots,\Lambda^U(X)[\mathfrak{i}_1 - 1]$ contain an equal number of elements $i_{\mathrm{lo}}$ with $0 \leq i < \mathfrak{j}_2$, and so (ii) must also hold. \end{proof} \end{corollary} 
 
When $\varphi$ is nondegenerate, another non-trivial restriction on the integer $i$ from part (2) due to Karpenko is essentially established in \cite[Thm. 81.2, Cor. 81.19]{EKM} (the characteristic assumption in \cite{EKM} was only imposed since the Steenrod operations of \cite{Primozic} where not available at that time). The same arguments, modulo the kind of adjustments implemented in the preceding discussion, yield an analogous result in our setting.  Since the proof is long and technical, we omit the details and simply state here the conclusion:

\begin{theorem} Suppose that $|\Lambda^U(X)|>2$, and let $1 \leq t < h_{\mathrm{nd}}$ be such that $\mathfrak{j}_t = \mathrm{min}\lbrace i>0\;|\; i_{\mathrm{lo}} \in \Lambda^U(X) \rbrace$ $($see Theorem \ref{THMKarpenkouppersummand}$)$. If $v_2(\mathfrak{j}_t - \mathfrak{i}_1) \geq v_2(\mathfrak{i}_1) + 2$, then $v_2(\mathfrak{i}_{t+1}) \leq v_2(\mathfrak{i}_1) + 1$.  In particular, 
$$v_2(\mathfrak{i}_1) \geq \mathrm{min}\lbrace v_2(\mathfrak{i}_2),\hdots,v_2(\mathfrak{i}_{h_{\mathrm{nd}}}) \rbrace -1. $$
\end{theorem}

\begin{remark} When $\varphi$ is nondegenerate, a sufficient condition for the inequality $|\Lambda^U(X)| > 2$ is that $\varphi$ does \emph{not} have maximal splitting: This is the ``binary motive theorem'' of Vishik, a special case of Theorem \ref{THMVishikconnections} below. When $\varphi$ is degenerate, this is no longer valid in general, but some sufficient conditions are (implicitly) discussed in \S \ref{SECexcellentconnections} below. 
\end{remark}

\section{Decompositions Arising from Stable Birational Equivalences of Quadratic Grassmannians} \label{SECVishik} Over fields of characteristic different from 2, a basic but important result of Vishik (see \cite[Thm. 4.7]{Vishik2}) allows to relate the Chow motives of the (smooth) projective quadrics attached to two anisotropic quadratic forms in situations where certain of their associated quadratic Grassmannians are stably birationally isomorphic. In particular, one can relate the motivic decomposition types of the two forms under the appropriate hypotheses (see Remarks \ref{REMSMDT}). In this section, we prove the analogous statement for the $\mathrm{MDT}$ invariant in our setting. While Vishik's arguments make use of quadratic Grassmannians and their motives, we give here a more direct argument lying within the framework of the previous sections. For nondegenerate forms, this yields the stronger motivic statement as in \cite{Vishik2}.\footnote{The characteristic-2 case of Vishik's result for nondegenerate forms appears to be absent from the literature (in particular, it is not discussed in \cite{EKM}).}

\subsection{Basic Result} Let $\varphi$ and $\psi$ be anisotropic quadratic forms over $F$ of types $(r,s)$ and $(r',s')$, respectively. We assume that $r,r' \geq 1$, i.e., that $\varphi$ and $\psi$ are nonquasilinear. Set $X := X_{\varphi}$ and $Y := Y_{\varphi}$. If necessary, we choose orientations of $X$ and $Y$. When considering separable extensions $K$ of $F$, we choose orientations of $X_K$ and $Y_K$ compatible with those for $X$ and $Y$ (note that $\varphi_K$ and $\psi_K$ remain nondefective here). When considering products of copies of these quadrics (over $F$ or a separable extension of $F$), we orient the product using the given orientations of the individual factors. 

\begin{lemma} \label{LEMVishik2} Suppose we have integers $1 \leq s < h_{\mathrm{nd}}(\varphi)$ and $1 \leq t < h_{\mathrm{nd}}(\psi)$ such that for every separable extension $K/F$, we have $\windex{\varphi_K} > \wittj{s-1}{\varphi}$ if and only if $\windex{\psi_K} > \wittj{t-1}{\psi}$. Let $F = F_0 \subseteq F_1 \subseteq \cdots \subseteq F_{h_{\mathrm{nd}}(\varphi)}$ be the nondefective splitting tower of $\varphi$, let $0 \leq s' <s$, and set $\varphi': = \varphi_{s'} = \anispart{(\varphi_{F_{s'}})}$ and $\psi': = \anispart{(\psi_{F_{s'}})}$. Then:
\begin{enumerate} \item The nondefective splitting pattern of $\psi'$ is contained in the set $\lbrace \wittj{k}{\psi} - \windex{\psi_{F_{s'}}}\;|\; 0 \leq k \leq h_{\mathrm{nd}}(\psi) \rbrace$, and contains $\wittj{t}{\psi} - \windex{\psi_{F_{s'}}}$. 
 \item There exists an integer $0 \leq t' < t$ such that $\wittj{t'}{\psi} - \windex{\psi_{F_{s'}}}$ lies in the nondefective splitting pattern of $\psi'$, and such that the following are equivalent for every separable extension $K/F_{s'}$:
 \begin{itemize} \item $\windex{\varphi'_K} > \wittj{s-1}{\varphi} - \wittj{s'}{\varphi}$ 
 \item $\windex{\psi'_K} > \wittj{t'}{\varphi} - \windex{\psi_{F_{s'}}}$;
 \item $\windex{\psi'_K} \geq \wittj{t}{\varphi} - \windex{\psi_{F_{s'}}} -1$. \end{itemize}
 
\end{enumerate}
\begin{proof} (1) Since any separable extension of $F_{s'}$ is a separable extension of $F$, the first statement holds by Lemma \ref{LEMnondefectiveheight}. For the second statement, let $K$ be the unique field in the nondefective splitting tower of $\psi$ for which $\windex{\psi_K} = \wittj{t}{\psi}$. By construction, $K/F$ is separable, and so $\windex{\varphi_K} > \wittj{s-1}{\varphi}$ by hypothesis. By repeated application of Lemma \ref{LEMfunctionfields} (2), the compositum $K \cdot F_{s'}$ is then a purely transcendental extension of $K$, and so $\windex{\psi'_{K \cdot F_{s'}}} = \windex{\psi_K} - \windex{\psi_{F_{s'}}} = \wittj{t}{\psi} - \windex{\psi_{F_{s'}}}$. Since $K/F$ and $F_{s'}/F$ are separable, so is $K \cdot F_{s'}/F$, and another application of Lemma \ref{LEMnondefectiveheight} then gives the claim.

(2) Since $\windex{\varphi_{F_{s'}}} = \wittj{s'}{\varphi} \leq \wittj{s-1}{\varphi}$, and since $F_{s'}/F$ is separable, we have $\windex{\psi_{F_{s'}}} \leq \wittj{t-1}{\psi}$. By (1), there is then a largest integer $0 \leq t'<t$ such that $\wittj{t'}{\psi} - \windex{\psi_{F_{s'}}}$ lies in the nondefective splitting pattern of $\psi'$. Since $\wittj{t}{\psi} - \windex{\psi_{F_{s'}}}$ also lies in this set by (1), it then follows from our standing assumption (and the fact that any separable extension of $F_{s'}$ is a separable extension of $F$) that $t'$ has the desired property.\end{proof}
\end{lemma}

\begin{lemma} \label{LEMVishik1} Let $1 \leq t < h_{\mathrm{nd}}(\psi)$, and set $a: = d_X - d_Y + \wittj{t-1}{\psi} + \wittj{t}{\psi} - 1$. Suppose that for every separable extension $K/F$, $\windex{\varphi_K} >0$ if and only if $\windex{\psi_K} > \mathfrak{j}_{t-1}$. Then $a\geq 0$, and any element of $\overline{\mathrm{Ch}}_{d_X + \wittj{t}{\psi} - 1}$ that involves $h^0 \times l_{\wittj{t}{\psi} - 1}$ also involves $l_a \times h^{\wittj{t-1}{\psi}}$. 
\begin{proof} As a matter of notation, let us set $l_i : = 0$ whenever $i<0$. Let $F_{t-1}$ be the unique field in the nondefective splitting tower for $\psi$ with $\windex{\psi_{F_{t-1}}} = \mathfrak{j}_{t-1}(\psi)$. By hypothesis, $\varphi$ remains anisotropic over $F_{t-1}$. Set $\psi' := \psi_{t-1} = \anispart{(\psi_{F_{t-1}})}$ and $Y' := X_{\psi'}$. If necessary, we choose an orientation of $Y'$ and use this to orient $X_{F_{t-1}} \times Y'$.  Let $p \colon R_{X \times Y} \rightarrow R_{X_{F_{t-1}} \times Y'}$ be the composition of the scalar extension map $R_{X \times Y} \rightarrow R_{(X \times  Y)_{F_{t-1}}}$ and the map $f \colon R_{(X \times  Y)_{F_{t-1}}} \rightarrow R_{X_{F_{t-1}} \times Y'}$ of Proposition \ref{PROPisotropicreduction}. By the latter, $p$ sends $F$-rational elements of the source to $F_{t-1}$-rational elements of the target, sends $h^0 \times l_{\wittj{t}{\psi}-1}$ to $h^0 \times l_{\wittj{1}{\psi_{t-1}}-1}$ and sends $l_a \times h^{\wittj{t-1}{\psi}}$ to $l_{a} \times h^0$. If $K/F_{t-1}$ is a separable extension, then we have $\windex{\varphi_K}>0$ if and only if $\windex{\psi'_K} >0$ by Lemma \ref{LEMVishik2}. To prove the lemma, we can therefore assume that $t=1$ (note that $d_Y' = d_Y - 2\mathfrak{j}_{t-1}(\psi)$, so replacing $\psi$ with $\psi_{t-1}$ does not change $a$). Under this assumption, $\varphi$ becomes isotropic over $F_1: = F(\psi)$. Consider the map $\mu \colon X_{F_1} = \mathrm{Spec}(F_1) \times X \rightarrow Y \times X$ given by the inclusion of the generic point of $Y$ on the first factor and the identity on the second. By \cite[Cor. 57.11]{EKM}, the pullback $\mu^* \colon \overline{\mathrm{Ch}}(Y \times X) \rightarrow \overline{\mathrm{Ch}}(X_{F_1})$ is surjective. Since $\varphi_{F_1}$ is isotropic, the target contains $l_0$. By the definition of $\mu$, it follows that there exists an indecomposable element of $\overline{\mathrm{Ch}}_{d_Y}(Y \times X)$ involving $h^0 \times l_0$. Applying $D^{\wittj{1}{\psi}-1,0}$ to this element, we can then find an indecomposable element $\beta \in \overline{\mathrm{Ch}}_{d_Y - \wittj{1}{\psi} +1}(Y \times X)$ that involves $h^{\wittj{1}{\varphi}-1} \times l_0$. Suppose now that $\alpha$ is an element of $\overline{\mathrm{Ch}}_{d_X + \wittj{1}{\psi} - 1}(X \times Y)$ that involves $h^0 \times l_{\wittj{1}{\psi} - 1}$. We can then write $\alpha = (h^0 \times l_{\wittj{1}{\psi} - 1}) + \lambda(l_a \times h^0) + \alpha'$ for some $\lambda \in \mathbb{F}_2$ and some $\alpha' \in (R_{X \times Y})_{d_X + \wittj{1}{\psi}-1}$ that involves neither $h^0 \times  l_{\wittj{1}{\psi} - 1}$ or $l_a \times h^0$.  Replacing $\alpha$ with $\mathrm{ess}(\alpha)$ if needed, we can further assume that $\alpha'$ involves no standard basis elements with $h^0$ as the second factor. Consider now the composite cycle $\alpha \circ \beta \in \overline{\mathrm{Ch}}_{d_Y}(Y^2)$. Since $\beta$ involves $h^{\wittj{1}{\psi}-1} \times l_0$, and since $\alpha$ involves $h^0 \times l_{\wittj{1}{\psi}-1}$, Lemma \ref{LEMcompositionformulas} shows that $\alpha \circ \beta$ involves $h^{\wittj{1}{\psi} - 1} \times l_{\wittj{1}{\psi}-1}$. Applying Proposition \ref{PROPsplittingpatternconnections} to an indecomposable element $\eta \in \overline{\mathrm{Ch}}_{d_Y}(Y^2)$ involving $h^{\wittj{1}{\psi} - 1} \times l_{\wittj{1}{\psi}-1}$ and satisfying $\eta \subset \alpha \circ \beta$, we then get that $\alpha \circ \beta$ also involves $l_0 \times h^0$. Since $\alpha'$ involves no standard basis elements with $h^0$ as the second factor, Lemma \ref{LEMcompositionformulas} then tells us that $\lambda(l_a \times h^0) \circ \beta \neq 0$, and so $a \geq 0$ and $\lambda = 1$, proving the lemma. \end{proof} \end{lemma}

Our version of Vishik's result is now the following:

\begin{theorem} \label{THMVishik} Suppose that there are integers $0 \leq m < r$ and $0 \leq n < r'$ such that for every separable extension $K/F$, we have $\windex{\varphi_K} > m$ if and only if $\windex{\psi_K} > n$.  If $\mathrm{MDT}(\varphi)$ admits an element $\Lambda$ with $a(\Lambda) = m$, then $\Lambda[n-m]$ is an element of $\mathrm{MDT}(\psi)$. 
\begin{proof} Let $s$ and $t$ be the unique integers with $\wittj{s-1}{\varphi} \leq m < \wittj{s}{\varphi}$ and $\wittj{t-1}{\psi} \leq n< \wittj{t}{\psi}$.  By Lemma \ref{LEMnondefectiveheight} and Proposition \ref{PROPsummandshifts}, our hypotheses are then equivalent to:
\begin{itemize} \item For every separable extension $K/F$, $\windex{\varphi_K} > \wittj{s-1}{\varphi}$ if and only if $\windex{\psi_K} > \wittj{t-1}{\psi}$;
\item $\mathrm{MDT}(\varphi)$ admits an element $\Lambda'$ with $a(\Lambda') = \wittj{s-1}{\varphi}$. \end{itemize}
We may thus assume that $m = \wittj{s-1}{\varphi}$ and $n = \wittj{t-1}{\psi}$.  Let us set
$$ a: = d_X-d_Y - \wittj{s-1}{\varphi} + \wittj{t-1}{\psi} + \wittj{t}{\varphi}-1 $$
and
$$ b: =  \wittj{t}{\varphi} - 1 - \wittj{s-1}{\varphi} = a + d_Y - d_X - \wittj{t-1}{\psi}. $$
We make the following claim: \vspace{.5 \baselineskip}

\noindent {\bf Claim.} In the above situation, $a \geq \wittj{s-1}{\varphi}$, and there exists a cycle $\beta \in \overline{\mathrm{Ch}}_{d_X +b}(X \times Y)$ with the following properties:
\begin{itemize} \item[(i)] $\beta$ involves $h^{\wittj{s-1}{\varphi}} \times l_{\wittj{t}{\psi}-1}$ and $l_a \times h^{\wittj{t-1}{\psi}}$;
\item[(ii)] If $\beta$ involves $h^u \times l_v$ (resp. $l_u \times h^v$) for some integers $u,v$, then $u_{\mathrm{lo}} \in \Lambda$ (resp. $(d_X-u)^{\mathrm{up}} \in \Lambda$);
\item[(iii)] $\beta$ involves no terms of the form $h^u \times h^v$ or $l_u \times l_v$ for integers $u$ and $v$. \end{itemize} \vspace{.5 \baselineskip}

Before proving it, let us first show how the claim yields the desired result. Let $\beta$ be as in the claim, and let $\beta^t$ be its image under pushforward along the factor exchange isomorphism $X \times Y \rightarrow Y \times X$.  Set
$$ \gamma: = (h^{\witti{t}{\varphi}-1} \times h^0)\beta^t \in \overline{\mathrm{Ch}}_{d_X + b - (\witti{t}{\psi}-1)}(Y \times X). $$
By (i),  $\beta$ involves $h^{\wittj{s-1}{\varphi}} \times l_{\wittj{t}{\psi} - 1}$ and $\gamma$ involves $h^{\wittj{t}{\psi}-1} \times l_a$. Using (ii) and (iii), together with Lemma \ref{LEMcompositionformulas}, one then checks that $\gamma \circ \beta \in \overline{\mathrm{Ch}}_{d_X - \wittj{s-1}{\varphi} + a}(X^2)$ involves 
$$(h^{\wittj{t}{\psi}-1} \times l_a) \circ (h^{\wittj{s-1}{\varphi}} \times l_{\wittj{t}{\psi} - 1}) = h^{\wittj{s-1}{\varphi}} \times l_a. $$
By Lemma \ref{LEMshells} it follows that $a < \wittj{s}{\varphi}$. But since $\beta$ involves $l_a \times h^{\wittj{t-1}{\psi}}$, (ii) tells us that $(d_X - a)^{\mathrm{up}} \in \Lambda$. Since $a(\Lambda) = \wittj{s-1}{\varphi}$, Corollary \ref{CORdimensionofsummand} then gives that $a \geq \wittj{s}{\varphi} -1$, and so $a = \wittj{s}{\varphi}-1$. Let us now consider the cycle
$$ \eta: = (h^{\witti{s}{\varphi}-1} \times h^0)\beta \in \overline{\mathrm{Ch}}_{d_X + b - (\witti{s}{\varphi}-1)}(X \times Y). $$
By (i) and the preceding remarks, $\beta^t$ involves $h^{\wittj{t-1}{\psi}} \times l_a = h^{\wittj{t-1}{\psi}} \times l_{\wittj{s}{\varphi}-1}$ and $\eta$ involves $h^{\wittj{s}{\varphi}-1} \times l_{\wittj{t}{\psi}-1}$. Using properties (ii) and (iii) of $\beta$, together with Lemma \ref{LEMcompositionformulas}, one then checks that $\eta \circ \beta^t \in \overline{\mathrm{Ch}}_{}(Y^2)$ involves 
$$ (h^{\wittj{s}{\varphi}-1} \times l_{\wittj{t}{\psi}-1}) \circ (h^{\wittj{t-1}{\psi}} \times l_{\wittj{s}{\varphi}-1}) = h^{\wittj{t-1}{\psi}} \times l_{\wittj{t}{\psi}-1}. $$
Let $\alpha$ be the unique indecomposable element of $\overline{\mathrm{Ch}}_{d_Y}(Y^2)$ which is essential and involves $l_{\wittj{t}{\psi}-1} \times h^{\wittj{t}{\psi}-1}$. By the preceding remarks, both $\alpha^t$ and $D^{\witti{t}{\psi}-1,0}(\eta \circ \beta^t) \in \overline{\mathrm{Ch}}_{d_Y}(Y^2)$ involve $h^{\wittj{t}{\psi}-1} \times l_{\wittj{t}{\psi}-1}$, and so $\alpha^t \subset D^{\witti{t}{\psi}-1,0}(\eta \circ \beta^t)$. We now claim that $\Lambda(\alpha) = \Lambda[\wittj{t-1}{\psi} - \wittj{s-1}{\varphi}] = \Lambda[m-n]$. Since $\alpha$ is indecomposable, this will prove what we want.  Let $0 \leq i < r'$. If $\alpha$ involves $h^i \times l_i$, then $\alpha^t$ involves $l_i \times h^i$, and so $\eta \circ \beta^t$ involves $l_{i + (\wittj{t}{\psi}-1)} \times h^i$. By Lemma \ref{LEMcompositionformulas}, $\beta^t \in \overline{\mathrm{Ch}}_{d_X + b}(Y \times X)$ must then involve $l_{i + (\wittj{t}{\psi}-1)} \times h^{i + (\wittj{t}{\psi}-1) - b}$, and so $\beta$ involves $h^{i + (\wittj{t}{\psi}-1) - b} \times l_{i + (\wittj{t}{\psi}-1)}$. Setting $k: = \wittj{t-1}{\psi} - \wittj{t-1}{\varphi} = b - (\witti{t}{\psi}-1)$, we see that $\alpha$ involves $h^i \times l_i$ only if $\beta$ involves $h^{i-k} \times l_{i +(\wittj{t}{\psi} - 1)}$. By property (ii) of $\beta$, this implies that $i_{\mathrm{lo}} \in \Lambda(\alpha)$ only if $(i-k)_{\mathrm{lo}} \in \Lambda$. In other words, $\Lambda(\alpha)_{\mathrm{lo}} \subseteq \Lambda[k]_{\mathrm{lo}}$. A similar calculation shows that $\Lambda(\alpha)^{\mathrm{up}} \subseteq \Lambda[k]^{\mathrm{up}}$, and so $\Lambda(\alpha) \subseteq \Lambda[k]$.  Now since $\alpha$ involves $l_{\wittj{t}{\psi}-1} \times h^{\wittj{t}{\psi}-1}$, we have $(d_Y - (\wittj{t}{\psi}-1))^{\mathrm{up}} \in \Lambda(\alpha)$. By Proposition \ref{PROPsplittingpatternconnections}, it follows that $\wittj{t-1}{\psi}_{\mathrm{lo}} \in \Lambda(\alpha)$. Since $a(\Lambda(k)) = a(\Lambda) + k = \wittj{s-1}{\varphi} + k = \wittj{t-1}{\psi}$, we then have that $a(\Lambda(\alpha)) = \wittj{t-1}{\psi}$.  As $\Lambda(\alpha)$ is an element of $\mathrm{MDT}(\psi)$, this puts us in a position to switch the roles of $\varphi$ and $\psi$ in the preceding arguments. In doing so, the roles of $\Lambda$ and $\Lambda(\alpha)$ are switched, and the role played by $k$ is assumed by $-k$. The preceding arguments then tells us that $\Lambda \subseteq \Lambda(\alpha)[-k]$. In other words, $\Lambda[k] \subseteq \Lambda(\alpha)$, and so $\Lambda(\alpha) = \Lambda[k] = \Lambda[\wittj{t-1}{\psi} - \wittj{s-1}{\varphi}]$, as desired. 

It remains to prove the claim. We first show by induction on $s$ that there exists an element $\nu \in \overline{\mathrm{Ch}}_{d_X + b}(X \times Y)$ involving $h^{\wittj{s-1}{\varphi}} \times l_{\wittj{t}{\psi}-1}$.  The case where $s=1$ was already done as part of the proof of Lemma \ref{LEMVishik1}.  Suppose now that $s\geq 2$, and let $F_1 = F(\varphi)$.  Set $\varphi' := \varphi_1 = \anispart{(\varphi_{F_1})}$ and $\psi' := \anispart{(\psi_{F_1})}$.  By Lemma \ref{LEMVishik2}, $\wittj{t}{\psi} - \windex{\psi_{F_1}}$ lies in the nondefective splitting pattern of $\psi$. Moreover, for any separable extension $K/F_1$, we have $\windex{\varphi'_K} > \wittj{s-1}{\varphi} - \witti{1}{\varphi}$ if and only if $\windex{\psi'_K} \geq \wittj{t}{\psi} - \windex{\psi_{F_1}}-1$.  Set $X': = X_{\varphi'}$ and $Y': = X_{\psi'}$. If necessary, we choose orientations of $X'$ and $Y'$, and use these to orient $X' \times Y'$.  By the induction hypothesis, $\overline{\mathrm{Ch}}_{d_X + b - \witti{1}{\varphi} - \windex{\psi_{F_1}}}(X' \times Y')$ has an element involving $h^{\wittj{s-1}{\varphi} - \witti{1}{\varphi}} \times l_{\wittj{t}{\psi} - 1 - \windex{\psi_{F_1}}}$.  Using the map $g$ from Proposition \ref{PROPisotropicreduction}, we then get a cycle in $\overline{\mathrm{Ch}}_{d_X + b}((X \times Y)_{F_1})$ involving $h^{\wittj{s-1}{\varphi}} \times l_{\wittj{t}{\psi}-1}$.  As in the proof of Lemma \ref{LEMVishik1}, \cite[Cor. 57.11]{EKM} allows us to lift this to a cycle $\xi \in \overline{\mathrm{Ch}}_{2d_X + b}(X^2 \times Y)$ involving $h^0 \times h^{\wittj{s-1}{\varphi}} \times l_{\wittj{t}{\psi}-1}$.  In $\overline{\mathrm{Ch}}_{d_X + d_Y}(X^2 \times Y)$, we also have the element $\xi' : = \alpha_{\Lambda} \times h^0$.  Using that $\Lambda(\alpha) = \wittj{s-1}{\varphi}$, one readily checks that the product $\xi \xi' \in \overline{\mathrm{Ch}}_{d_X+b}(X^2 \times Y)$ involves $h^{\wittj{s-1}{\varphi}} \times l_0 \times l_{\wittj{t}{\psi}-1}$. Pushing forward along the projection from $X^2 \times Y$ to the product of its two outer factors then gives the desired cycle $\nu$. Let us now set $\beta: = \mathrm{ess}(\nu \circ \alpha_{\Lambda})$. Since $\alpha_{\Lambda}$ involves $h^{\wittj{s-1}{\varphi}} \times l_{\wittj{s-1}{\varphi}}$, Lemma \ref{LEMcompositionformulas} gives that $\beta$ involves
$$ (h^{\wittj{s-1}{\varphi}} \times l_{\wittj{t}{\psi}-1}) \circ (h^{\wittj{s-1}{\varphi}} \times l_{\wittj{s-1}{\varphi}}) = h^{\wittj{s-1}{\varphi}} \times l_{\wittj{t}{\psi}-1}.$$
At the same time,  $\beta$ satisfies condition (ii) in the statement of our claim by construction. It also satisfies condition (iii): Let $0 \leq u <r$ and $0\leq v < r'$. Since $\beta$ is essential, it does not involve $h^u \times h^v$. If it involved $l_u \times l_v$, then we would have that $(h^u \times h^v)\beta = l_0 \times l_0$, and so $X \times Y$ would have a $0$-cycle of degree $1$. But then the same would be true of $X$ and $Y$, contradicting the anisotropy of $\varphi$ and $\psi$ via Springer's theorem. To complete the proof, it now only remains to show that $a \geq \wittj{s-1}{\varphi}$ and that $\beta$ involves $l_a \times h^{\wittj{t-1}{\psi}}$.  To see this, let $F_{s-1}$ be the unique field in the nondefective splitting tower of $\varphi$ with $\windex{\varphi_{F_{s-1}}} : = \wittj{s-1}{\varphi}$.  Overriding the notation used in the construction of $\beta$, let us now set $\varphi' := \varphi_{s-1} = \anispart{(\varphi_{F_{s-1}})}$ and $\psi' := \anispart{(\psi_{F_{s-1}})}$.  As before, set $X': = X_{\varphi'}$ and $Y' : = X_{\psi'}$. If needed, we choose orientations of $X'$ and $Y'$, and use these to orient $X' \times Y'$. Let $p \colon R_{X \times Y} \rightarrow R_{X' \times Y'}$ be the composition of the scalar extension map $R_{X \times Y} \rightarrow R_{(X \times Y)_{F_{s-1}}}$ and the map $f \colon R_{(X \times Y)_{F_{s-1}}} \rightarrow R_{X' \times Y'}$ of Proposition \ref{PROPisotropicreduction}. By the latter, $p(\beta)$ is then an element of $\overline{\mathrm{Ch}}_{d_X + b}(X' \times Y')$ involving $h^0 \times l_{\wittj{t}{\varphi} - 1 - \windex{\psi_{F_{s-1}}}}$.  By Lemma \ref{LEMVishik2}, there is an integer $0 \leq t' < t$ such that for any separable extension $K/F_{s-1}$, we have $\windex{\varphi'_K} > 0$ if and only if $\windex{\psi'_K} > \wittj{t'}{\varphi} - \windex{\psi_{F_{s-1}}}$.  Since $p(\beta)$ involves $h^0 \times l_{\wittj{t}{\varphi} - 1 - \windex{\psi_{F_{s-1}}}}$, Lemma \ref{LEMVishik1} then tells us that it also involves $l_{a'} \times h^{\wittj{t'}{\psi} - \windex{\psi_{F_{s-1}}}}$ for some integer $a' \geq 0$. In particular, $\beta$ involves $l_{\wittj{s-1}{\varphi} + a'} \times h^{\wittj{t'}{\psi}}$. Let $K$ be the unique field in the nondefective splitting tower of $\psi$ such that $\windex{\psi_K} = \wittj{t'+1}{\psi}$. Since $\beta$ involves $l_{\wittj{s-1}{\varphi} + a'} \times h^{\wittj{t'}{\psi}}$, Proposition \ref{PROPisotropyrestrictions} then implies that $\windex{\varphi_K} > \wittj{s-1}{\varphi} +a'$. Since $a' >0$, our standing assumption on $\varphi$ and $\psi$ then tells us that $\windex{\psi_K} > \wittj{t-1}{\psi}$, and so $t' = t-1$.  Thus, $\beta$ involves $l_{\wittj{s-1}{\varphi} + a'} \times h^{\wittj{t-1}{\psi}}$. Since $\beta$ has degree $d_X + b = d_Y - (\wittj{t-1}{\psi} + a)$, we then have that $a = \wittj{s-1}{\varphi} + a'$. Since $a' \geq 0$, this completes the proof. \end{proof} \end{theorem}

In the course of the proof, we showed that the integer $a: = d_X-d_Y - \wittj{s-1}{\varphi} + \wittj{t-1}{\psi} + \wittj{t}{\varphi}-1$ coincides with $\wittj{s}{\varphi} - 1$ under the standing hypotheses. In other words:

\begin{corollary} Suppose there are integers $1 \leq s \leq h_{\mathrm{nd}}(\varphi)$ and $1 \leq t \leq h_{\mathrm{nd}}(\psi)$ such that for every separable extension $K/F$, we have $\windex{\varphi_K} > \wittj{s-1}{\varphi}$ if and only if $\windex{\psi_K} > \wittj{t-1}{\psi}$. Then
$$ \mydim{\varphi} - \wittj{s-1}{\varphi} - \wittj{s}{\varphi} = \mydim{\psi} - \wittj{t-1}{\psi} - \wittj{t}{\psi}. $$ \end{corollary} 

Now, since $a(\Lambda^U(X)) = 0$, the $m=0$ case of Theorem \ref{THMVishik} gives:

\begin{corollary} \label{CORVishik} Suppose there exists an integer $0 \leq n < r'$ such that for every separable extension $K/F$, we have $\windex{\varphi_K} > 0$ if and only if $\windex{\psi_K} > n$. Then $\Lambda^U(X)[n]$ is an element of $\mathrm{MDT}(\psi)$. 
\end{corollary}

In particular, we get our earlier claim (extending the stable birational invariance of the Izhboldin dimension):

\begin{corollary} \label{CORstbinvarianceofuppersummands} If $\varphi \stb \psi$, then $\Lambda^U(X) = \Lambda^U(Y)$. 
\begin{proof} In this case, we have $\windex{\varphi_K} > 0$ if and only if $\windex{\psi_K}>0$ for every separable extension $K/F$.  By the previous corollary, it follows that $\Lambda^U(X)$ is an element of $\mathrm{MDT}(\psi)$. Since $a(\Lambda^U(X)) = 0$, we then have that $\Lambda^U(Y) = \Lambda^U(X)$.  \end{proof}
\end{corollary}

\begin{remarks}[Nondegenerate forms] In the case where $\varphi$ and $\psi$ are nondegenerate, Theorem \ref{THMVishik} gives the following stronger result: Suppose there are integers $0 \leq m <r$ and $0\leq n <r'$ such that for every separable extension $K/F$, we have $\windex{\varphi_K} > m$ if and only if $\windex{\psi_K} >n$. If $M(X)$ admits a direct summand $N$ in $\mathrm{Chow}(F,\mathbb{F}_2)$ such that $a(N) = m$, then $M(Y)$ admits a direct summand isomorphic to the Tate twist $N(n-m)$. This follows from Theorem \ref{THMVishik} and the discussion of \cite[Ch. XVII]{EKM}. Similarly, Corollary \ref{CORstbinvarianceofuppersummands} yields that when $\varphi \stb \psi$, the upper motives of $X$ and $Y$ in $\mathrm{Chow}(F,\mathbb{F}_2)$ are isomorphic.
\end{remarks}

\subsection{Examples: Pfister Neighbours and Strongly Excellent Forms} \label{SUBSECPNSandexcellent} Let $\varphi$ be an anisotropic nonquasilinear quadratic form over $F$ of type $(r,s)$. We set $X: = X_{\varphi}$, $h: = h_{\mathrm{nd}}(\varphi)$, and write $\mydim{\varphi} = 2^n + m$ for integers $n \geq 0$ and $1 \leq m \leq 2^n$. 

\begin{lemma} \label{LEMMDTofneighbour} If $\varphi$ is a Pfister neighbour with ambient general Pfister form $\pi$ and complementary form $\varphi^c$, then:
\begin{enumerate} \item For every integer $0 \leq i < m$ and every separable extension $K/F$, we have $\windex{\pi_K} > 0$ if and only if $\windex{\varphi_K} > i$;
\item If $m<r$, then for every integer $m \leq i <r$ and every separable extension $K/F$, we have $\windex{\varphi^c_K}>i-m$ if and only if $\windex{\varphi_K} > i$. \end{enumerate}
\begin{proof} (1) This holds since $\pi$ and $\varphi$ are stably birationally equivalent forms with maximal splitting (Lemma \ref{LEMPneighbours}). 

(2) Recall that $\varphi^c$ is an anisotropic form of dimension $2^n - m = \mydim{\varphi} - 2m$ satisfying $\varphi^c \sim \varphi \perp \pi$.  Like $\varphi$, $\varphi^c$ is dominated by $\pi$. In particular, if $K/F$ is a separable extension such that either $\varphi_K$ or $\varphi^c_K$ is isotropic, then $\pi_K$ is isotropic, and hence hyperbolic. Then $\varphi^c_K \sim \varphi_K$, and so $\windex{\varphi^c_K} = \windex{\varphi_K} - m$.  The desired assertion therefore holds.  \end{proof} \end{lemma}

Applying Theorem \ref{THMVishik}, we get:

\begin{theorem} \label{THMPfisterneighbourdecomposition} Suppose that $\varphi$ is a Pfister neighbour with complementary form $\varphi^c$. Then $r \geq m$, and the following hold:
\begin{enumerate} \item If $r=m$, then $\mathrm{MDT}(\varphi)$ consists of the sets $\lbrace 0_{\mathrm{lo}}, (2^n-1)^{\mathrm{up}} \rbrace[i]$ with $0 \leq i < m$;
\item If $r > m$, then $\varphi^c$ is a nonquasilinear nondefective form, and $\mathrm{MDT}(\varphi)$ consists of the sets $\lbrace 0_{\mathrm{lo}}, (2^n-1)^{\mathrm{up}} \rbrace[i]$ with $0 \leq i < m$, together with the sets $\Lambda[m]$ for $\Lambda \in \mathrm{MDT}(\varphi^c)$. \end{enumerate}
\begin{proof} By Lemma \ref{LEMPneighbours}, we have 
$$ r \leq 2^n - s = \mydim{\varphi} -m - s = (2r + s) -m-s = 2r - m, $$ 
and so $r \geq m$. Let $\pi$ be the ambient Pfister form of $\varphi$. Then $h_{\mathrm{nd}}(\pi) = 1$ (Lemma \ref{LEMclosePfisterneighbours}), and so $\Lambda^U(X_{\pi}) = \lbrace 0_{\mathrm{lo}}, (2^n-1)^{\mathrm{up}}\rbrace$ by Proposition \ref{PROPMDTfornondefectiveheight1}. In view of Lemma \ref{LEMMDTofneighbour} (1), Corollary \ref{CORVishik} then implies that $\mathrm{MDT}(\varphi)$ contains the sets $\lbrace 0_{\mathrm{lo}}, (2^n-1)^{\mathrm{up}} \rbrace[i]$ with $0 \leq i < m$. We now consider the two cases.

(1) If $r = m$, then the union of the elements of $\mathrm{MDT}(\varphi)$ found above is all of $\Lambda(X)$, and so $\mathrm{MDT}(\varphi)$ is as stated.

(2) If $r>m$, then Lemma \ref{LEMPneighbours} (3) gives that $\witti{1}{\varphi} = m <r$, and so $h \geq 2$. Moreover, we have $\varphi_1 \simeq \varphi^c_{F(\varphi)}$ by Lemma \ref{LEMPneighbours} (2), so $\varphi$  is nonquasilinear and nondefective of type $(r-m,s)$. Now if $m \leq i <r$, then Lemma \ref{LEMMDTofneighbour} (2) tells us that for every separable extension $K/F$, we have $\windex{\varphi^c_K}>i-m$ if and only if $\windex{\varphi_K} > i$. By Theorem \ref{THMVishik}, it follows that $\mathrm{MDT}(\varphi)$ also contains sets $\Lambda[m]$ with $\Lambda \in \mathrm{MDT}(\varphi^c)$. Since $|\Lambda(X_{\varphi^c})| = 2(r-m) = |\Lambda| - 2m$, the union of these sets and the other elements of $\mathrm{MDT}(\varphi)$ found previously is all of $\Lambda(X)$, and so $\mathrm{MDT}(\varphi)$ is as stated. \end{proof}
\end{theorem}

We can enhance this as follows. Recall that we write $\varphi_1$ for the form $\anispart{(\varphi_{F(\varphi)})}$. When $\varphi$ is a Pfister neighbour with complementary form $\varphi^c$, we have $\varphi_1 \simeq (\varphi^c)_{F(\varphi)}$ by Lemma \ref{LEMPneighbours} (2). We now have: 

\begin{proposition} \label{PROPMDTofcomplementaryform} In the situation of Theorem \ref{THMPfisterneighbourdecomposition} $\mathrm{(2)}$, we have $\mathrm{MDT}(\varphi^c) = \mathrm{MDT}(\varphi_1)$. Thus, $\mathrm{MDT}(\varphi)$ consists of the sets $\lbrace 0_{\mathrm{lo}}, (2^n-1)^{\mathrm{up}} \rbrace[i]$ with $0 \leq i < m$, together with the sets $\Lambda[m]$ for $\Lambda \in \mathrm{MDT}(\varphi_1)$. 
\begin{proof} Set $Y: = X_{\varphi^c}$. Let $\pi$ be the ambient general Pfister form of $\varphi$, and set $X: = X_{\pi}$. If $y \in Y \times Y$, then $Y(F(y)) \neq \emptyset$, and so $\varphi^c_{F(y)}$ is isotropic. Since $\pi$ dominates $\varphi$, it follows that $\pi_{F(y)}$ is isotropic, and hence hyperbolic. The canonical map $\mathrm{Ch}(X_{F(y)}) \rightarrow R_{X_{F(y)}}$ is then an isomorphism (see \S \ref{SUBSECpartialcelldec}), so we may identify $ \mathrm{Ch}(X_{F(y)})$ with $R_{X_{F(y)}}$. Let $0 \leq i \leq d_Y$.  Since $d_Y = 2^n -m - 2 < 2^n-1 = \frac{d_X}{2}$, the only nonzero element of $\mathrm{Ch}^{i}(X_{F(y)})$ is $h^i$, which lies in the image of the scalar extension map $\mathrm{Ch}^i(X) \rightarrow \mathrm{Ch}^i(X_{F(y)})$. By \cite[Lem. 88.5]{EKM}, it follows that the scalar extension map $\mathrm{Ch}_{d_Y}(Y^2) \rightarrow \mathrm{Ch}_{d_Y}(Y^2_{F(\pi)})$ is surjective.  The induced map $\overline{\mathrm{Ch}}_{d_Y}(Y^2) \rightarrow \overline{\mathrm{Ch}}_{d_Y}(Y^2_{F(\pi)})$ is then an isomorphism, and so $\mathrm{MDT}(\varphi^c) = \mathrm{MDT}(\varphi^c_{F(\pi)})$. Now since $\varphi \stb \pi$, the fields $F(\varphi)$ and $F(\pi)$ admit a common extension which is purely transcendental over both. By Lemma \ref{LEMMDTunderpurelytranscendental}, we then have that 
$$ \mathrm{MDT}(\varphi^c) = \mathrm{MDT}(\varphi^c_{F(\pi)}) = \mathrm{MDT}(\varphi^c_{F(\varphi)}) = \mathrm{MDT}(\varphi_1),$$
as desired.  \end{proof} \end{proposition}

This yields:

\begin{proposition} \label{PROPquasistronglyexcellent} Suppose that $\varphi_i$ is a Pfister neighbour for every $0 \leq i < h$.  Then there exist unique integers $n+1 = n_1 > n_2> \cdots > n_{h} > \mathrm{log}_2(s) + 1$ such that $\mydim{\varphi} = 2^{n_1} - 2^{n_2}  + \cdots + (-1)^{h-1}2^{n_h} + (-1)^{h}s$. Moreover:
\begin{enumerate} \item For each $0 \leq i < h$, we have $\mydim{\varphi_i} = 2^{n_{i+1}} - 2^{n_{i+2}} + \cdots + (-1)^{h-1-i}2^{n_h} + (-1)^{h-i}s$. 
\item If we set $m_{i+1}: = \mydim{\varphi_i} - 2^{n_{i+1}-1} = 2^{n_{i+1}-1} - 2^{n_{i+2}} + 2^{n_{i+3}} + \cdots + (-1)^{h-1-i}2^{n_{h}} + (-1)^{h-i}s$ for each $0 \leq i < h$, then $\mathrm{MDT}(\varphi)$ consists of the sets 
$$ \lbrace 0_{\mathrm{lo}}, (2^{n_{i+1}-1}-1)^{\mathrm{up}} \rbrace[m_1 + \cdots + m_{i} + j] $$
with $0 \leq i < h$ and $0 \leq j < m_{i+1}$.  \end{enumerate}
\begin{proof} Per the statement, let us set $n_1 : = n+1$. The integer $m_1$ in (2) is then equal to $m$, which equals $\witti{1}{\varphi}$ by Lemma \ref{LEMPneighbours} (3). By hypothesis, $\varphi$ is a neighbour of an $n_1$-fold general Pfister form, and so $\mydim{\varphi_1} = 2^{n_1} - \mydim{\varphi_1}$. We now proceed by induction on $h$. If $h = 1$, then $\varphi_1 \simeq \mathrm{ql}(\varphi)_{F(\varphi)}$, and so $s < 2^{n_1-1}$ and $\mydim{\varphi} = 2^{n_1} - s$.  It is clear that $n_1$ is the unique integer satisfying these conditions. Moreover, we have $r = \witti{1}{\varphi} = m$, and so Theorem \ref{THMPfisterneighbourdecomposition} (1) gives that $\mathrm{MDT}(\varphi)$ consists of the sets $ \lbrace 0_{\mathrm{lo}}, (2^{n_{1}-1}-1)^{\mathrm{up}} \rbrace[j]$ with $0 \leq j < m = m_0$.  This proves the desired assertions in this case. Suppose now that $h \geq 2$.  In this case,  $r > \witti{1}{\varphi} = m$. By Theorem \ref{THMPfisterneighbourdecomposition} and Proposition \ref{PROPMDTofcomplementaryform}, $\mathrm{MDT}(\varphi)$ then consists of the sets $ \lbrace 0_{\mathrm{lo}}, (2^{n_{1}-1}-1)^{\mathrm{up}} \rbrace[j]$ with $0 \leq j < m = m_1$, together with the sets $\Lambda[m] = \Lambda[m_1]$ with $\Lambda \in \mathrm{MDT}(\varphi_1)$. Note, however, that $\varphi_1$ is an anisotropic form of type $(r-m,s)$, nondefective height $h-1$ and dimension $<2^{n_1}$ satisfying the same condition as $\varphi$. By the induction hypothesis, it follows that there exist unique integers $n_1 > n_2 > \cdots > n_{h} > \mathrm{log}_2(s) + 1$ such that:
\begin{itemize} \item For each $1 \leq i <h$, $\mydim{\varphi_i} = 2^{n_{i+1}} - 2^{n_{i+2}} + \cdots + (-1)^{h-1-i}2^{n_{h}} + (-1)^{h-i}s$;
\item If we set $m_{i+1}: = \mydim{\varphi_i} - 2^{n_{i+1}-1} = 2^{n_{i+1}-1} - 2^{n_{i+2}} + 2^{n_{i+3}} + \cdots + (-1)^{h-1-i}2^{n_{h}} + (-1)^{h-i}s$ for each $1 \leq i < h$, then $\mathrm{MDT}(\varphi_1)$ consists of the sets 
$$ \lbrace 0_{\mathrm{lo}}, (2^{n_{i+1}-1}-1)^{\mathrm{up}} \rbrace[m_1 + \cdots + m_{i} + j] $$
with $1 \leq i < h$ and $0 \leq j < m_{i+1}$.  
\end{itemize}
In view of the preceding remarks, the desired assertions then follow.  \end{proof} \end{proposition}

By Lemmas \ref{LEMPneighbours} (2) and \ref{LEMstrongexcellence}, the hypothesis in the proposition is satisfied when $\varphi$ is strongly excellent (see the discussion following Lemma \ref{LEMstrongexcellence}). Thus:

\begin{corollary} \label{CORMDTstronglyexcellent} If $\varphi$ is strongly excellent, then there exist unique integers $n+1 = n_1 > n_2 > \cdots > n_{h} > \mathrm{log}_2(s) + 1$ such that $\mydim{\varphi} = 2^{n_1} - 2^{n_2}  + \cdots + (-1)^{h-1}2^{n_h} + (-1)^h s$. Moreover:
\begin{enumerate} \item For each $0 \leq i < h$, we have $\mydim{\varphi_i} = 2^{n_{i+1}} - 2^{n_{i+2}} + \cdots + (-1)^{h-1-i}2^{n_{h}} + (-1)^{h-i}s$. 
\item If we set $m_{i+1}: = \mydim{\varphi_i} - 2^{n_{i+1}-1} = 2^{n_{i+1}-1} - 2^{n_{i+2}} + 2^{n_{i+3}} + \cdots + (-1)^{h-1-i}2^{n_h} + (-1)^{h-i}s$ for each $0 \leq i < h$, then $\mathrm{MDT}(\varphi)$ consists of the sets 
$$ \lbrace 0_{\mathrm{lo}}, (2^{n_{i+1}-1}-1)^{\mathrm{up}} \rbrace[m_1 + \cdots + m_i + j] $$
with $0 \leq i < h$ and $0 \leq j < m_{i+1}$.  \end{enumerate} \end{corollary}

We expect that Proposition \ref{PROPquasistronglyexcellent} says nothing more than this, i.e., that if $\varphi_i$ is a Pfister neighbour for every $0 \leq i < h$, then $\varphi$ is strongly excellent. Proving this amounts to showing that if $h \geq 2$ and both $\varphi$ and $\varphi_1$ are Pfister neighbours, then $\varphi^c$ is also a Pfister neighbour. Specifically, what has to be shown here is that if $\pi$ is the ambient general Pfister form of $\varphi_1$, then $\pi$ descends from $F(\varphi)$ to $F$.  However, we currently do not know how to prove this, even in the case where $\varphi$ is nondegenerate.

\subsection{Addendum: Virtual Pfister Neighbours} \label{SUBSECvirtualPNs} As in the previous subsection, let $\varphi$ be a nonquasilinear anisotropic quadratic form over $F$ of type $(r,s)$. Set $X: = X_{\varphi}$, $h := h_{\mathrm{nd}}(\varphi)$, and write $\mydim{\varphi} = 2^n + m$ for integers $n \geq 0$ and $1 \leq m \leq 2^n$.

If there exists an extension $K/F$ such that $\varphi_K$ is an anisotropic Pfister neighbour, then we say that $\varphi$ is a \emph{virtual Pfister neighbour}.  Using the main result of the previous subsection, we get the following restriction on the MDT invariant for these forms:

\begin{lemma} \label{LEMMDTforVPNs} If $\varphi$ is a virtual Pfister neighbour, then:
\begin{enumerate} \item $m$ is in the nondefective splitting pattern of $\varphi$;
\item For each integer $0 \leq i < m$, $i_{\mathrm{lo}}$ and $(2^n+i - 1)^{\mathrm{up}}$ are connected in $\Lambda(X)$;
\item If $1 \leq t \leq h$ is such that $\wittj{t}{\varphi} = m$, then $\wittj{t-1}{\varphi}_{\mathrm{lo}} \in \Lambda^U(X)$.  \end{enumerate}
\begin{proof} (1) If $\varphi$ is a Pfister neighbour, then $m = \witti{1}{\varphi}$ (Lemma \ref{LEMPneighbours} (3)). If $\varphi$ is a virtual Pfister neighbour, $m$ then lies in the nondefective splitting pattern of $\varphi$ by Lemma \ref{LEMnondefectiveheight}. 

(2) If $K/F$ is a field extension with $\mathfrak{i}_{\mathrm{d}}(\varphi) = 0$, then it is clear that any connections existing in $\Lambda(X_K)$ also exist in $\Lambda(X)$. To prove the assertion, we can therefore assume that $\varphi$ is a Pfister neighbour. But in this case, the claim holds by Theorem \ref{THMPfisterneighbourdecomposition}. 

(3) As $0_{\mathrm{lo}} \in \Lambda^U(X)$, we have $(2^n - 1)^{\mathrm{up}} \in \Lambda^U(X)$ by (2). Since $\wittj{t}{\varphi} = m = \mydim{\varphi} - 2^n = d_X - 2^n - 2$, we then also have that $\wittj{t-1}{\varphi}_{\mathrm{lo}} \in \Lambda^U(X)$ by Proposition \ref{PROPsplittingpatternconnections}. 
\end{proof} \end{lemma} 

It is an intriguing problem to determine sufficient conditions for an anisotropic quadratic form to be a virtual Pfister neighbour.  While the nondefective splitting pattern is unable to detect this property in general, we do have the following special cases:

\begin{proposition} \label{PROPVPNs} $\varphi$ is a virtual Pfister neighbour in the following cases:
\begin{enumerate} \item $\varphi$ has maximal splitting, i.e, $\witti{1}{\varphi} = m$;
\item $m = 1$;
\item $m = 2$ and $2$ lies in the nondefective splitting pattern of $\varphi$;\end{enumerate} \end{proposition} 

Note that case (2) has been treated in \cite[Prop. 3.1]{HoffmannLaghribi2}. A small modification of the argument gives the more general case (1) (see below). Over fields of characteristic not 2, the analogues of cases (1) and (3) are due to Hoffmann (\cite[Cor. 3]{Hoffmann1}) and Izhboldin (\cite[Thm. 5.8]{Izhboldin1}), respectively.  The basis of the arguments (which goes back to \cite{Hoffmann1}) is the following: Let $T_1,\hdots,T_{n+1}$ be indeterminates, and set $K: = F(T_1,\hdots,T_{n+1})$. Consider the $(n+1)$-fold Pfister form $\pi: = \qpfister{T_1,\hdots}{T_{n+1}}$ over $K$. It is straightforward to check that $\pi$ is anisotropic, and showing that $\varphi$ is a Pfister neighbour amounts to showing that there exists an extension $L/K$ such that $\pi_L$ is anisotropic and dominates a nonzero scalar multiple of $\varphi_L$. We have the following lemma:

\begin{lemma} \label{LEMVPNpreliminaries} In the above situation, set $\psi: = \varphi \perp \pi$, and let $K = K_0 \subset K_1 \subset \cdots \subset K_{h_{\mathrm{nd}}(\psi)}$ be the nondefective splitting tower of $\psi$.
\begin{enumerate} \item Let $i$ be a nonnegative integer in the nondefective spitting pattern of $\varphi$. Then:
\begin{itemize} \item[$\mathrm{(i)}$] There exists an integer $1 \leq t < h_{\mathrm{nd}}(\psi)$ such that $\windex{\varphi_{K_t}} = i$ and $\mydim{\psi_t} = 2^n + m - 2i$;
\item[$\mathrm{(ii)}$] If $i \geq m$, then there exists a separable extension $L/K$ such that $\windex{\varphi_L} = i$ and $\mydim{\anispart{(\psi_L)}} = 2^n -m + 2i$.  \end{itemize}
\item Let $1 \leq t < h_{\mathrm{nd}}(\psi)$ be such that $\varphi_{K_t}$ is anisotropic. Suppose that any one of the following holds:
\begin{itemize} \item[$\mathrm{(i)}$] $\pi_{K_{t+1}}$ is isotropic;
\item[$\mathrm{(ii)}$] $\mydim{\anispart{((\psi_t)_{K_t(\varphi)})}} < 2^n + m - 2\witti{1}{\varphi}$;
\item[$\mathrm{(iii)}$] $\mydim{\anispart{(\varphi_{K_{t+1}})}} < 2^{n+1} - \mydim{\psi_{t+1}}$. \end{itemize}
Then $\varphi_{K_t} \prec \pi_{K_t}$, and so $\varphi_{K_t}$ is a Pfister neighbour. 
\end{enumerate} 
\begin{proof} (1) Since $\pi_{K(\pi)}$ is hyperbolic, we have $\psi_{K(\pi)} \sim \varphi_{K(\pi)}$. But it is straightforward to see that $K(\pi)$ is a purely transcendental extension of $F$, and so (i) holds by Lemma \ref{LEMnondefectiveheight}.  Suppose now that $i \geq m$. By Lemma \ref{LEMnondefectiveheight}, there exists a separable extension $F'/F$ such that $\windex{\varphi_{F'}} = i$. Set $\varphi' : = \anispart{(\varphi_{F'})}$. Then $\mydim{\varphi'} = \mydim{\varphi} - 2i = 2^n + m - 2i < 2^n$. By \cite[Prop. 3.1]{HoffmannLaghribi2}, it follows that there exists a separable extension $L$ of $F'(T_1,\hdots,T_{n+1}) = K \cdot F'$ such that $\pi_L$ is anisotropic and $\varphi'_L \prec \pi_L$.  Then $\windex{\varphi_L} = i$, and we have
$$ \psi_L \sim (\varphi \perp \pi)_L \sim \varphi'_L \perp \pi_L \sim (\varphi'_L)_{\pi_L}^c. $$
Since $\mydim{(\varphi'_L)_{\pi_L}^c} = 2^{n+1} - \mydim{\varphi'} = 2^n + m - 2i$, it follows that $\mydim{\anispart{(\psi_L)}} = 2^n -m + 2i$, proving the claim (note that $L/F$ is separable, being a tower of separable extensions).

(2) By definition, we have $\psi_t \sim \varphi_{K_t} \perp \pi_{K_t}$. Let us now consider the three cases of interest.

Suppose first that we are in case (i), and let $j \leq t$ be the largest integer for which $\pi_{K_j}$ is anisotropic. By Lemma \ref{LEMWittequivalenceofnondefective}, we have
$$ \psi_j \perp \pi_{K_j} \sim \varphi_{K_j} \perp \pi_{K_j} \perp \pi_{K_j} \sim \varphi_{K_j}. $$
Since $\mydim{\varphi_{K_j}} = 2^n + m < \mydim{(\psi_j \perp \pi_{K_j})}$, and since $\pi_{K_j}$ is anisotropic, $\pi_{K_j}$ and $\pi_{K_j}$ represent a common nonzero value of $K_j$. Since $\pi_{K_{j+1}}$ is isotropic, and hence hyperbolic, the Cassels-Pfister subform theorem (\cite[Thm. 22.5]{EKM}) then gives that $\psi_j \prec \pi_{K_j}$. By Lemma \ref{LEMWittequivalenceofnondefective}, we then have that
$$ \varphi_{K_j} \sim \varphi_{K_j} \perp \pi_{K_j} \perp \pi_{K_j} \sim \psi_j \perp \pi_{K_j} \sim (\psi_j)_{\pi_{K_j}}^c, $$
and so $\varphi_{K_j} \simeq (\psi_j)_{\pi_{K_j}}^c \prec \pi_{K_j}$ (both forms being anisotropic). In particular, $\varphi_{K_t} \prec \pi_{K_t}$. 

Suppose now that we are in case (ii) or (iii). Let $L$ be $K_t(\varphi)$ or $K_{t+1}$, depending on whether we are in the former or latter case, and set $\varphi': = \anispart{(\varphi_L)}$ and $\psi' := \anispart{(\psi_L)} = \anispart{((\psi_t)_L)}$.  Then $\psi' \sim \varphi' \perp \pi_L$, and the dimension hypothesis in each case tells us that $\mydim{\varphi'}+\mydim{\psi'} < 2^{n+1} = \mydim{\pi}$. Since $L/K$ is separable, Lemma \ref{LEMnondefectiveheight1} then gives that
$$ \witti{0}{\varphi' \perp \pi_L} = \windex{\varphi' \perp \pi_L} = \frac{\mydim{\varphi'} + \mydim{\pi} - \mydim{\psi'}}{2} > \mydim{\varphi'}. $$
If $\pi_L$ were anisotropic, it would then follow from \cite[Prop. 3.11]{HoffmannLaghribi1} that $\varphi' \prec \pi_L$.  But we would then have $\psi' \sim \varphi' \perp \pi_L \sim \varphi'_{\pi_L}$, giving that $\mydim{\psi'} = 2^{n+1} - \mydim{\varphi'}$, a contradiction. Thus, $\pi_L$ is isotropic, and hence hyperbolic. If $L = K_{t+1}$, we are then back in case (i). Suppose therefore that $L = K_t(\varphi)$. Again, if $\pi_{K_{t+1}}$ is isotropic, then we are in case (i), so we may assume otherwise. In particular, we may assume that $\pi_{K_t}$ is anisotropic. Now since $t \geq 1$, $\varphi_{K_t} \perp \pi_{K_t} = \psi_{K_t}$ is isotropic. Since both $\varphi_{K_t}$ and $\pi_{K_t}$ are anisotropic, they then represent a common element of $K_t^\times$. Since $\pi_L$ is hyperbolic, another application of the Cassels-Pfister subform theorem then gives that $\varphi_{K_t} \prec \pi_{K_t}$, as desired. \end{proof}\end{lemma}

We can now prove Proposition \ref{PROPVPNs}:

\begin{proof}[Proof of Proposition \ref{PROPVPNs}] As in Lemma \ref{LEMVPNpreliminaries}, set $\psi: = \varphi \perp \pi$, and let $K = K_0 \subset K_1 \subset \cdots \subset K_{h_{\mathrm{nd}}(\psi)}$ be the nondefective splitting tower of $\psi$.  By part (1)(i) of the lemma (with $i=0$) there exists an integer $1 \leq j < h_{\mathrm{nd}}(\psi)$ such that $\varphi_{K_j}$ is anisotropic and $\mydim{\psi_j} = \mydim{\varphi} = 2^n +m$. Now in all cases we are considering, $m$ lies in the nondefective splitting pattern of $\varphi$. By part (1)(ii) of the lemma, it follows that there exists an extension $L/K$ such that $\windex{\varphi_L} = m$ and $\mydim{\anispart{(\psi_L)}} = 2^n + m$.  By repeated application of Lemma \ref{LEMfunctionfields} (2), the compositum $L \cdot K_j$ is then a purely transcendental extension of $L$.  In particular, $\mydim{\anispart{(\psi_{L \cdot K_j})}} = \mydim{\anispart{(\psi_L)}} = 2^n + m = \mydim{\psi_j}$, and so $\psi_j$ remains anisotropic over $L \cdot K_j$. Since $\varphi_{L \cdot K_j}$ is isotropic, it follows that $\varphi_{K_j}$ and $\psi_j$ are not stably birationally equivalent.  At the same time, part (1)(i) of Lemma \ref{LEMVPNpreliminaries} also tells us that $\Izhdim{\psi_j} \geq \Izhdim{\varphi} \geq \Izhdim{\varphi_{K_j}}$. Since $\varphi_{K_j}$ and $\psi_j$ are not stably birationally equivalent, Theorems \ref{THMKM} and \ref{THMstb} then imply that $\varphi_{K_{j+1}}$ is anisotropic.  Since $\psi_{j+1} \sim \varphi_{K_{j+1}} \perp \pi_{K_{j+1}}$, and since $\mydim{\psi_{j+1}} < \mydim{\psi_j} = \mydim{\varphi}$, it follows that $\pi_{K_{j+1}}$ is not hyperbolic, and is hence also anisotropic. There are now two cases to consider. \vspace{.5 \baselineskip}

\noindent {\it Case 1.} $\varphi$ has maximal splitting. In this case, we have
$$ \mydim{\anispart{(\psi_{K_{j+1}(\varphi)})}} \leq \mydim{\psi_{j+1}} < \mydim{\psi_j} = 2^n + m = 2^n - m + 2\witti{1}{\varphi}. $$
Since $\varphi_{K_{t+1}}$ is anisotropic, Lemma \ref{LEMVPNpreliminaries} (2) (case (ii)) then tells us that it is a Pfister neighbour, and so $\varphi$ is a virtual Pfister neighbour. \vspace{.5 \baselineskip}

\noindent {\it Case 2.} $m = 2$ and $\witti{1}{\varphi} = 1$. In this case, part (1)(i) of Lemma \ref{LEMVPNpreliminaries} shows that $j \leq h_{\mathrm{nd}}(\psi) - 2$, and that $\mydim{\psi_{j+1}} = 2^n$ and $\mydim{\psi_{j+2}} = 2^n - 2 < 2^n + m - 2\witti{1}{\varphi}$. If $\varphi_{K_{j+2}}$ is anisotropic then Lemma \ref{LEMVPNpreliminaries} (2) (case (ii) with $t = j+2$) tells us that it is a Pfister neighbour.  On the other hand, if $\varphi_{K_{j+2}}$ is isotropic, then
$$ \mydim{\anispart{(\varphi_{K_{j+2}})}} < \mydim{\varphi} = 2^n + 2 = 2^{n+1} - (2^n - 2) = 2^{n+1} - \mydim{\psi_{j+2}}, $$
and so the same lemma (now case (iii) with $t = j+1$) then tells us that $\varphi_{K_{j+1}}$ is a Pfister neighbour. Thus, in this case, either $\varphi_{K_{j+1}}$ or $\varphi_{K_{j+2}}$ is an anisotropic Pfister neighbour. Either way, $\varphi$ is a virtual Pfister neighbour, and so the result holds. \end{proof}

\section{The Degenerate Pfister Neighbour Problem}  \label{SECPfisterneighbourconjecture} In this section we fix an anisotropic quadratic form $\varphi$ of type $(r,s)$ over $F$. We assume that $r \geq 1$ (i.e., that $\varphi$ is not quasilinear) and set $X: = X_{\varphi}$.  We also write $\mydim{\varphi} = 2^n + m$ for integers $n \geq 0$ and $1 \leq m \leq 2^n$. Our goal is to investigate the conditions under which $\varphi$ is a Pfister neighbour. By Lemma \ref{LEMPneighbours} (1), a necessary condition is that $r+s \leq 2^n$, and so we assume that this is satisfied in everything that follows. 

Now, if $\varphi$ is a Pfister neighbour, then $\varphi_1$ is defined over $F$. More precisely, we have $\varphi_1 \simeq (\varphi^c)_{F(\varphi)}$, where $\varphi^c$ is the complementary form of $\varphi$ in its ambient general Pfister form (Lemma \ref{LEMPneighbours} (2)). For nondegenerate forms, a well-known result (essentially due to Knebusch) asserts that the converse holds: If $\varphi$ is nondegenerate and $\varphi_1$ is defined over $F$, then $\varphi$ is a Pfister neighbour (\cite[Thm. 28.1]{EKM}). If we relax the nondegeneracy assumption, however, then this is no longer true in general. For instance, if $\varphi$ has nondefective height $1$, then $\varphi_1 \simeq \mathrm{ql}(\varphi)_{F(\varphi)}$ (Lemma \ref{LEMnondefectiveheight1}), but $\varphi$ need not be a Pfister neighbour.  Simple examples may be constructed as follows:

\begin{example} Suppose that $r = 1$. Since $r+s \leq 2^n$ and $2r + s = \mydim{\varphi} > 2^n$, we then have that $s = 2^n - 1$.  Now, since $r = 1$, we automatically have that $h_{\mathrm{nd}} = 1$ (Lemma \ref{LEMnondefectiveheight1}). However, $\varphi$ need not be a Pfister neighbour in this case. For example, let $X_1,\hdots,X_{2^n}$ be indeterminates,  and let $\sigma$ be the form $[1,X_1] \perp \langle X_2,\hdots,X_{2^n} \rangle$ of type $(1,2^n-1)$ over the rational function field $\mathbb{F}_2(X_1,\hdots,X_{2^n})$. It is straightforward to see that $\sigma$ is anisotropic. If $n \geq 2$, however, then $\sigma$ is not a Pfister neighbour. Indeed, suppose that $\sigma$ were a neighbour of an anisotropic $(n+1)$-fold Pfister form $\pi$ over $\mathbb{F}_2(X_1,\hdots,X_{2^n})$.  The complementary form $\sigma^c$ is then $\langle X_2,\hdots,X_{2^n} \rangle$. Over $K: = \mathbb{F}_2(X_1,\hdots,X_{2^n})(\sqrt{X_2X_3})$, $\sigma^c$ is isotropic, so $\pi_K$ is isotropic and hence hyperbolic. Then $\sigma_K \sim (\sigma \perp \pi)_K \sim \sigma^c_K$. Since $\sigma^c_K$ is isotropic, it follows that $\witti{0}{\varphi_K} \geq 2$. But it is again straightforward to check that $[1,X_1] \perp \langle X_3,\hdots,X_{2^n} \rangle$ remains anisotropic over $K$. Since this is a codimension-1 subform of $\sigma$, we cannot have that $\witti{0}{\sigma_K} \geq 2$, and so $\sigma$ is not a Pfister neighbour. 
\end{example}

Nevertheless, it has been conjectured by Hoffmann and Laghribi that the such examples cannot arise when $s$ is sufficiently small relative to $r$:

\begin{conjecture}[{\cite[Conj. 6.5]{HoffmannLaghribi1}}] \label{CONJHoffmannLaghribi} Suppose that $r+s \leq 2^n$, and that $s < 2r$. If $\varphi_1$ is defined over $F$, then $\varphi$ is a Pfister neighbour. 
\end{conjecture}

At present, this has only been established in the case where $s \leq 4$ (\cite[Thm. 6.6]{HoffmannLaghribi1}). Note that since $r + s \leq 2^n$, the inequality $s < 2r$ implies that $s < \frac{2^{n+1}}{3}$. We expect that the latter condition is in fact sufficient:

\begin{conjecture} \label{CONJPfisterneighbours}Suppose that $r+s \leq 2^n$, and that $s< \frac{2^{n+1}}{3}$. If $\varphi_1$ is defined over $F$, then $\varphi$ is a Pfister neighbour. \end{conjecture} 

We provide in this section some evidence for this conjecture. The basic result is the following, which shows that when $\varphi_1$ is defined over $F$, $\varphi$ exhibits key behaviour expected of a Pfister neighbour (see Lemma \ref{LEMPneighbours} and Theorem \ref{THMPfisterneighbourdecomposition}, in particular):

\begin{theorem} \label{THMPfisterneighbourproblem} Suppose that $r+s\leq 2^n$, and that $\varphi_1$ is defined over $F$. Then:
\begin{enumerate} \item $\varphi$ has maximal splitting;
\item $\varphi$ is a virtual Pfister neighbour;
\item $\Lambda^U(X) = \lbrace 0_{\mathrm{lo}}, (2^n-1)^{\mathrm{up}} \rbrace$;
\item There exists an anisotropic form $\psi$ of type $(2^n - s,s)$ over $F$ such that $\varphi \stb \psi$. Moreover, if $h_{\mathrm{nd}} = 1$, then $\psi$ may be taken to be $\varphi$ itself, i.e., $r = 2^n - s$. \end{enumerate}
\begin{proof} Let $\tau$ be a form over $F$ with $\varphi_1 \simeq \tau_{F(\varphi)}$. Since $\varphi_{F(\varphi)}$ is nondefective, we have
$$ \mathrm{ql}(\tau)_{F(\varphi)} \simeq \mathrm{ql}(\tau_{F(\varphi)}) \simeq \mathrm{ql}(\varphi_1) \simeq \mathrm{ql}(\varphi_{F(\varphi)}) \simeq \mathrm{ql}(\varphi)_{F(\varphi)}. $$
By Lemma \ref{LEMWittequivalenceofnondefective}, we then have that $\witti{0}{(\mathrm{ql}(\varphi) \perp \mathrm{ql}(\tau))_{F(\varphi)}} = s$. Since $F(\varphi)/F$ is separable, however, this implies that $\witti{0}{\mathrm{ql}(\varphi) \perp \mathrm{ql}(\tau)} = s$, and so $\mathrm{ql}(\varphi) \simeq \mathrm{ql}(\tau)$ by another application of Lemma \ref{LEMWittequivalenceofnondefective}. Let $\varphi'$ (resp. $\tau'$) be a nondegenerate form of dimension $2r$ (resp. $2r-2\witti{1}{\varphi})$) over $F$ such that $\varphi \simeq \varphi' \perp \mathrm{ql}(\varphi)$ (resp. $\tau \simeq \tau' \perp \mathrm{ql}(\varphi)$). Set $\psi: = \varphi' \perp \tau' \perp \mathrm{ql}(\varphi)$.  By definition, $\psi$ has type $(2r - \witti{1}{\varphi},s)$. Moreover, if $\varphi$ has nondefective height $1$, then $\tau' = 0$, and so $\psi = \varphi$. We state the following: \vspace{.5 \baselineskip}

\noindent {\bf Claim.} For any separable extension $K/F$, we have $\windex{\varphi_K} > 0$ if and only if $\windex{\psi_K} > 2r - \witti{1}{\varphi} - 1$. 
\vspace{.5 \baselineskip}

Before proving the claim, let us first use it to complete the proof of the proposition.  Set $i: = \windex{\psi}$. If $i$ were equal to $2r - \witti{1}{\varphi}$, then we would have that $\psi \sim \mathrm{ql}(\varphi)$. By Lemma \ref{LEMWittequivalenceofnondefective}, however, this would give that
$$ \varphi \sim \varphi \perp \tau' \perp \tau' \simeq \psi \perp \tau' \sim \mathrm{ql}(\varphi) \perp \tau' \simeq \tau, $$
contradicting the anisotropy of $\varphi$ (note that $\mydim{\tau} = \mydim{\varphi} - 2\witti{1}{\varphi} < \mydim{\varphi}$). We therefore have that $i< 2r - \witti{1}{\varphi}$, and so $\psi': = \anispart{\psi}$ is a nondefective and nonquasilinear form of type $(2r - \witti{1}{\varphi} - i,s)$.  If $K/F$ is a separable extension, then the claim tells us that $\windex{\varphi_K} > 0$ if and only if $\windex{\psi'_K} >2r - \witti{1}{\varphi} - i - 1$. By Corollary \ref{CORVishik}, it follows that $\Lambda: = \Lambda^U(X)[2r - \witti{1}{\varphi} - i-1]$ is an element of $\mathrm{MDT}(\psi')$. Let $h$ be the nondefective height of $\psi'$. Since $\psi'$ has type $(2r - \witti{1}{\varphi} - i,s)$, and since $a(\Lambda) = 2r - \witti{1}{\varphi} - i - 1$, Proposition \ref{PROPsummandshifts} and Corollary \ref{CORdimensionofsummand} imply that $\Lambda = \lbrace (2r - \witti{1}{\varphi} - i - 1)_{\mathrm{lo}}, (\mydim{\psi'} - 2 - \wittj{h-1}{\psi'})^{\mathrm{up}} \rbrace$.  Since $b(\Lambda^U(X)) = \Izhdim{\varphi} - 1$ (Proposition \ref{PROPuppersummand}), it then follows that 
$$ (\Izhdim{\varphi} - 1) + (2r - \witti{1}{\varphi} - i - 1) = \mydim{\psi'} - 2 - \wittj{h-1}{\psi'}. $$  
Since $\Izhdim{\varphi} = 2r + s - \witti{1}{\varphi}$ and $\mydim{\psi'} = 2(2r -\witti{1}{\varphi}-i-1) + s$, this amounts to the equality $i = -\wittj{h-1}{\psi'}$, and so $i = \wittj{h-1}{\psi'} = 0$. In other words, $\psi$ is anisotropic of nondefective height $1$.  By the claim, we then have that $\varphi \stb \psi$.  In particular, $\Izhdim{\psi} = \Izhdim{\varphi} \geq 2^n$ (Theorem \ref{THMstb}, Corollary \ref{CORmaxsplitting}). On the other hand, since $\psi$ has nondefective height $1$, we have $\mydim{\psi_1} = \mydim{\mathrm{ql}(\psi)} = s<2^n$ (Lemma \ref{LEMnondefectiveheight1}).  By Theorem \ref{THMi1} (Karpenko's theorem on the values of $\mathfrak{i}_1$), we must then have that $\Izhdim{\psi} = 2^n$. By the preceding remarks, we then also have that $\Izhdim{\varphi} = 2^n$, and so (1) holds. In view of Proposition \ref{PROPVPNs}, this also gives (2). For (3), we have already seen above that $|\Lambda^U(X)| = |\Lambda|= 2$. By Proposition \ref{PROPuppersummand}, however, we then have that $\Lambda^U(X) =\lbrace 0_{\mathrm{lo}}, (\Izhdim{\varphi}-1)^{\mathrm{up}} \rbrace =  \lbrace 0_{\mathrm{lo}}, (2^n-1)^{\mathrm{up}} \rbrace$, and so the desired assertion holds. It now only remains to prove (4). Since we have already shown that $\varphi \stb \psi$, we just have to show that $\psi$ has type $(2^n - s, s)$, i.e., that $2r - \witti{1}{\varphi} = 2^n - s$.  But since $\mydim{\varphi} = 2r + s$, this is simply a reformulation of the fact that $\Izhdim{\varphi} = 2^n$. 

We now complete the proof by proving the claim. Let $K/F$ be a separable extension. Suppose first that $\windex{\psi_K} > 2r - \witti{1}{\varphi} - 1$. Since $\psi$ has type $(2r - \witti{1}{\varphi},s)$, we then have that $\psi_K \sim \mathrm{ql}(\psi)_K = \mathrm{ql}(\varphi)_K$.  By Lemma \ref{LEMWittequivalenceofnondefective}, we then have that
$$ \varphi_K \sim (\varphi \perp \mathrm{ql}(\varphi))_K \sim (\varphi \perp \psi)_K \sim (\tau' \perp \mathrm{ql}(\varphi))_K \simeq \tau_K, $$
and so $\varphi_K$ is isotropic (again,  we have $\mydim{\tau} < \mydim{\varphi}$). Since $K/F$ is separable, we then have that $\windex{\varphi_K} > 0$ (Lemma \ref{LEMnondefectiveheight}). Conversely, if $\windex{\varphi_K} >0$, then the extension $K(\varphi)/K$ is purely transcendental by Lemma \ref{LEMfunctionfields} (2). In particular, we have $\windex{\psi_K} = \windex{\psi_{K(\varphi)}} \geq \windex{\psi_{F(\varphi)}}$. But since $\psi = \varphi' \perp \tau$, we have
$$ \psi_{F(\varphi)} \simeq \varphi'_{F(\varphi)} \perp \tau_{F(\varphi)} \simeq \varphi'_{F(\varphi)} \perp \varphi_1 \sim (\varphi' \perp \varphi)_{F(\varphi)} \sim (\varphi' \perp \varphi' \perp \mathrm{ql}(\varphi))_{F(\varphi)} \sim \mathrm{ql}(\varphi)_{F(\varphi)} $$
by Lemma \ref{LEMWittequivalenceofnondefective}, and so $\windex{\psi_F(\varphi)} > 2r - \witti{1}{\varphi} - 1$. In particular, $\windex{\psi_K} > 2r - \witti{1}{\varphi} - 1$, and so the claim holds. \end{proof} \end{theorem}

Note, in particular, that we have the following dimension restriction on the forms of nondefective height $1$ (under our assumption that $r + s \leq 2^n$):

\begin{corollary} \label{CORdimensionofnondefectiveheight1} Suppose that $r + s \leq 2^n$. If $\varphi$ has nondefective height $1$, then $\mydim{\varphi} = 2^{n+1} - s$.
\begin{proof} In this case, part (4) of Theorem \ref{THMPfisterneighbourproblem} tells us that $r = 2^n - s$, and so $\mydim{\varphi} = 2r + s = 2(2^n - s) + s = 2^{n+1} -s$. 
\end{proof} \end{corollary}

As remarked above, $\varphi_1$ is trivially defined over $F$ in the case where $h_{\mathrm{nd}}(\varphi) = 1$. In particular, Conjecture \ref{CONJPfisterneighbours} includes the following as a special case:

\begin{conjecture} \label{CONJnondefectiveheight1} Suppose that $r + s \leq 2^n$, and that $s < \frac{2^{n+1}}{3}$. If $\varphi$ has nondefective height $1$, then is a Pfister neighbour.  \end{conjecture}

Recall that we use the term \emph{close Pfister neighbour} for a Pfister neighbour of nondefective height 1. If $\varphi$ is a close Pfister neighbour, then $\mydim{\varphi} = 2^{n+1} - s$ by Lemma \ref{LEMclosePfisterneighbours}. Thus, Corollary \ref{CORdimensionofnondefectiveheight1} confirms the dimension part of Conjecture \ref{CONJnondefectiveheight1}. Moreover, as far as Conjecture \ref{CONJPfisterneighbours} goes, Theorem \ref{THMPfisterneighbourproblem} gives the following:

\begin{corollary} \label{CORequivalenceofconjectures} Conjectures \ref{CONJHoffmannLaghribi}, \ref{CONJPfisterneighbours} and \ref{CONJnondefectiveheight1} are equivalent.
\begin{proof} We have already noted that Conjecture \ref{CONJPfisterneighbours} implies \ref{CONJHoffmannLaghribi}. 

Conjecture \ref{CONJHoffmannLaghribi} $\Rightarrow$ Conjecture \ref{CONJnondefectiveheight1}: Suppose that $s < \frac{2^{n+1}}{3}$. If $\varphi$ has nondefective height $1$, then $\varphi_1$ is defined over $F$ (see the preceding discussion). Moreover, Corollary \ref{CORdimensionofnondefectiveheight1} tells us that $\mydim{\varphi} = 2^{n+1} - s$, and so
$$ 2r = \mydim{\varphi} - s = 2^{n+1} - 2s > 2^{n+1} - 2\left(\frac{2^{n+1}}{3}\right) = \frac{2^{n+1}}{3} > s. $$
By Conjecture \ref{CONJHoffmannLaghribi}, $\varphi$ is then a Pfister neighbour.

Conjecture \ref{CONJnondefectiveheight1} $\Rightarrow$ Conjecture \ref{CONJPfisterneighbours}: Suppose that $s < \frac{2^{n+1}}{3}$ and that $\varphi_1$ is defined over $F$. By Theorem \ref{THMPfisterneighbourproblem}, $\varphi$ has maximal splitting, and there exists an anisotropic form $\psi$ of type $(2^n - s,s)$ over $F$ such that $\varphi \stb \psi$. By Theorem \ref{THMstb} (or Corollary \ref{CORVishik}), we then have that $\Izhdim{\psi} = \Izhdim{\varphi} = 2^n$, and so $\witti{1}{\psi} = \mydim{\psi} - 2^n = 2(2^n -s) + s - 2^n = 2^n - s$. By Lemma \ref{LEMnondefectiveheight1}, this means that $\psi$ has nondefective height $1$. Since $s< \frac{2^{n+1}}{3}$, Conjecture \ref{CONJnondefectiveheight1} then implies that $\psi$ is a Pfister neighbour. But since $\varphi \stb \psi$, the same is then true of $\varphi$, and so the claim holds.  \end{proof} \end{corollary}

If we impose a slightly stronger assumption on $s$, then we can also relate Conjecture \ref{CONJPfisterneighbours} to a well-known conjecture on nondegenerate forms.  In fact, we state here two conjectures, the first of which is (essentially) due to Vishik, and the second of which is (essentially) due to Hoffmann (note that the $r+s \leq 2^n$ condition is vacuously satisfied here):

\begin{conjectures} \label{CONJSRostmotives} Suppose that $\varphi$ is nondegenerate.
\begin{enumerate} \item If $|\Lambda^U(X)| = 2$, then $\varphi$ is a Pfister neighbour.
\item If $\varphi$ has maximal splitting and $\mydim{\varphi} > 2^n + 2^{n-2}$, then $\varphi$ is a Pfister neighbour.\end{enumerate}
\end{conjectures}

Note that in (2), the standing hypotheses imply that $\witti{1}{\varphi}>\witti{t}{\varphi}$ for all $1 \leq t < h_{\mathrm{nd}}(\varphi)$. It then follows from Proposition \ref{PROPsummandshifts} (2) that $|\Lambda^U(X)| = 2$, and so (2) is in fact implied by (1).  Nevertheless, one may hope that a more direct approach to (2) is achievable, and so we include it in our discussion. Using Theorem \ref{THMPfisterneighbourproblem}, we can now show the following:

\begin{proposition} \label{PROPreductionofPNproblem} Suppose that $r + s \leq 2^n$, and that $\varphi_1$ is defined over $F$. Then $\varphi$ is a Pfister neighbour in the following cases: 
\begin{enumerate} \item $s \leq 2^{n-1}$ and Conjecture \ref{CONJSRostmotives} $\mathrm{(1)}$ holds;
\item $s \leq 2^{n-1} - 2^{n-3}$ and Conjecture \ref{CONJSRostmotives} $\mathrm{(2)}$ holds.  \end{enumerate}
\begin{proof} We may assume that $s \leq 2^{n-1}$. By Theorem \ref{THMPfisterneighbourproblem} (4), there exists an anisotropic form $\psi$ of type $(2^n - s,s)$ over $F$ such that $\varphi \stb \psi$. Per the proof of Corollary \ref{CORequivalenceofconjectures}, $\psi$ has nondefective height $1$, and so $\psi_1$ is defined over $F$. Moreover, since $r+s \leq 2^n$, we have $\mydim{\varphi} = 2r + s = 2(r+s) - s \leq 2^{n+1} -s = \mydim{\psi}$. Since $\varphi$ is a Pfister neighbour if and only if $\psi$ is, we can then replace $\varphi$ with $\psi$ in order to reduce to the case where $\mydim{\varphi} = 2^{n+1} - s$. By Theorem \ref{THMPfisterneighbourproblem} (1), we then have that $\witti{1}{\varphi} = \mydim{\varphi} - 2^n = 2^n - s \geq s$.  Let $\tau$ be a nondegenerate subform of codimension $s-1$ in $\varphi$. Since $\witti{1}{\varphi} \geq s$, $\tau_{F(\varphi)}$ is isotropic and so $\varphi \stb \tau$. To prove that $\varphi$ is a Pfister neighbour in cases (1) and (2), it then suffices to show that the same is true of $\tau$. 

(1) Since $\tau \stb \varphi$, we have $\Lambda^U(X_{\tau}) = \Lambda^U(X)$ by Corollary \ref{CORstbinvarianceofuppersummands}. But since $\varphi_1$ is defined over $F$, Theorem \ref{THMPfisterneighbourproblem} (3) then tells us that $|\Lambda^U(X_{\tau})| = |\Lambda^U(X)| = 2$. Since $\tau$ is nondegenerate,  it is then a Pfister neighbour by our hypothesis.

(2) Since $\tau \stb \varphi$, we have $\Izhdim{\tau} = \Izhdim{\varphi} = 2^n$ by Theorem \ref{THMstb} (or Corollary \ref{CORstbinvarianceofuppersummands}).  In other words, $\tau$ has maximal splitting. At the same time, we also have that
$$ \mydim{\tau} > \mydim{\varphi} - s = 2^{n+1} - 2s \geq 2^{n+1} - (2^{n-1} - 2^{n-3}) = 2^n + 2^{n-2}. $$
Since $\tau$ is nondegenerate, it is then a Pfister neighbour by our hypothesis. \end{proof} \end{proposition} 

\section{Excellent Connections and the First Higher Isotropy Index} \label{SECexcellentconnections} In this final section, we fix an anisotropic quadratic form $\varphi$ of type $(r,s)$ over $F$. We assume that $r \geq 1$ (i.e., that $\varphi$ is not quasilinear), and set $h := h_{\mathrm{nd}}(\varphi)$ and $X: = X_{\varphi}$. For ease of notation, we also set $\mathfrak{j}_t: = \wittj{t}{\varphi}$ and $\mathfrak{i}_t: = \witti{t}{\varphi}$ for all $1 \leq t \leq h$. 

\subsection{Excellent Pairs} Let $j$ and $n_1,\hdots,n_j$ be the unique nonnegative integers for which $\mydim{\varphi} = 2^{n_1} - 2^{n_2} + \cdots + (-1)^{j-1}2^{n_j}$ and $n_1>n_2> \cdots > n_{j-1} > n_j + 1$.  For each integer $1 \leq i \leq j$, set
$$ m_i : = 2^{n_i-1}- 2^{n_{i+1}} + \cdots + (-1)^{j-i}2^{n_j}. $$
Let $j'$ be $j$ or $j-1$ depending on whether $\mydim{\varphi}$ is even or odd. Accounting for Remark \ref{REMSMDT} (2), the following definition is taken from \cite[\S 3]{Vishik2}:

\begin{definition} Let $a,b$ be integers with $0 \leq a,d_X-b < r$. We say that the pair $(a,b)$ is \emph{excellent} (for $\varphi$) if there exists an integer $1 \leq k \leq j'$ such that the following hold:
\begin{enumerate} \item $b-a = 2^{n_k-1}-1$;
\item $a,d_X-b \in \left[\sum_{i=1}^{k-1}m_i, \sum_{i=1}^k m_i\right)$. \end{enumerate} \end{definition} 

The following lemma, which is a translation of Corollary \ref{CORMDTstronglyexcellent}, explains the terminology:

\begin{lemma} \label{LEMexcellentpairs} Let $a$ and $b$ be integers with $0 \leq a,d_X - b < r$. If $\varphi$ is strongly excellent, then the pair $(a,b)$ is excellent if and only if $a_{\mathrm{lo}}$ and $b^{\mathrm{up}}$ are connected in $\Lambda(X)$. 
\end{lemma}

\subsection{Excellent Connections for Nondegenerate Forms} In the situation of Lemma \ref{LEMexcellentpairs}, one may whether the necessity part of the statement remains valid if we relax the assumption that $\varphi$ be strongly excellent. Over fields of characteristic different from $2$, the analogous problem was shown to have a positive answer by Vishik in \cite[Thm. 1.3]{Vishik2}. Aside from the fact that anisotropic forms are nondegenerate in that setting, the only other reason for the characteristic restriction in the latter was the use of Brosnan's Steenrod operations for Chow groups modulo $2$. Thanks to \cite{Primozic}, exactly the same arguments now go through for nondegenerate forms in characteristic $2$, and so we have: 

\begin{theorem} \label{THMexcellentconnections} Suppose that $\varphi$ is nondegenerate, and let $a$ and $b$ be integers with $0 \leq a,d_X - b <r$. If the pair $(a,b)$ is excellent, then $a_{\mathrm{lo}}$ and $b^{\mathrm{up}}$ are connected in $\Lambda(X)$.  \end{theorem}

Theorem \ref{THMexcellentconnections} is perhaps the most profound known result concerning the MDT invariant. As shown by Vishik, it has a number of significant implications, including the following:

\begin{theorem} \label{THMVishikconnections} Let $1 \leq t < h$.  Suppose that $\varphi$ is nondegenerate, and that $\mathrm{MDT}(\varphi)$ admits an element $\Lambda$ with $a(\Lambda) = \mathfrak{j}_{t-1}$. Let $j$ and $n_1,\hdots,n_j$ be the unique nonnegative integers for which $\mydim{\varphi} - 2\mathfrak{j}_{t-1} - \mathfrak{j}_t = 2^{n_1} -2^{n_2} + \cdots + (-1)^{j-1}2^{n_j}$ and $n_1>n_2> \cdots > n_{j-1} > n_j + 1$. For each integer $1 \leq k \leq j$, consider the integer
$$d_k: = \left(\frac{\mydim{\varphi} - 2\mathfrak{j}_{t-1} - \mathfrak{j}_t}{2}\right) + \sum_{i=k}^j (-1)^{k+i-1}2^{n_i - 1}. $$
Then $(d_k)_{\mathrm{lo}} \in \Lambda$. Moreover, if $t \leq t' \leq h$ is such that $\mathfrak{j}_{t'-1} < d_k \leq \mathfrak{j}_{t'}$, then $d_k + \mathfrak{i}_t \leq \mathfrak{j}_{t'}$. 
 \begin{proof} The second statement follows from the first and Proposition \ref{PROPsummandshifts} (2). It therefore suffices to prove the first statement. In the notation of \cite{Vishik2}, the element $\Lambda \in \mathrm{MDT}(\varphi)$ gives an indecomposable direct summand $N$ of $M(X)$ in $\mathrm{Chow}(F,\mathbb{F}_2)$ with $a(N) = \mathfrak{j}_{t-1}$ (see Remark \ref{REMSMDT} (1)). The integer $d_X - 2\mathfrak{j}_{t-1} - \mathfrak{j}_t$ is then equal to $\mathrm{dim}(N) + 1$, and the claim is equivalent to the second statement in \cite[Thm. 2.1]{Vishik2} (formulated in our setting). Given Theorem \ref{THMexcellentconnections}, however, the proof of the latter goes through in characteristic $2$: The auxiliary roles played by \cite[Prop. 1.1]{Vishik2} and \cite[Obs. 2.3]{Vishik2} are assumed here by Proposition \ref{PROPsplittingpatternconnections} and Lemma \ref{LEMMDTforisotropicforms}, and the assertion of \cite[Lem. 2.2]{Vishik2} remains valid in light of Corollary \ref{CORstbinvarianceofuppersummands}. With these remarks, the proof then goes through verbatim. \end{proof} \end{theorem} 

We expect that Theorem \ref{THMexcellentconnections} remains valid if we relax the requirement that $\varphi$ be nondegenerate. Before discussing this and its implications, we first give some applications of Theorem \ref{THMVishikconnections} to problems in the degenerate setting. 

\subsection{Forms of Nondefective Height 1 Revisited} Let $n$ be the unique integer for which $2^n < \mydim{\varphi} \leq 2^{n+1}$. In this subsection, we return to the case where $h = 1$, i.e., where $\varphi$ has nondefective height $1$. If $r +s \leq 2^n$ (equivalently, $\mydim{\varphi} \leq 2^{n+1}-s$) and $s < \frac{2^{n+1}}{3}$, then Conjecture \ref{CONJnondefectiveheight1} predicts that $\varphi$ is a (close) Pfister neighbour in this case. With a slight adjustment of the second condition, the first condition may in fact be removed:

\begin{proposition} \label{PROPnondefectiveheight1withoutr+s} Suppose that $\varphi$ has nondefective height $1$. If $s \leq \frac{\mydim{\varphi}}{3}$, then $\mydim{\varphi} = 2^{n+1} - s$. 
\begin{proof} Since $\mydim{\varphi} = 2r + s$, the assumption on $s$ tells us that $r \geq s$. Since $h=1$, we then have that $\mathfrak{i}_1 = r \geq s$ (Lemma \ref{LEMnondefectiveheight1}).  Let $\psi$ be a form of type $(r,1)$ dominated by $\varphi$. Since $\mathfrak{i}_1 \geq s$, we have that $\varphi \stb \psi$ (Lemma \ref{LEMstbsubforms}). In particular, we have $\Lambda^U(X) = \Lambda^U(X_{\psi})$ (Corollary \ref{CORstbinvarianceofuppersummands}).  Since $h=1$, Proposition \ref{PROPMDTfornondefectiveheight1} then gives that $|\Lambda^U(X_{\psi})| = |\Lambda^U(X)| = 2$. Since $\psi$ is nondegenerate, Theorem \ref{THMVishikconnections} then tells us that $\Izhdim{\psi}$ is a power of $2$. But since $\varphi \stb \psi$, Theorem \ref{THMstb} then tells us that the same is true of 
$$ r+s = \mydim{\varphi} - r = \mydim{\varphi} - \mathfrak{i}_1 = \Izhdim{\varphi}. $$
By Corollary \ref{CORmaxsplitting}, we then have that $r + s = 2^n$, and so $\mydim{\varphi} = 2(r+s) - s = 2^{n+1} - s$, as claimed. \end{proof} \end{proposition} 

In view of the discussion of the previous section, we therefore make the following conjecture on the classification of forms of nondefective height $1$: 

\begin{conjecture} \label{CONJsecondconjectureonnondefectiveheight1} If $\varphi$ has nondefective height $1$, and $s \leq \frac{\mydim{\varphi}}{3}$, then $\varphi$ is a Pfister neighbour.  \end{conjecture}

For larger values of $s$, it is less clear what can be expected. 

\subsection{The Values of the First Higher Isotropy Index} In this subsection, we consider the possible values of the integer $\mathfrak{i}_1$. Recall from Theorem \ref{THMi1} that $\mathfrak{i}_1$ is at most the largest $2$-power divisor of $\mydim{\varphi} - \mathfrak{i}_1 = \Izhdim{\varphi}$.  If we fix the value of $\mydim{\varphi}$ and take no further information into account, then this result cannot be bettered: Given integers $d \geq 2$ and $i \geq 1$ such that $i$ is at most the largest power of $2$ that divides $d - i$, there exist an extension $K/F$ and an anisotropic quadratic form $\psi$ over $K$ such that $\mydim{\psi} = d$ and $\witti{1}{\psi} = i$. If we take into account the type $(r,s)$, however, then the situation already changes. First, the standard construction that demonstrates the optimality of the previous assertion (cf. the proof of \cite[Thm. 79.9]{EKM}) yields only the following:

\begin{lemma} \label{LEMvaluesofi1} Let $v$ be a nonnegative integer, and let $x$ be the unique integer for which $x2^v < r \leq (x+1)2^v$. Suppose that $i$ is a positive integer satisfying the following conditions:
\begin{itemize} \item[$\mathrm{(i)}$] $2r + s - i$ is divisible by $2^v$;
\item[$\mathrm{(ii)}$] $i \leq r - x2^v$. \end{itemize}
Then there exist an extension $K/F$ and an anisotropic quadratic form $\psi$ of type $(r,s)$ over $K$ such that $\witti{1}{\psi} =i$.  Moreover, $\Izhdim{\psi} = 2r + s - i = y2^v$ for some integer $y \geq 2x + 1$. 
\begin{proof} By (i), there exists an integer $y$ such that $2r + s - i = y2^v$. By (ii), we have
$$ y2^v = 2r + s - i \geq 2r + s - (r - x2^v) = r + s + x2^v > 2x2^v, $$
and so $y \geq 2x + 1$. Let $X_1,\hdots,X_v Y_1,\hdots,Y_{2x +2},Z_1,\hdots, Z_{y - (2x + 1)}$ be indeterminates, and set $K: = F(X_1,\hdots,Z_{y - (2x+1)})$.  Over $K$, we may consider the form $\tau: = \mathfrak{b} \otimes q$, where $\mathfrak{b} := \pfister{X_1,\hdots,X_v}_b$ and $q : = Y_1[1,Y_2] \perp \cdots \perp Y_{2x + 1}[1,Y_{2x + 2}] \perp \langle Z_1,\hdots,Z_{y - (2x + 1)} \rangle$.  We claim that $\witti{1}{\tau} = 2^v$. Consider the field $L: = K[T]/(T^2 + T + Y_2)$. A routine computation (using the rationality of $L/F(X_1,\hdots,\hat{Y_2},\hdots,Z_{y-(2x+1)})$, where the hat indicates omission) shows that $\windex{\tau_{L}} = 2^v$. Since $L/K$ is separable, Lemma \ref{LEMnondefectiveheight} then gives that $\witti{1}{\tau} \leq 2^v$. On the other hand, if $L/K$ is any other separable extension, then $\mathfrak{b}_L$ is anisotropic (anisotropic quasilinear quadratic forms remain anisotropic under separable extensions) and is still a divisor of $\anispart{(\tau_L)}$ (a routine consequence of the roundness of bilinear and quadratic Pfister forms). We therefore have that $\witti{1}{\varphi} \geq 2^v$, whence the claim.  Observe now that $\tau$ has dimension $(y+1)2^v$ and type $(r',s')$, where $r' = (x+1)2^v,$ and $s' =(y-(2x+1))2^v$.  Since
\begin{alignat*}{3} s \quad & = \quad && y2^v + i - 2r \\ \quad & \leq \quad && y2^v + (r - x2^v) - 2r \\ 
\quad & = \quad && (y-x)2^v - r \\ \quad & = \quad && \left((x+1)2^v - r\right) + \left(y-(2x+1)\right)2^v \\ \quad & = \quad && (r'-r) + s', \end{alignat*}
$\tau$ therefore admits a subform $\psi$ of type $(r,s)$.  Then
$$ \mydim{\psi} = 2r + s = y2^v + i = (y+1)2^v - (2^v-i) = \mydim{\tau} - (2^v -i), $$
and since $i\leq 2^v$ (assumption (ii)),  we then have that $\witti{1}{\psi} = i$ by Lemma \ref{LEMstbsubforms}. This proves the first statement, and we then have that $\Izhdim{\psi} = y2^v$, whence the second.  \end{proof} \end{lemma} 

All examples of anisotropic forms with nontrivial first higher isotropy index of which we are currently aware of are either of this type or stably of this type (i.e., acquire the desired shape over an extension that preserves the anisotropy of the form).  We therefore formulate the following question, a positive integer to which would yield a refinement of Theorem \ref{THMi1} that takes the type invariant into account:

\begin{question} \label{Qi1} Let $u$ be the largest integer for which $\Izhdim{\varphi}$ is divisible by $2^u$, and let $x$ be the unique integer for which $x2^u < r \leq (x+1)2^u$. Is is then true that $\mathfrak{i}_1 \leq r - x2^u$?
\end{question}

Theorem \ref{THMi1} gives the following:

\begin{lemma} \label{LEMQi1} Let $u$ be the largest integer for which $\Izhdim{\varphi}$ is divisible by $2^u$, and let $x$ be the unique integer for which $x2^u < r \leq (x+1)2^u$. Then $\Izhdim{\varphi} = y2^u$ for an odd integer $y \geq 2x + 1$. In particular $\mathfrak{i}_1 \leq 2(r-x2^u) + s - 2^u$.  
\begin{proof} By hypothesis there is an odd integer $y$ with $\Izhdim{\varphi} = y2^u$. Set $\alpha: = r - x2^u$. Note that $\alpha>0$ by definition. Now by Theorem \ref{THMi1}, we have $\mathfrak{i}_1 \leq 2^u$, and so
$$ y2^u = \Izhdim{\varphi} = 2r + s - \mathfrak{i}_1 \geq 2r + s - 2^u = (2x-1)2^u + 2 \alpha + s. $$
Since $y$ is odd, the first statement follows. Then
$$ \mathfrak{i}_1 = 2r + s - y2^u \leq 2r + s - (2x+1)2^u = 2\alpha + s - 2^u, $$
and so the second statement also holds.  \end{proof} \end{lemma}

Consider in particular the case where $s < 2^u + \mathfrak{i}_1$.  With the notation of the lemma, we then have that
$$ (y-2x)2^u = \Izhdim{\varphi} - (2x)2^u = 2(r-x2^u) + s - \mathfrak{i}_1 < 2(r-x2^u) + 2^u < 3\cdot 2^u. $$
Since $y$ is odd and at least $2x + 1$, it follows that $y = 2x + 1$, i.e., $\Izhdim{\varphi} = (2x+1)2^u$.  Put another way, we have
$$ \mathfrak{i}_1 = 2(r-x2^u) -(2^u - s). $$
Thus, when $s<2^u + \mathfrak{i}_1$, giving a positive answer to Question \ref{Qi1} amounts to showing that $r - x2^u \leq 2^u - s$. In other words, it amounts to a positive answer to the following:

\begin{conjecture} \label{CONJi1smalls} Let $u$ be the largest integer for which $\Izhdim{\varphi}$ is divisible by $2^u$, and let $x$ be the unique integer for which $x2^u<r \leq (x+1)2^u$. If $s<2^u + \mathfrak{i}_1$, then $r + s \leq (x+1)2^u$. \end{conjecture}

In this direction, Theorem \ref{THMVishikconnections} allows us to say the following:

\begin{theorem} \label{THMi1restrictionssingular} Let $u$ be the largest integer for which $\Izhdim{\varphi}$ is divisible by $2^u$, and let $x$ be the unique integer for which $x2^u < r \leq (x+1)2^u$. If $\witti{1}{\varphi} \geq s$, then $s < 2^u$ and $\mathfrak{i}_1 \leq r -x2^u \leq 2^u - s$.  In particular, Conjecture \ref{CONJi1smalls} holds in this case.
\begin{proof} We proceed as in the proof of Proposition \ref{PROPnondefectiveheight1withoutr+s}. Let $\psi$ be a form of type $(r,1)$ dominated by $\varphi$. Since $\mathfrak{i}_1 \geq s$, we have that $\varphi \stb \psi$ (Lemma \ref{LEMstbsubforms}). In particular, we have that $\Izhdim{\varphi} = \Izhdim{\psi}$ and $\Lambda^U(X) = \Lambda^U(X_{\psi})$ (Lemma \ref{LEMstbsubforms}, Corollary \ref{CORstbinvarianceofuppersummands}).  Let $j$ and $n_1,\hdots,n_j$ be the unique integers for which $\Izhdim{\varphi} = 2^{n_1} - 2^{n_2} + \cdots + (-1)^{j-1}2^{n_j}$ and $n_1>n_2> \cdots > n_{j-1} > n_j + 1$. By definition, we have $n_j = u$. Consider the integer
$$ d_j: = \frac{\Izhdim{\varphi}}{2} - 2^{u-1} = \left(\frac{\mydim{\varphi} - 2\mathfrak{j}_0 -  \mathfrak{j}_1}{2}\right)+(-1)^{2j-1} 2^{n_j-1}. $$
Since $\psi$ is nondegenerate, and since $\Izhdim{\psi} = \Izhdim{\varphi}$, Theorem \ref{THMVishikconnections} tells us that $(d_j)_{\mathrm{lo}} \in \Lambda^U(X_{\psi})$. Since $\Lambda^U(X_{\psi}) = \Lambda^U(X)$, Proposition \ref{PROPsummandshifts} (2) then tells us that $\frac{\Izhdim{\varphi}}{2} - 2^{u-1} + \mathfrak{i}_1 \leq r$. But since $\Izhdim{\varphi} = 2r + s - \mathfrak{i}_1$, this says that $\mathfrak{i}_1 \leq 2^u - s$. In particular, $s < 2^u$. By the remarks preceding Conjecture \ref{CONJi1smalls}, we then have that $\mathfrak{i}_1 = 2(r-x2^u) - (2^u -s)$. Since $\mathfrak{i}_1 \leq 2^u - s$, we then in fact have that $\mathfrak{i}_1 \leq r - x2^u \leq 2^u - s$, as desired.  \end{proof} \end{theorem}

The preceding result is enough, for instance, to completely describe the situation in which $\mydim{\varphi} \leq 9$: 

\begin{examples} \label{EXSi1slowdimension} Suppose $\mydim{\varphi} \leq 9$. We discuss the possible values of $\mathfrak{i}_1$ and when they may occur. Note that $\mathfrak{i}_1 \leq r$ by Lemma \ref{LEMnondefectiveheight}. \vspace{.5 \baselineskip}

\noindent \underline{$\mydim{\varphi} \in \lbrace 2,3,5,9 \rbrace$}. In this case $\mathfrak{i}_1 = 1$ by Corollary \ref{CORmaxsplitting}.  \vspace{.5 \baselineskip}

\noindent \underline{$\mydim{\varphi} =4$}. In this case,  $\mathfrak{i}_1$ is either $1$ or $2$. If $\mathfrak{i}_1 = 2$, then $r = 2$, so $\varphi$ is nondegenerate of Knebusch height $1$. It is well known that the nondegenerate anisotropic forms of dimension $4$ and Knebusch height $1$ are precisely the anisotropic general $2$-fold Pfister forms (\cite[Prop. 25.6]{EKM}), so $\mathfrak{i}_1 = 2$ when $\varphi$ is a general $2$-fold Pfister form, and $\mathfrak{i}_1 = 1$ otherwise. \vspace{.5 \baselineskip}

\noindent \underline{$\mydim{\varphi} = 6$}. In this case, $\mathfrak{i}_1$ is again either $1$ or $2$ by Corollary \ref{CORmaxsplitting}. If $\varphi$ is a Pfister neighbour, then $\witti{1}{\varphi} = 2$ by Lemma \ref{LEMPneighbours} (3).  Conversely, if $\mathfrak{i}_1 = 2$, then $\varphi$ has maximal splitting, and is hence a Pfister neighbour by \cite[Thm. 1.2]{HoffmannLaghribi2}.  Thus, $\mathfrak{i}_1 = 2$ if $\varphi$ is a Pfister neighbour, and $\mathfrak{i}_1 = 1$ otherwise. Note that in order to be in the first case, we require that $r \geq 2$, i.e., that $\varphi$ is either nondegenerate or of type $(2,2)$. 

\vspace{.5 \baselineskip}

\noindent \underline{$\mydim{\varphi} = 8$}.  In this case, $\mathfrak{i}_1$ is $1$, $2$ or $4$ by Theorem \ref{THMi1}. In order for $i_1$ to be $4$, we must have that $r =4$, i.e., that $\varphi$ is nondegenerate of Knebusch height $1$. Again, it is well known that the nondegenerate anisotropic forms of dimension $8$ and Knebusch height $1$ are precisely the anisotropic $3$-fold general Pfister forms (\cite[Prop. 25.6]{EKM}). Now in order for $\mathfrak{i}_1$ to equal $2$, we must have that $r \geq 2$, and so $\varphi$ must have type $(4,0)$, $(3,2)$ or $(2,4)$. However, Theorem \ref{THMi1restrictionssingular} excludes the second case, so $\varphi$ must be nondegenerate or of type $(2,4)$.  Consider first the case where $\varphi$ is nondegenerate. If $\varphi \simeq \mathfrak{b} \otimes [1,a]$ for some $4$-dimensional bilinear form $\mathfrak{b}$ over $F$ and $a \in F$, then $\mathfrak{i}_1$ is even.  Indeed, consider the separable quadratic extension $K: = F[T]/(T^2 + T + a)$. Since $[1,a]_K$ is hyperbolic, so is $\varphi_K$, and hence $(\varphi_1)_{K(\varphi)}$. By \cite[Cor. 23.6]{EKM}, it follows that $\varphi_1 \simeq \mathfrak{c} \otimes [1,a]_K$ for some symmetric bilinear form $\mathfrak{c}$ over $K$. But $\varphi$ has trivial discriminant, so the same is true of $\varphi_1$, and hence $\mathfrak{c}$ must be of even dimension, proving the claim. In particular,  since $r = 4$, we have $\witti{1}{\varphi} = 2$ provided that $\varphi$ is not a general $3$-fold Pfister form. Conversely, suppose that $\mathfrak{i}_1 = 2$. The $\Izhdim{\varphi} = 8 - 2$, and so $2_{\mathrm{lo}} \in \Lambda^U(X)$ by Theorem \ref{THMVishikconnections}. By Corollary \ref{CORconnectiontoshell2}, it follows that $\mathfrak{i}_2 = 2$.  In view of the preceding discussion, this means that $\varphi_1$ is a $2$-fold general Pfister form. In particular, the Clifford algebra of $\varphi_1$ is Brauer equivalent to a quaternion algebra. By the index reduction theorem (see \cite[\S 30]{EKM}), the same is then true of the Clifford algebra of $\varphi$. In particular, there exists a $2$-fold general Pfister form $\tau$ over $F$ such that the $12$-dimensional form $\varphi \perp \tau$ represents an element of $I^3_q(F)$. Scaling $\tau$ if needed, we can assume that $\varphi \perp \tau$ is isotropic. The anisotropic part of $\varphi \perp \tau$ then has dimension at most $8$ (\cite[Thm. 4.10]{DeMedts}), and so $\varphi \perp \tau \sim \pi$ for a general $3$-fold Pfister form $\pi$ over $F$ (\cite[Cor. 25.12]{EKM}). Since $\pi \perp \tau \sim \varphi$ (Lemma \ref{LEMWittequivalenceofnondefective}), $\pi \perp \tau$ is isotropic. By \cite[Thm. 24.2 and Prop. 24.1]{EKM}, it follows that there exist symmetric bilinear forms $\mathfrak{c}$ and $\mathfrak{d}$ over $F$, and an element $a \in F$ such that $\pi \simeq \mathfrak{c} \otimes [1,a]$ and $\tau \simeq \mathfrak{d} \otimes [1,a]$. If we again set $K: = F[T]/(T^2 +T+a)$, we then get that $\varphi_K \sim \pi_K \perp \tau_K$ is hyperbolic. By \cite[Cor. 23.6]{EKM}, it then follows that $\varphi \simeq \mathfrak{b} \otimes [1,a]$ for a $4$-dimensional symmetric bilinear form $\mathfrak{b}$ over $F$. Consider now the case where $\varphi$ has type $(2,4)$. If $\varphi \simeq \pfister{a}_b \otimes q$ for some $a \in F^\times$ and form $q$ of type $(1,2)$ over $F$, then $\mathfrak{i}_1$ is even by the argument in the proof of Lemma \ref{LEMvaluesofi1}.  Since $r = 2$, we must then have that $\mathfrak{i}_1 = 2$.  Conversely, if $\mathfrak{i}_1 = 2$, then $\varphi \simeq \pfister{a}_b \otimes q$ for some $a \in F^\times$ and form $q$ of type $(1,2)$ over $F$ by \cite[Theorem 7.5]{HoffmannLaghribi1}.  In summary:
\begin{itemize} \item $\mathfrak{i}_1 = 4$ if and only if $\varphi$ is a general $3$-fold Pfister form;
\item $\mathfrak{i}_1 = 2$ if and only if $\varphi$ is not a general $3$-fold Pfister form, and one of the following holds:
\begin{itemize} \item $\varphi  \simeq \mathfrak{b} \otimes [1,a]$ for some $4$-dimensional symmetric bilinear form $\mathfrak{b}$ of non-trivial determinant over $F$;
\item $\varphi \simeq \pfister{a}_b \otimes q$ for some $a \in F^\times$ and some form $q$ of type $(1,2)$ over $F$; \end{itemize}
\item $\mathfrak{i}_1 = 1$ in all other cases. \end{itemize} \end{examples}

As far as Question \ref{Qi1} goes, the first open case appears in dimension $12$: If $\varphi$ is a $12$-dimensional form of type $(3,6)$, we do not know if it is possible for $\mathfrak{i}_1$ to equal $2$.  Our expectation is that this is not possible, and we present in the next subsection a conjectural approach to this based on the methods of this article. This approach should allow to treat Conjecture \ref{CONJi1smalls}, as well as other cases of Question \ref{Qi1} in which $s$ is small relative to $r$. In general, however, it is unclear to what extent the methods used here are applicable in situations where $s$ is ``large''. For example, we also do not know if $\mathfrak{i}_1$ can equal $2$ if $\varphi$ is a $14$-dimensional form of type $(3,8)$ or $(5,4)$. Unlike the situation for forms of type $(3,6)$, however, we do not have any clear approach to this problem at present. 

\subsection{Excellent Connections for Degenerate Forms} We conclude this article by making the conjecture that Theorem \ref{THMexcellentconnections} remains valid in the case where $\varphi$ is degenerate:

\begin{conjecture} \label{CONJexcellentconnections} Let $a$ and $b$ be integers with $0 \leq a,d_X - b <r$. If the pair $(a,b)$ is excellent, then $a_{\mathrm{lo}}$ and $b^{\mathrm{up}}$ are connected in $\Lambda(X)$.  \end{conjecture}

The proof of Theorem \ref{THMexcellentconnections} given in \cite{Vishik2} is currently not transferrable to the degenerate setting, since although we have an action of cohomological-type Steenrod operations on $\ratchow{X \times X}$,  we do not have an action such operations on $\mathrm{Ch}(X \times X)$ itself. Over fields of characteristic different from $2$, however, there are also \emph{homological-type} Steenrod operations acting on the mod-2 Chow groups of any variety, smooth or not (see \cite{Brosnan} or \cite[Ch. XI]{EKM}). We expect that analogous operations should exist in characteristic $2$, and that these operations are more-or-less sufficient to prove Conjecture \ref{CONJexcellentconnections}. We give here some implications of the validity of Conjecture \ref{CONJexcellentconnections}:

\begin{proposition} \label{PROPimplicationsofexcellentconnections} Suppose that Conjecture \ref{CONJexcellentconnections} holds.  Then:
\begin{enumerate} \item Let $1 \leq t < h$, and suppose that $\mathrm{MDT}(\varphi)$ admits an element $\Lambda$ with $a(\Lambda) = \mathfrak{j}_{t-1}$. Let $j$ and $n_1,\hdots,n_j$ be the unique nonnegative integers for which $\mydim{\varphi} - 2\mathfrak{j}_{t-1} - \mathfrak{j}_t = 2^{n_1} -2^{n_2} + \cdots + (-1)^{j-1}2^{n_j}$ and $n_1>n_2> \cdots > n_{j-1} > n_j + 1$. For each integer $1 \leq k \leq j$, consider the integer
$$d_{k}: = \left(\frac{\mydim{\varphi} - 2\mathfrak{j}_{t-1} - \mathfrak{j}_t}{2}\right) + \sum_{i=k}^j (-1)^{k+i-1}2^{n_i - 1}. $$
If $d_k < r$, we then have that $(d_k)_{\mathrm{lo}} \in \Lambda$. Moreover, if $t \leq t' \leq h$ is such that $\mathfrak{j}_{t'-1} < d_{k} \leq \mathfrak{j}_{t'}$, then then $d_{k} + \mathfrak{i}_t \leq \mathfrak{j}_{t'}$. 
\item Write $\mydim{\varphi} = 2^n + m$ for integers $n \geq 0$ and $1 \leq m \leq 2^n$. If $r<m$, then $\witti{1}{\varphi} \leq m-r$. 
\item Conjecture \ref{CONJi1smalls} holds. 
\end{enumerate}
\begin{proof} (1) Again, the second statement follows from the first and Proposition \ref{PROPsummandshifts} (2). Given the validity of Conjecture \ref{CONJexcellentconnections}, however, the proof of Theorem \ref{THMVishikconnections} carries over verbatim. 

Now, in proving the remaining statements, let first us note that since $a(\Lambda^U(X)) = 0 = \mathfrak{j}_0$, we are in a position to apply (1) with $t=1$. 

(2) Since $\mathfrak{i}_1 \leq r$, we have $\mathfrak{i}_1 < m$, and hence $\Izhdim{\varphi} = 2^{n+1} -d$ for some integer $1 \leq d < 2^n$. Taking $t = 1$ in (1), we then have that the integer $d_2$ is defined and equal to $\Izhdim{\varphi} - 2^n = m - \mathfrak{i}_1$. If we had $m - \mathfrak{i}_1 < r$, part (1) would then give that $m \leq r$, contradicting our assumption. We must therefore have that $m - \mathfrak{i}_1 \geq r$, i.e., $\mathfrak{i}_1 \leq m-r$. 

(3) As in the statement of Conjecture \ref{CONJi1smalls}, let $u$ be the largest integer for which $\Izhdim{\varphi}$ is divisible by $2^u$. Suppose that $s< 2^u + \mathfrak{i}_1$.  We again consider the statement of (1) with $t=1$.  As in the proof of Theorem \ref{THMi1restrictionssingular}, the integer $d_j$ is equal to $\frac{\Izhdim{\varphi}}{2} - 2^{u-1}$.  Since $s < 2^n + \mathfrak{i}_1$, we have
$$ \frac{\Izhdim{\varphi}}{2} - 2^{u-1} = \frac{2r + s - \mathfrak{i}_1}{2} - 2^{u-1} =  r +  \frac{s - (2^u + \mathfrak{i}_1)}{2} <r, $$
and so $\frac{\Izhdim{\varphi}}{2} - 2^{u-1} + \mathfrak{i}_1 \leq r$ by (1). Exactly as in the proof of Theorem \ref{THMi1restrictionssingular}, this then leads to the conclusion that $\mathfrak{i}_1 \leq r - x2^u \leq 2^u -s$. 
\end{proof} \end{proposition}

\begin{remarks} \begin{enumerate}[leftmargin=*] \item The statement in part (1) of the preceding proposition may be interpreted as follows: Suppose $\psi$ is an anisotropic nondegenerate form over an extension of $F$ such that $\mydim{\psi} = \mydim{\varphi}$ and $\wittj{t}{\psi} = \mathfrak{j}_t$ for all $1 \leq t \leq h$.  The set $\Lambda(X)$ may be viewed as a subset of $\Lambda(X_{\psi})$. We expect that any connections among elements of this subset that are forced to exist by the Knebusch splitting pattern of $\psi$ (in particular, those given by the $t=1$ case of Theorem \ref{THMVishikconnections}) remain valid in $\Lambda(X)$ (now understood as an invariant of $\varphi$). More generally,  recall that the shell pyramid diagram for $X$ is the shell pyramid diagram for $X_{\psi}$, but with the shells indexed by integers $\geq h$ deleted (see \S \ref{SUBSECcyclesonX^2}). Suppose that we have a restriction on the diagram for $X_{\psi}$ that is entirely determined by the Knebusch splitting pattern of $\psi$. If we delete the shells indexed by integers $\geq h$, then we expect that what remains of the restriction is a valid restriction on the diagram for $X$. 
\item Note the statement in part (2) implies that $\mathfrak{i}_1 \neq 2$ when $\varphi$ is a $12$-dimensional form of type $(3,6)$ (i.e., settles the first open case of Question \ref{Qi1}). This is not implied by Conjecture \ref{CONJi1smalls}, and so the scope of Conjecture \ref{CONJexcellentconnections} is broader. At the same time, the statement in part (1) says nothing about the problem of whether $\mathfrak{i}_1$ can equal $2$ when $\varphi$ is a $14$-dimensional form of type $(3,8)$ or $(5,4)$. \end{enumerate}
\end{remarks} 

We conclude with the following lemma, which gives some meagre evidence for Conjecture \ref{CONJexcellentconnections}:

\begin{lemma} \label{LEMexcellentconnectionsindimensionatmost9} Conjecture \ref{CONJexcellentconnections} holds in the case where $\mydim{\varphi} \leq 9$. 
\begin{proof} In view of Theorem \ref{THMexcellentconnections}, we may assume that $\varphi$ is degenerate, i.e., that $s \geq 2$. We then have that $\mydim{\varphi} \geq 4$.  \vspace{.5 \baselineskip}

\noindent \underline{$\mydim{\varphi}=4$}.  In this case, $\varphi$ has type $(1,2)$, and there are no excellent pairs to consider. \vspace{.5 \baselineskip}

\noindent \underline{$\mydim{\varphi} = 5$}. In this case, $\varphi$ has type $(1,3)$, and the only excellent pair to consider is $(0,3)$. But in this case we have $\mathfrak{i}_1 = 1$ (Examples \ref{EXSi1slowdimension}), and so $0_{\mathrm{lo}}$ and $3^{\mathrm{up}}$ are connected in $\Lambda(X)$ by Proposition \ref{PROPsplittingpatternconnections}. \vspace{.5 \baselineskip}

\noindent \underline{$\mydim{\varphi} = 6$}. In this case, $\varphi$ has type $(1,4)$ or $(2,2)$. In the first case, however, there are no excellent pairs to consider, and so we can assume that $\varphi$ has type $(2,2)$.  Here, there are two excellent pairs to consider, namely $(0,3)$ and $(1,4)$.  But since $r = 2$, $2$ lies in the nondefective splitting pattern of $\varphi$. By Proposition \ref{PROPVPNs}, it follows that $\varphi$ is a virtual Pfister neighbour. But Lemma \ref{LEMMDTforVPNs} then gives the desired connections in $\Lambda(X)$ between $0_{\mathrm{lo}}$ and $3^{\mathrm{up}}$ as well as $1_{\mathrm{lo}}$ and $4^{\mathrm{up}}$.\vspace{.5 \baselineskip}

\noindent \underline{$\mydim{\varphi} = 7$}. In this case, $\varphi$ has type $(1,5)$ or $(2,3)$.  In the first case, there are no excellent pairs to consider. Suppose therefore that $\varphi$ has type $(2,3)$. Here, there is one excellent pair to consider, namely $(1, 4)$. But since $\varphi$ has type $(2,3)$, we have $\mathfrak{i}_1 = 1$ (Examples \ref{EXSi1slowdimension}). Since $r = 2$, we must then also have that $h = 2$ and $\mathfrak{i}_2 = 1$. By Proposition \ref{PROPsplittingpatternconnections}, it then follows that $1_{\mathrm{lo}}$ and $4^{\mathrm{up}}$ are connected in $\Lambda(X)$. \vspace{.5 \baselineskip}

\noindent \underline{$\mydim{\varphi} = 8$}.  In this case, $\varphi$ has type $(1,6)$, $(2,4)$ or $(3,2)$. In the first two cases, there are no excellent pairs to consider. Suppose therefore that $\varphi$ has type $(3,2)$. Then there are two excellent pairs to consider, namely $(1,4)$ and $(2,5)$. Since $\varphi$ has type $(3,2)$, we have $\mathfrak{i}_1 = 1$ (Examples \ref{EXSi1slowdimension}). Then $\varphi_1$ is a $6$-dimensional form of type $(1,3)$. By the preceding discussion, $0_{\mathrm{lo}}$ and $3^{\mathrm{up}}$, as well as $1_{\mathrm{lo}}$ and $4^{\mathrm{up}}$ are then connected in $\Lambda(X_{\varphi_1})$. By Lemma \ref{LEMMDTforisotropicforms}, $1_{\mathrm{lo}}$ and $4^{\mathrm{up}}$, as well as $2_{\mathrm{lo}}$ and $5^{\mathrm{up}}$ are then connected in $\Lambda(X)$. \vspace{.5 \baselineskip}

\noindent \underline{$\mydim{\varphi} = 9$}.  In this case, $\varphi$ has type $(1,7)$, $(2,5)$ or $(3,3)$. In all cases, we have the excellent pair $(0,7)$.  But $\mathfrak{i}_1 = 1$ in this case (Examples \ref{EXSi1slowdimension}), so $0_{\mathrm{lo}}$ and $7^{\mathrm{up}}$ are connected in $\Lambda(X)$ by Proposition \ref{PROPsplittingpatternconnections}. If $\varphi$ has type $(1,7)$ or $(2,5)$, then there are no other excellent pairs to consider. Suppose now that $\varphi$ has type $(3,3)$. We then have one other excellent pair to consider, namely $(2,5)$.  Again, however, since $\mathfrak{i}_1 = 1$, $\varphi_1$ is a $7$-dimensional form of type $(2,3)$. By the preceding discussion, $1_{\mathrm{lo}}$ and $4^{\mathrm{up}}$ are then connected in $\Lambda(X_{\varphi_1})$. By Lemma \ref{LEMMDTforisotropicforms}, $2_{\mathrm{lo}}$ and $5^{\mathrm{up}}$ are then connected in $\Lambda(X)$. \vspace{.5 \baselineskip}
\end{proof} \end{lemma}

When $\mydim{\varphi} = 10$, there is just one case in which we do not know the validity of Conjecture \ref{CONJexcellentconnections}, namely the case where $\varphi$ has type $(3,4)$ and nondefective splitting pattern $(1,3)$. Here, $2$ is not in the nondefective splitting pattern of $\varphi$, and so $\varphi$ is not a virtual Pfister neighbour. In particular, the argument used to handle the case of $6$-dimensional forms in the proof of the preceding lemma is not applicable here. Note that it is relatively straightforward to classify the anisotropic forms of type $(3,4)$ and nondefective splitting pattern $(1,3)$: They are precisely the $10$-dimensional forms similar to $\anispart{(\qpfister{a,b}{c} \perp \langle 1,d,e,de \rangle)}$ for some $a,b,c,d,e \in F^\times$.  Thus, if we exclude this specific class of forms, then the statement of Conjecture \ref{CONJexcellentconnections} is also valid in dimension $10$.  \vspace{1 \baselineskip}

\noindent {\bf Acknowledgements.} The work of both authors was supported by the NSERC Discovery Grant No. RGPIN-2019-05607.

\bibliographystyle{alphaurl}

\end{document}